\documentclass[11pt]{amsart}

\usepackage{amsmath,amssymb,amsthm,mathrsfs,url,color,enumerate, mathabx,wrapfig}
\usepackage{graphicx}
\usepackage[super]{nth}
\usepackage[open, openlevel=2, depth=3, atend]{bookmark}
\hypersetup{pdfstartview=XYZ}
\makeatletter
\renewcommand\paragraph{\@startsection{paragraph}{4}{0pt}{\baselineskip}{0.5pt}{\slshape} }
\renewcommand{\@seccntformat}[1]{{\csname the#1\endcsname.} }
\makeatother

\newcommand{\C}{\mathbb{C}}
\newcommand{\R}{\mathbb{R}}

\newcommand{\Ss}{\mathbb{S}}
\newcommand{\N}{\mathbb{N}}
\newcommand{\Z}{\mathbb{Z}}
\newcommand{\Hh}{\mathbb{H}}

\newcommand{\J}{\mathbb{J}}
\newcommand{\U}{\mathbb{U}}

\DeclareMathOperator{\Jac}{Jac}
\DeclareMathOperator{\Tr}{Tr}
\DeclareMathOperator{\argsh}{argsh}
\DeclareMathOperator{\vol}{vol}
\DeclareMathOperator{\Res}{Res}

\newtheorem{theorem}{Theorem}

\newtheorem*{theorem*}{Theorem}
\newtheorem*{theoreme*}{Théorème}
\newtheorem{proposition}{Proposition}[section]
\newtheorem{definition}[proposition]{Definition}
\newtheorem{definition-f}[proposition]{Définition}

\newtheorem{lemma}[proposition]{Lemma}

\newtheorem{remark}{Remark}

\newtheorem*{remark*}{Remark}
\newtheorem{conjecture}{Conjecture}

\title[Resonance-free regions for some finite volume manifolds]{Resonance-free regions for negatively curved manifolds with cusps}
\author{Yannick Bonthonneau}
\email{yannick.bonthonneau@ens.fr}
\address{DMA, U.M.R. 8553 CNRS, \'Ecole Normale Superieure, 45 rue d'Ulm,
75230 Paris cedex 05, France}
\begin{document}

\begin{abstract}
The Laplace-Beltrami operator on cusp manifolds has continuous spectrum. The resonances are complex numbers that replace the discrete spectrum of the compact case. They are the poles of a meromorphic function $\varphi(s)$, $s\in \C$, the \emph{scattering determinant}. We construct a semi-classical parametrix for this function in a half plane of $\C$ when the curvature of the manifold is negative. We deduce that for manifolds with one cusp, there is a zone without resonances at high frequency. This is true more generally for manifolds with several cusps and generic metrics.

We also study some exceptional examples with almost explicit sequences of resonances away from the spectrum.
\end{abstract}

\maketitle

The object of our study are complete $d+1$-dimensional negatively curved manifolds of finite volume $(M,g)$ with a finite number $\kappa$ of real hyperbolic cusp ends. The Laplace operator is denoted $\Delta$ in the analyst's convention that $-\Delta \geq 0$. The resolvent $R(s) = (-\Delta - s(d-s))^{-1}$ is a priori defined on $L^2(M)$ for $\Re s > d/2$. Thanks to the analytic structure at infinity, one shows that $R$ can be analytically continued to $\C$ as a meromorphic family of operators $C^\infty_c \to C^\infty$ whose set of poles is called the \emph{resonant set} $\Res(M)$. The original proof is due to Selberg in constant $-1$ curvature, to Lax and Phillips \cite{Lax-Phillips-Automorphic-76} for surfaces, and this subject was studied by both Yves Colin de Verdi\`ere \cite{CdV-81, CdV-83} and Werner M\"uller \cite{Muller-83, Muller-86, Muller-92}. It fits in the general theory of spectral analysis on geometrically finite manifolds with constant curvature ends, see \cite{Mazzeo-Melrose-87, Guillope-Zworski}. 

The spectrum of $-\Delta$ divides into both discrete $L^2$ spectrum, that may be finite, infinite or reduced to $\{0\}$, and continuous spectrum $[d^2/4, + \infty)$. We can find a precise description of the structure of its spectral decomposition given by the Spectral Theorem in \cite{Muller-83}. For each cusp $Z_i$, $i=1 \dots \kappa$, there is a meromorphic family of \emph{Eisenstein functions} $\{E_i(s)\}_{s\in \C}$ on $M$ such that
\begin{equation}\label{eq:functional_equation_Eisenstein_eigen}
-\Delta E_i(s) = s(d-s)E_i(s).
\end{equation}
The poles of the family are contained in $\{\Re s < d/2\}\cup (d/2, d]$, and are called \emph{resonances}. We also consider the vector $E=(E_1, \dots, E_\kappa)$. Let $\{u_\ell\}_\ell$ be the discrete $L^2$ eigenvalues. Then, any $f\in C^\infty_c(M)$ expands as:
\begin{equation*}
f= \sum_\ell \langle u_\ell, f\rangle u_\ell + \frac{1}{4\pi} \sum_i \int_{-\infty}^{+\infty} E_i\left(\frac{d}{2} + i t \right)\left\langle E_i\left(\frac{d}{2} + i t\right), f \right\rangle dt  \ \text{  \cite[eq. 7.36]{Muller-83},}
\end{equation*}
where $\langle \cdot, \cdot \rangle$ is the $L^2$ duality product. An important feature of the Eisenstein functions is the following: in cusp $Z_j$, the zeroth Fourier mode in $\theta$ of $E_i$ writes as
\begin{equation}\label{eq:functional_equation_Eisenstein_Fourier}
\delta_{ij}y^s + \phi_{ij}(s)y^{d-s}.
\end{equation}
where $\phi_{ij}$ is a meromorphic function. If we take the determinant of the \emph{scattering matrix} $\phi=\{\phi_{ij}\}$, we obtain the \emph{scattering determinant} $\varphi(s)$. It is known that the set $\mathcal{R}$\index{$\mathcal{R}$} of poles of $\varphi$ is the same as that of $\{E(s)\}_{s}$ --- again, see \cite[theorem 7.24]{Muller-83}. It also coincides with the poles of the analytic continuation of the kernel of the resolvent of the Laplacian, \cite{Muller-83}.

Thanks to the symmetries of the problem, $\varphi(s) \varphi(d-s)=1$, so that studying the poles of $\varphi$ in $\{ \Re s < d/2\}$ is equivalent to studying the zeroes in $\{ \Re s > d/2 \}$. In this article, we will be giving information on the zeroes of $\varphi$, keeping in mind that the really important objects are the poles.

The first examples of cusp manifolds to be studied had constant curvature, and were arithmetic quotients of the hyperbolic plane. Let $\Gamma_0(N)$ be the congruence subgroup of order $N$, that is, the kernel of the morphism $\pi: SL_2(\Z) \to SL_2(\Z_N)$. Then, $\Hh / \Gamma_0(N)$ is a \emph{cusp surface}. For such examples, and more generally, for constant curvature cusp surfaces $\Hh / \Gamma$, if cusp $Z_i$ is associated with the point $\infty$ in the half plane model, then the associated Eisenstein functions can be written as a series
\begin{equation}\label{eq:Expansion_Eisenstein_constant_curvature}
E_i(s)(z) = \sum_{[\gamma] \in\Gamma_{i}\backslash \Gamma} \left[\Im(\gamma z)\right]^s
\end{equation}
where $\Gamma_i$ is the maximal parabolic subgroup of $\Gamma$ associated with $Z_i$. Recall a Dirichlet series is a function of the form
\begin{equation*}
f(s) = \sum_{k\geq 0} \frac{a_k}{\lambda_k^s}, \text{ where $(\lambda_k)$ is a strictly increasing sequence of real numbers.}
\end{equation*}
Selberg proved --- see \cite{Selberg-2} --- that there is a non-zero Dirichlet series $L$ converging absolutely for $\{\Re s >d\}$ so that 
\begin{equation}\label{eq:Selberg_Dirichlet_series}
\varphi(s) = \left(\frac{\pi \Gamma(s-1/2)}{\Gamma(s)}\right)^{\kappa/2} L(s)
\end{equation}
This implies:
\begin{theorem*}[Selberg]
For constant curvature cusp surfaces, the resonances are contained in a vertical strip of the form $\{ 1/2 - \delta \leq \Re s \leq 1/2\}$, where $\delta >0$ (with maybe the exception of a finite number of resonances in $(1/2,1]$).
\end{theorem*}

M\"uller was actually the one who asked if Selberg's theorem still holds in variable curvature. Froese and Zworski \cite{Froese-Zworski} gave a counter-example, that had positive curvature. The following theorem gives a partial answer in negative curvature.
\begin{theorem}\label{theorem:main}
For $M$ a cusp manifold, let $\mathcal{G}(M)$ be the set of $C^\infty$ metrics $g$ on $M$ such that $(M,g)$ is a cusp manifold with negative sectional curvature. If $U \subset \subset M$ is open, let $\mathcal{G}_U(M)$ be the set of metrics in $\mathcal{G}(M)$ that have constant curvature outside of $U$. Endow $\mathcal{G}(M)$ and $\mathcal{G}_U(M)$ with the $C^2$ topology on metrics. Then
\begin{enumerate}[	\begin{upshape}(I)\end{upshape}  ]
	\item There are hyperbolic cusp surfaces $M$ and non-empty open sets $U \subset \subset M$ such that for all $g\in \mathcal{G}_U(M)$, $\Res(M,g)$ is still contained in a vertical strip.
	\item Given any cusp manifold $M$, for an open and dense set of $g \in\mathcal{G}(M)$, or all of $\mathcal{G}(M)$ when there is only one cusp, there is a $\delta > d/2$ such that for any constant $C>0$, 
\begin{equation*}
\left\{ s \in \Res(M,g),\ \Re s < d - \delta ,\ \Re s > -C \log | \Im s| \right\} \text{ is finite.}	
\end{equation*}
	\item \label{metrics-two-cusp-exceptional} There is a 2-cusped surface $(M,g)$ with the following properties. The resonant set $\Res(M,g)$ is the union of $\Res_{strip}$, $\Res_{far}$ and an exceptional set $\Res_{exc}$, so that
\begin{align*}
\Res_{strip} &= \left\{s \in \Res(M,g),\ \Re s > d - \delta \right\} \\
\Res_{exc} &= \left\{ s_i, \overline{s_i},\ i\in \N,\ s_i = \tilde{s}_i + \mathcal{O}(1) \right\} \\
\Res_{far} & \cap \left\{ \Re s > - C \log | \Im s| \right\} \text{ is finite for any $C>0$}.
\end{align*}
where $\delta > 0$, and the $\tilde{s}_i$'s and $\overline{\tilde{s}_i}$'s are the zeroes of $s e^{- sT} - C_0$ for some constants $T > 0$ and $C_0 \neq 0$.
	\item \label{metrics-with-non-vanishing-parametrix} For a bigger open and dense set of metrics $g\in \mathcal{G}(M)$, containing the example in \eqref{metrics-two-cusp-exceptional}, there are constants $\delta> 0$, and $C_0 >0$ such that
\begin{equation*}
\left\{ s \in \Res(M,g),\ \Re s < d - \delta,\ \Re s > - C_0 \log | \Im s| \right\} \text{ is finite.}	
\end{equation*}	
\end{enumerate}
\end{theorem}

\begin{figure}
\centering
\def\svgwidth{\linewidth}
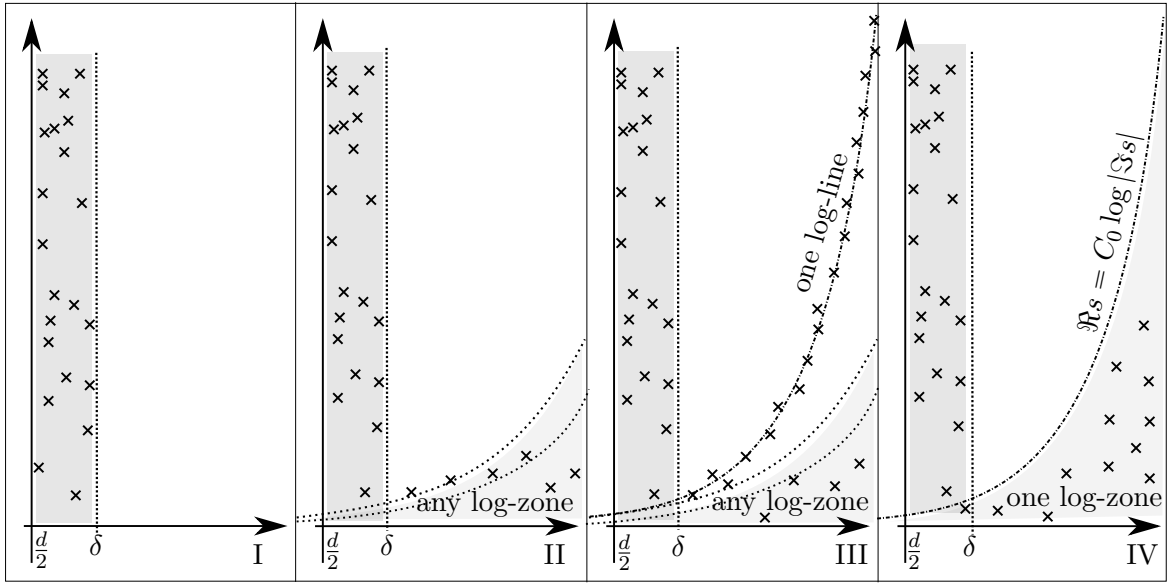
\caption{\label{fig:different cases}The zeroes of $\varphi$: 4 cases in theorem \ref{theorem:main}.}
\end{figure}

\begin{conjecture}\label{conj:one-log-zone}
The set of metrics in \eqref{metrics-with-non-vanishing-parametrix} is actually $\mathcal{G}(M)$.
\end{conjecture}

\begin{conjecture}
For an open and dense set of $g\in \mathcal{G}(M)$, there is an infinite number of resonances outside of any strip $d/2 > \Re s > d - \delta$.
\end{conjecture}
Our reason for conjecturing this is that the existence of such resonances seems to be more stable than their absence.

The main tool to prove theorem \ref{theorem:main} is a parametrix for the scattering determinant $\varphi$ in a half plane $\{ \Re s > \delta_g\}$. Thanks to the form of that parametrix --- sums of Dirichlet series --- we will be able to determine zones where $\varphi$ does not vanish. 
\begin{theorem}\label{theorem:Dirichlet-series-expansion-vertical-strip}
Let $(M,g)$ be a negatively curved cusp manifold. There is a constant $\delta_g > d/2$ and Dirichlet series $L_0,\ \dots,\ L_n,\ \dots $ with abscissa of absolute convergence $\delta_g$ such that if at least one of the $L_n$'s does not identically vanish, for $\Re s > \delta_g$, as $\Im s \to \pm \infty$,
\begin{equation*}
\varphi(s) \sim s^{-\kappa d/2} L_0 + s^{-\kappa d/2 - 1} L_1 + \dots .
\end{equation*}
(Recall, $\kappa$ is the number of cusps). Actually, the constant $\delta_g$ is the pressure of the potential $(F^{su}+d)/2$, where $F^{su}$ is the unstable jacobian. In constant curvature, $F^{su}=-d$ and $\delta_g = d$.
\end{theorem}
This is a consequence of a more precise estimate --- see Theorem \ref{theorem:parametrix_varphi}. The $L_n$'s are defined by dynamical quantities related to \emph{scattered geodesics}. Those are geodesics that come from one cusp and escape also in a cusp --- maybe the same --- spending only a finite time in the compact part of $M$, called the \emph{Sojourn Time}. This terminology was introduced by Victor Guillemin \cite{Guillemin}. In that article, for the case of constant curvature, he gave a version similar to ours of \eqref{eq:Selberg_Dirichlet_series}. He also conjectured that something along the lines of our theorem should hold --- see the concluding remarks pp. 79 in \cite{Guillemin}. Lizhen Ji and Maciej Zworski gave a related result in the case of locally symmetric spaces \cite{Ji-Zworski}. 

Sojourn Times are objects in the general theory of classical scattering --- see \cite{Petkov-Stoyanov}. Maybe ideas from different scattering situations may help to prove Conjecture \ref{conj:one-log-zone}, that may be reformulated as 
\newtheorem*{conj1'}{Conjecture 1'}
\begin{conj1'}\label{conj:non-vanishing-of-coefficients}
Given $g\in \mathcal{G}(M)$, at least one $L_i$ is not identically zero.
\end{conj1'}

The structure of the article is the following. In section \ref{sec:potential-and-scattered} we recall some definitions and results on cusp manifolds, and prove the convergence of a modified Poincar\'e series. Section \ref{sec:param-E} is devoted to building a parametrix for the Eisenstein functions, via a WKB argument, using the modified Poincar\'e series. In section \ref{sec:param-phi}, we turn to a parametrix for the scattering determinant. To use Stationary Phase, most of the effort goes into proving the non-degeneracy of a phase function. The purpose of section \ref{section:continuity-parametrix-metric} is to study the behaviour of the series $L_i$ when we vary the metric. Finally, we prove theorem \ref{theorem:main} in section \ref{section:Application}. In appendix \ref{appendix:horocycles}, for lack of a reference, we give a proof of a regularity result on horocycles. This result may be of interest for the study of negatively curved geometrically finite manifolds in general.

This work is part of the author's PhD thesis. In a forecoming article \cite{Bonthonneau-4}, we will deduce precise spectral counting results from Theorem \ref{theorem:Dirichlet-series-expansion-vertical-strip}.

\textbf{Acknowledgment} We thank Colin Guillarmou and St\'ephane Nonnenmacher for suggesting the idea that led to this article. We also thank Colin Guillarmou, Nalini Anantharaman and Maciev Zworski for their very helpful suggestions.

\setcounter{section}{-1}
\section{Preliminaries}

We start by recalling well-known facts in scattering theory on manifolds with cusps, and we refer to the articles of M\"uller \cite{Muller-83, Muller-86, Muller-92} for details and proofs. Let  $(M,g)$ be a complete Riemannian manifold that can be decomposed as follows: 
\begin{equation*}
M = M_0 \sqcup Z_1 \sqcup \dots \sqcup Z_\kappa,
\end{equation*}
where $M_0$ is a compact manifold with smooth boundary and negative curvature, and $Z_i$ are hyperbolic cusps
\begin{equation}\label{eq:def-Z}\index{$a_0$, $a_i$}
(a_i, +\infty ) \times \mathbb{T}_i^d \simeq Z_i \owns x =(y,\theta),\quad \theta = (\theta_1, \dots, \theta_d), \quad i= 1 \dots \kappa,
\end{equation}
where $a_i > 0$, and $\mathbb{T}_i^d = \mathbb{T}^d_{\Lambda_i} = \R^d / \Lambda_i$ are $d$-dimensional flat tori, and the metric on $Z_i$ in coordinates $( y,\theta) \in  (a_i , + \infty) \times \mathbb{T}_i^d$ is
\begin{equation*}
{d}s^2 = \frac{\mathrm{d}y^2 + \mathrm{d}\theta^2}{y^2}.
\end{equation*}
Notice that the manifold has finite volume when equipped with this metric. The choice of the coordinate $y$ on a cusp is unique up to a scaling factor, and we choose it so that all $\mathbb{T}_i^d$'s have volume $1$.  Such a manifold will be referred to as a \emph{cusp-manifold}. Mind that we require that they have \emph{negative curvature}.

The non-negative Laplacian $-\Delta$ acting on $C_c^\infty(M)$ functions has a unique self-adjoint extension to $L^2(M)$ and its spectrum consists of
\begin{enumerate}[ {[}1{]}]
\item Absolutely continuous spectrum $\sigma_{ac}=[d^2/4, +\infty)$ with multiplicity $\kappa$ (the number of cusps).
\item Discrete spectrum $\sigma_d =\{\lambda_0 = 0 < \lambda_1 \leq \dots \leq \lambda_\ell \leq \dots \}$, possibly finite, and which may contain eigenvalues embedded in the continuous spectrum.
\end{enumerate}

Conditions \eqref{eq:functional_equation_Eisenstein_eigen} and \eqref{eq:functional_equation_Eisenstein_Fourier} imply the uniqueness of the Eisenstein functions. One can deduce that
\begin{equation*}
E(d-s) = \phi(s)^{-1} E(s)
\end{equation*}
Hence,
\begin{equation*}
\phi(s)\phi(d-s)=Id, \quad \overline{\phi(s)} = \phi(\overline{s}), \quad \phi(s)^T = \phi(s), \text{ and } \varphi(s)\varphi(d-s) = 1.
\end{equation*}
The line $\Re s =d/2$ corresponds to the continuous spectrum. On that line, $\phi(s)$ is unitary, $\varphi(s)$ has modulus $1$.

The set $\mathcal{R}$ of poles of $\varphi$, $\phi$ and $(E_j)_{j=1\dots k}$ are all the same, we call them \emph{resonances}. It is contained in $\{\Re s < d/2\} \cup (d/2, 1]$. The union of this set with the set of $s\in \C$ such that $s(d-s)$ is an $L^2$ eigenvalue, is called the \emph{resonant set}, and denoted $\Res(M,g)$. Following \cite[pp.287]{Muller-92}, the multiplicities $m(s)$ are defined as :
\begin{enumerate}[ {[}1{]}]
	\item If $\Re s \geq d/2$, $s \neq d/2$, $m(s)$ is the dimension of $\ker_{L^2}(\Delta_g-s(d-s))$.
	\item If $\Re s < d/2$, $m(s)$ is the dimension of $\ker_{L^2}(\Delta_g-s(d-s))$ minus the order of $\varphi$ at $s$.
	\item $m(d/2)$ equals $({\rm Tr}(\phi(d/2)) + m)/2$ plus twice the dimension of $\ker_{L^2}(\Delta_g-d^2/4)$.
\end{enumerate}
We let 
\begin{equation*}
\mathcal{Z} = \{ d- s \ | \ s \text{ is a resonance }\}, \ \text{this is the set of zeroes of $\varphi$.}
\end{equation*}

\section[Scattered geodesics, Potential theory]{Scattered geodesics and some potential theory on cusp manifolds}\label{sec:potential-and-scattered}

Recall that a manifold $N$ is said to have \emph{bounded geometry} when its injectivity radius is strictly positive, and when $\nabla^k R$ is bounded for all $k=0, 1, \dots$, $R$ being the Riemann curvature tensor of $N$. Since the injectivity radius goes to zero in a cusp, a cusp manifold cannot have bounded geometry. However, its universal cover $\widetilde{M}$ does. Since the curvature of $M$ is negative, $\widetilde{M}$ is also a Hadamard space --- diffeomorphic to $\R^{d+1}$ --- and we can define its visual boundary $\partial_\infty \widetilde{M}$, and visual compactification $\overline{M} = \widetilde{M}\cup \partial_\infty \widetilde{M}$.

In all the article, unless stated otherwise, we will refer to the projection $T^\ast \widetilde{M} \to \widetilde{M}$ as $\pi$; when we say \emph{geodesic}, we always mean \emph{unit speed} geodesic.

The results given without proof are from the book \cite{PPS-12}.

\subsection{Hadamard spaces with bounded geometry and negative curvature}

Let us define the \emph{Busemann cocycle} in the following way. For $p\in \partial_\infty \widetilde{M}$, Let
\begin{equation*}
\beta_p(x,x'):=\lim\limits_{w \to p} d(x,w) - d(x',w).
\end{equation*}
For each $p\in \partial_\infty \widetilde{M}$, we pick $m_p \in \widetilde{M}$ --- we will specify this choice later, see remark \ref{remark:choice_normalization_horocycle}. Then, we define the horosphere $H(p, r)$ (resp. the horoball $B(p, r)$) of radius $r\in \R$ based at $p$ as
\begin{equation}
H(p,r):=\left\{x \in \widetilde{M} \ \middle|\ \beta_p(x,m_p) = - \log r \right\} \text{ and } B(p,r):=\left\{x \in \widetilde{M} \ \middle| \ \beta_p(x,m_p) \leq -\log r \right\}.
\end{equation} 
We also define
\begin{equation}
G_p(x) : = \beta_p(x,m_p).
\end{equation}
Beware that with these notations, horoballs $B(p,r)$ increase in size as $r$ decreases. The number $r$ will correspond to a height $y$ in the coming developments.

Since the curvature of $\widetilde{M}$ is pinched-negative $-k_{max}^2 \leq K \leq -k_{min}^2$, $\widetilde{M}$ has the \emph{Anosov property}. That is, at every point of $S^\ast \widetilde{M}$, there are subbundles such that
\begin{equation*}\index{$E^s$, $E^u$}
T (S^\ast M) = \R \mathbf{X} \oplus E^s \oplus E^u
\end{equation*}
where $\mathbf{X}$\index{$\mathbf{X}$} is the vector field of the \emph{geodesic flow} $\varphi_t$. This decomposition is invariant under $\varphi_t$, and there are constants $C>0, \lambda > 0$ such that for $t>0$
\begin{equation*}
\| d \varphi_t |_{E^s} \| \leq C e^{-\lambda t} \text{ and } \| d \varphi_{-t} |_{E^u} \| \leq C e^{- \lambda t}.
\end{equation*}
The subbundle $E^s$ (resp. $E^u$) is tangent to the \emph{strong stable} (resp. \emph{unstable}) foliation $W^s$ (resp. $W^u$). The subbundles $E^s$, $E^u$ are only H\"older --- see \cite[theorem 7.3]{PPS-12} --- but each leaf of $W^s$, $W^u$ is a $\mathscr{C}^\infty$ submanifold of $\widetilde{M}$ --- see lemma \ref{lemma:Regularity_horocycles}.

\begin{remark}\label{remark:measuring_regularity}
We have to say how we measure regularity on $\widetilde{M}$ and $T\widetilde{M}$. In $TT\widetilde{M}$, we have the \emph{vertical} subbundle $V= \ker T\pi : TT\widetilde{M} \to T\widetilde{M}$. Since $\widetilde{M}$ is riemannian, we also have a horizontal subbundle $H$ given by the connection $\nabla$. Both $V$ and $H$ can be identified with $T\widetilde{M}$, and the Sasaki metric is the one metric on $T\widetilde{M}$ so that $V \perp H$ and those identifications are isometries. 

We endow $T \widetilde{M}$ with the Sasaki metric, and then also $T^\ast \widetilde{M}$ by requesting that $ v \mapsto \langle v , \cdot \rangle$ is an isometry. For a detailed account on the Sasaki metric, see \cite{Gudmundsson-Kappos}. On all the manifolds that appear, when they have a metric, we define their $\mathscr{C}^k$ spaces, $k\in \N $, using the norm of their covariant derivatives:
\begin{equation*}
\| f \|_{\mathscr{C}^n} := \sup_{k=0,\dots,n} \| \nabla^k f\|_\infty.
\end{equation*}
Then, $\mathscr{C}^\infty = \cap_{n \geq 0} \mathscr{C}^n$. For a more detailed account of $\mathscr{C}^k$ spaces on a riemannian manifold, see for example the appendix ``functionnal spaces in a cusp'' in \cite{Bonthonneau-2}.
\end{remark}

There are useful coordinates for describing the geodesic flow $\varphi_t$ on $S^\ast \widetilde{M}$. We associate its endpoints $p^-$, $p^+$ with a geodesic. Then we have the identification
\begin{equation*}
S^\ast \widetilde{M} \simeq \partial^2_\infty \widetilde{M}\times \R 
\end{equation*}
given by $\xi \mapsto (p^-, p^+, t = \beta_{p^-}(\pi \xi, m_{p^-}))$. Here $\partial^2_\infty \widetilde{M}$ is obtained by removing the diagonal from $\partial_\infty \widetilde{M} \times \partial_\infty \widetilde{M}$. In those coordinates, $\varphi_t$ is just the translation by $t$ in the last variable. Moreover, the \emph{strong unstable manifold} of $\xi$ is the set $\{p^- = p^-(\xi), t=t(\xi)\}$. For the strong stable manifold, it is a bit more complicated in this choice of coordinates.

We deduce that $W^u(\xi)$ is the set outer normal bundle to the horosphere based at $p^-(\xi)$, through $\pi \xi$. The horospheres $H(p,r)$ are $\mathscr{C}^\infty$ submanifolds, and each $G_p$ is a smooth function so that $dG_p \in\mathscr{C}^\infty(\widetilde{M})$. The proof uses the fact that the unstable manifolds $W^u$ are $\mathscr{C}^\infty$ (lemma \ref{lemma:Regularity_horocycles}), and the fact that there can be no conjugate points in negative curvature.

For $p\in \partial_\infty \widetilde{M}$, we introduce $W^{u0}(p)$ as the set of $\xi \in S^\ast \widetilde{M}$ such that $p^-(\xi)=p$. It is the set of outer normals to horospheres based at $p$. It is the graph of $dG_p$, and
\begin{equation*}
G_p(\pi\varphi_t(x,\mathrm{d} G_p)) = G_p + t.
\end{equation*}
We will refer to $W^{u0}(p)$ as the \emph{incoming Lagrangian} from $p$.

\subsection{Parabolic points and scattered geodesics}\label{section:parabolic-points-scattered-geodesics}
Now, let $\Gamma= \pi_1(M)$. It is a discrete group acting freely on $\widetilde{M}$ by isometries. The elements of $\Gamma$ can be seen to act by homeomorphisms on $\overline{M}$. We can define the limit set $\Lambda(\Gamma)$ as $\overline{\Gamma \cdot x^0} \cap \partial_\infty \widetilde{M}$, where the closure was taken in $\overline{M}$, and $x^0$ is an arbitrary point in $\widetilde{M}$. This does not depend on $x^0$. 

If $\gamma \in \Gamma$ is not the identity, one can prove that it has either. (1) Exactly one fixed point in $\widetilde{M}$, (2) Exactly two fixed points on $\partial_\infty \widetilde{M}$, (3) Exactly one fixed point in $\partial_\infty \widetilde{M}$. Then we say that it is (1) elliptic, (2) loxodromic, or (3) parabolic. Here there are no elliptic elements in $\Gamma$, since $\Gamma$ acts freely on $\widetilde{M}$. Our study will be focused of the parabolic elements of $\Gamma$.

All the parabolic elements $\gamma$ of $\Gamma$ are \emph{regular}, in the following sense: there is $r_\gamma \in \R_+^\ast$ so that if $p_\gamma$ is the fixed point of $\gamma$, $B(p_\gamma, r_\gamma)$ has constant curvature $-1$. We denote by $\Gamma_{par}$ the set of parabolic elements in $\Gamma$. The set of $p_\gamma$'s is the set of \emph{parabolic points} of $\partial_\infty \widetilde{M}$, $\Lambda_{par}$.

Let $p\in \Lambda_{par}$. Then, horoballs centered at $p$ will project down to $M$ as neighbourhoods of some cusp $Z_i$, and we say that $p$ is a parabolic point that \emph{represents} $Z_i$. When $\gamma.p=p$, we also say that $\gamma$ represents $Z_i$. Objects (points in the boundary, or elements of $\Gamma_{par}$) representing the same cusp will be called equivalent. $\Gamma$ acts on $\Gamma_{par}$ by conjugation, and elements of the same orbit under $\Gamma$ are equivalent --- however observe that the equivalence classes gather many different orbits under $\Gamma$.

If $p$ is a parabolic point representing $Z_i$, write $p \in \Lambda_{par}^i$. Let $\Gamma_p < \Gamma$ be its stabilizer. It is a \emph{maximal parabolic subgroup}. We always have $\Gamma_p \simeq \pi_1(\mathbb{T}^d_i) \simeq \Z^d$. The set of parabolic points equivalent to $p$ is in bijection with $\Gamma_p \backslash \Gamma = \{ \Gamma_p \gamma,\ \gamma \in \Gamma\}$.

The following lemma seems to be well known in the literature. However, since we cannot give a reference for a proof, we have written one down.
\begin{lemma}
Since $M$ has finite volume, $\Lambda(\Gamma)$ is the whole boundary, and the parabolic points are dense in $\partial_\infty \widetilde{M}$.
\end{lemma}

\begin{proof}
Let us pick a cusp $Z_i$, and a point $p\in \Lambda_{par}^i$. Then, we consider $x\in \partial Z_i$ in the boundary of $Z_i$ in $M$. We can lift $x$ to $\tilde{x} \in H(p,a_i)$. The orbit under $\Gamma$ of any $\tilde{x}'\in H(p, a_i)$ will remain at bounded distance of the orbit of $\tilde{x}$ under $\Gamma$. We deduce that $\Lambda(\Gamma)$ is the intersection of the closure of $\cup_\gamma \gamma H(p,a_i)$ with the boundary $\partial_\infty\widetilde{M}$. This implies in particular that $\Lambda_{par}^i \subset \Lambda(\Gamma)$.

Now, we can find a distance $d$ on $\overline{M}$ that is compatible with its topology. Indeed, take a point $m\in \widetilde{M}$, and consider the distance $\tilde{d}$ obtained on $\overline{M}$ by requesting that 
\begin{equation*}
v \in B(0,1)\subset T \widetilde{M} \mapsto \exp_{z}\left\{ v \times \mathrm{argth} |v|\right\} \text{ is an isometry.}
\end{equation*}

Then, for that distance, the sequence of images $\gamma H(p,a_i)$ have shrinking radii. Now, take a sequence of points $\tilde{x}_j \in \gamma_j H(p,a_i)$, so that $\tilde{x}_j \to q \in \Lambda(\Gamma)$. We have $\gamma_j H(p, a_i) = H(\gamma_j p, a_i)$, and so $\tilde{d}(\tilde{x}_j, \gamma_j p) \to 0$. This proves that $\Lambda(\Gamma) = \overline{\Lambda_{par}^i}$.

Next, consider the open set $U$ in $\widetilde{M}$ obtained by taking only points of $\widetilde{M}$ that project to points in the compact part $\mathring{M}_0 \subset \subset M$. There is $C>0$ such that given $\tilde{x}_1\in U$, for any $\tilde{x}_2\in U$, there is a $\gamma\in \Gamma$ such that $d(\gamma \tilde{x}_1, \tilde{x}_2) \leq C$. 

Let $\overline{U}$ be the closure of $U$ in $\overline{M}$. Since $U$ is at distance at most $C$ of the orbit of any of its points under $\Gamma$, we deduce that the limit set is $\overline{U} \cap \partial_\infty \widetilde{M}$.

Then, we find that $\Lambda(\Gamma) = \overline{\cup_\gamma \gamma H(p,a_i)}\cap \partial_\infty \widetilde{M} = \overline{U}\cap \partial_\infty \widetilde{M}$. But, we also have $\overline{\cup_\gamma \gamma H(p,a_i)}\cap \partial_\infty \widetilde{M}=\overline{\cup_\gamma \gamma B(p,a_i)}\cap \partial_\infty \widetilde{M}$. We deduce that
\begin{equation}
\Lambda(\Gamma) = \left\{\overline{\cup_\gamma \gamma B(p,a_i)}\cup \overline{U}\right\}\cap \partial_\infty \widetilde{M} = \overline{U\cup_\gamma \gamma B(p,a_i)} \cap \partial_\infty \widetilde{M} = \partial_\infty \widetilde{M}.
\end{equation}

\end{proof}

\begin{remark}\label{remark:choice_normalization_horocycle}
We will not use the functions $G_p$ when $p$ is not a parabolic point. When $p\in \Lambda_{par}^i$, one can choose the point $m_p$ so that $G_p$ coincides with $-\log \tilde{y}_p$ on the horoball $H(p, r_p)$, where $\tilde{y}_p$ is obtained on $H(p, r_p)$ by lifting the height function $y$ on the cusp $Z_i$. With this choice,  for $p\in \Lambda_{par}$ and $\gamma \in \Gamma$, we have the equivariance relation
\begin{equation}\label{eq:equivariance-G_p}
G_{\gamma^{-1} p} = G_p \circ \gamma.
\end{equation}
\end{remark}

Geodesics that enter a cusp eventually come back to $M_0$ when they are not \emph{vertical}, that is, when they are not directed along $\pm \partial_y$. A geodesic that \emph{is} vertical in a cusp is said to escape in that cusp. 
\begin{definition}
The \emph{scattered geodesics} are geodesics on $M$ that escape in a cusp for both $t\to +\infty$ and $ t\to -\infty$.
\end{definition}
The set of scattered geodesics is denoted by $\mathcal{SG}$. Such a geodesic, when lifted to $\widetilde{M}$, goes from one parabolic point to another, and hence is entirely determined by its endpoints. Take $p,\ q$ representing $Z_i,\ Z_j$. For $\gamma,\ \gamma' \in \Gamma$, the pair of endpoints $(p, \gamma q)$ and $(\gamma' p, \gamma' \gamma q)$ represent the same geodesic on $M$. We let $\mathcal{SG}_{ij}$ be the set of geodesics scattered from $Z_i$ to $Z_j$. From the above, we deduce that when $i\neq j$,
\begin{equation}\label{eq:identification-scattered-geodesics}
\mathcal{SG}_{ij} \simeq \Gamma_i \backslash \Gamma \slash \Gamma_j \quad \text{ and } \quad \mathcal{SG}_{ii} \simeq \Gamma_i \backslash (\Gamma - \Gamma_i) / \Gamma_i,
\end{equation} 
where $\Gamma_{i}$ (resp. $\Gamma_j$) is any maximal parabolic subgroup representing $Z_{i}$ (resp $Z_j$).

On the other hand, we can consider the set of $C^1$ curves that start in $Z_i$ above the torus $\{ y= a_i\}$ and end in $Z_j$, above the torus $\{y= a_j\}$. Among those curves, we can consider the classes of equivalence under free homotopy. Let $\pi^{ij}_1(M)$ be the set of such classes. One can prove that in each class $[c]\in \pi^{ij}_1(M)$, there is exactly one element $\overline{c}$ of $\mathcal{SG}_{ij}$. In particular, this proves that $\mathcal{SG}$ is countable. Hence, we have an identification $\mathcal{SG}{ij} \simeq \pi^{ij}_1(M)$. In what follows, when there is no ambiguity on the metric, we will write directly $c\in \pi^{ij}_1(M)$. In section \ref{section:continuity-parametrix-metric}, we will study variations of the metric, and will come back to the notation $[\overline{c}] \in \pi^{ij}_1(M)$.

For a scattered geodesic $c_{ij}$, we define its \emph{Sojourn Time} in the following way. Take one of its lifts $\tilde{c}_{ij}$ to $\widetilde{M}$, with endpoints $p$, $q$. Let $T$ be the (algebraic) time that elapses between the first time $\tilde{c}_{ij}$ hits $\{\tilde{y}_{p}= a_i\}$, and the last time it crosses $\{\tilde{y}_{q}= a_j\}$. Then, let
\begin{equation}\label{eq:def-T}
\mathcal{T}(c_{ij}) := T - \log a_i - \log a_j.
\end{equation}
This does not depend on the choice of $a_i$ and $a_j$ (as defined in \eqref{eq:def-Z})\index{$a_0$, $a_i$}, nor on the choice of the lift $\tilde{c}_{ij}$. We say that $\mathcal{T}(c_{ij})$ is the \emph{Sojourn Time} of $c_{ij}$, and we can see $\mathcal{T}$ as a function on $\pi^{ij}_1(M)$. Given $T>0$, there is a finite number of $c\in \mathcal{SG}_{ij}$ with sojourn time less than $T$ (otherwise, we would have two such curves that would be so close from one another that they would be homotopic).

We denote by $\mathcal{ST}$ (resp. $\mathcal{ST}_{ij}$) the set of $\mathcal{T}(c)$ for scattered geodesics (resp. between $Z_i$ and $Z_j$). We also call the \emph{Sojourn Cycles} and denote by $\mathcal{SC}$ the set of sums 
\begin{equation}\label{eq:def-SC}
\mathcal{T}_1 + \dots + \mathcal{T}_\kappa 
\end{equation}
where $\sigma$ is a permutation of $\{ 1, \dots, \kappa \}$, and $\mathcal{T}_i\in \mathcal{ST}_{i, \sigma(i)}$. A set of scattered geodesics $\{c_1, \dots, c_\kappa \}$ such that $c_i \in \mathcal{SG}_{i\sigma(i)}$ will be called a \emph{geodesic cycle}.

\subsection{A convergence lemma for modified Poincar\'e series}\label{section:Poincare_series}

Poincar\'e series is a classical object of study in the geometry of negatively curved spaces --- see \cite{Dalbo-Otal-Peigne} for example. For $\Gamma$ a group of isometries on $\widetilde{M}$, its Poincar\'e series at $x\in \widetilde{M}$ is
\begin{equation*}
P_{\Gamma}(x,s) = \sum_{\gamma\in\Gamma} e^{-s d(x, \gamma x)},\quad s\in \R.
\end{equation*}
More generally, given a \emph{Potential} on $S \widetilde{M}$, i.e a H\"older function $V$ on $S \widetilde{M}$ invariant by $\Gamma$, its \emph{Poincar\'e series} is
\begin{equation*}
P_{\Gamma, V}(x,s):= \sum_{\gamma \in \Gamma} e^{\int_x^{\gamma x} V - s}
\end{equation*}
where $\int_x^{\gamma x} V - s$ is the integral of $V -s$ along the geodesic from $x$ to $\gamma x$. The convergence of both series does not depend on $x$, only on $s$.

\begin{remark}
We will not write $\int (V-s)$ to reduce the size of the expressions. We will assume that the integrand is all that is written after the sign $\int$, until we encounter another $\int$ sign. 

When $p$ is a point on the boundary, $x$ and $x'$ in $\widetilde{M}$, $\int_x^p - \int_{x'}^p V$ will refer to the limit of $\int_x^{\tilde{p}} V - \int_{x'}^{\tilde{p}} V$ as $\widetilde{M}\owns\tilde{p} \to p$. When $V$ is H\"older, this limit exists because the geodesics $[x,p]$ and $[x',p]$ are exponentially close.

When we sum over $\{ [\gamma] \in \Xi\backslash \Gamma\}$ we mean that we sum over a set of representatives for $\Xi\backslash \Gamma$ ($\Xi$ being assumed to be a subgroup of $\Gamma$).

We only work with reversible potentials $V$. That means that $\imath V$ is cohomologous to $V$ (following \cite{PPS-12}, $\imath$ is the antipodal map in $S \widetilde{M}$). In other words, we require that
\begin{equation}
\int_x^y V - \imath V = A(y) - A(x)
\end{equation}
where $A$ is a bounded H\"older function on $S \widetilde{M}$, invariant by $\Gamma$. In particular when this is the case, we can replace $V$ by $\imath V$ in the integrals, losing a $\mathcal{O}(1)$ remainder. It is then harmless to integrate along a geodesic in a direction or the other.
\end{remark}

In our case where $\Gamma$ is the $\pi_1$ of $M$, it is a general fact that there is a finite $\delta(\Gamma, V)\in \R$ such that $P_{\Gamma, V}$ converges for $s> \delta(\Gamma, V)$ and diverges for $s< \delta(\Gamma, V)$. This number is called the \emph{critical exponent} of $(\Gamma, V)$. We also call $\delta_\Gamma=\delta(\Gamma, 0)$ the critical exponent of $\Gamma$.

\begin{remark}
The exponent of convergence of a maximal parabolic subgroup $\Gamma_p$ is always $\delta_{\Gamma_p} = d/2$. Additionally, the Poincar\'e series for $\Gamma_p$ diverges at $d/2$ ($\Gamma_p$ is \emph{divergent}). This can be seen computing explicitely with the formula for the distance between two points $(y, \theta)$ and $(y, \theta')$ in the half-space model of the real hyperbolic space $\Hh^{d+1}$
\begin{equation}\label{eq:length-upper-space}
d((y, \theta),(y, \theta)) = 2 \mathrm{argsh} \frac{|\theta - \theta'|}{2y}.
\end{equation}
\end{remark}

\begin{definition}\label{lemma:def-admissible-potential}
In what follows, we say that a potential $V$ is \emph{admissible} if the following holds. First, $V$ is H\"older function on $S \widetilde{M}$, invariant by $\Gamma$ and reversible. Second, there are positive constants $C, \lambda$, and a constant $V_\infty\in \R$ such that whenever $T >0$, if $\pi \varphi_t(\xi)$ stays in an open set of constant curvature $-1$ for $t\in [0, T]$, then for $t \in [0, T]$, 
\begin{equation}\label{eq:def-admissible-potential}
|V(\varphi_t(\xi)) - V_\infty | \leq C e^{-\lambda t}.
\end{equation}
\end{definition}

Observe that an admissible potential has to be bounded. We will mostly use the potential $V_0=(F^{su}+d)/2$ where $F^{su}$ is the unstable jacobian (see \eqref{eq:def-unstable-jacobian} and \eqref{eq:def-V_0}). We start with the following lemma:
\begin{lemma}\label{lemma:divergence}
Let $V$ be an admissible potential. Then $\delta(\Gamma, V) > \delta_{\Gamma_p} + V_\infty$.
\end{lemma}

If $V=0$, this is the consequence of \cite[Proposition 2]{Dalbo-Otal-Peigne}. We will actually follow their proof closely, but before, we need 2 observations on triangles in $\widetilde{M}$.
\begin{remark}\label{remark:Topogonov}
\textbf{1.} Consider a triangle with sides $a,\ b,\ c$ and angles $\alpha, \beta, \gamma$ in a complete Hadamard space $\widetilde{M}_k$ of curvature $-k^2$. We have 
\begin{equation*}
\cosh k c = \cosh k a \cosh k b - \sinh k a \sinh k b \cos \gamma.
\end{equation*}
Assume that $\gamma > \pi /2$ (the triangle is obtuse). Then we find that there is a constant $C_k>0$ --- smooth in $k \neq 0$ and $k\neq \infty$ --- such that
\begin{equation}\label{eq:obtuse-triangle}
| c - (a+b) | \leq C_k.
\end{equation}
Since the curvature of $M$ is pinched, by the Topogonov comparison theorem for triangles, the same is true for obtuse triangles in $ \widetilde{M} $, with a constant $C$ controlled by $k_{min}$ and $k_{max}$.

\textbf{2.} Now, we consider a triangle with sides $c_0, c_1, c_2$ in $\widetilde{M}$, and $V$ an admissible potential on $\widetilde{M}$. Take $C>0$. Among those triangles, we restrict ourselves to the ones such that the length of $c_0$ is at most $C$. Then
\begin{equation}\label{eq:triangle-small-side}
\int_{c_1} - \int_{c_2} V = \mathcal{O}(1).
\end{equation}
this is still valid if the vertex at $c_1 \cap c_2$ is at infinity. Actually, to prove this, first observe that it suffices to make computations for that case when $c_1 \cap c_2$ is at infinity. Then it follows directly from the fact that the two curves are exponentially close in that case. 
\end{remark}

\begin{proof}[Proof of lemma \ref{lemma:divergence}]
The limit set of $\Gamma_p$ is reduced to $\{ p\}$. In $H(p, a_i)$, $\Gamma_p$ has a Borelian fundamental $\mathscr{B}$ domain whose closure is compact. We can obtain a fundamental domain $\mathscr{G}$ for $\Gamma_p$ on $\partial_\infty \widetilde{M}\setminus \{p\}$ by taking the positive endpoints of geodesics from $p$ through $\mathscr{B}$. From \cite[Proposition 3.9]{PPS-12}, which is due to Patterson, there exists a Patterson density $\mu$ of dimension $\delta(\Gamma, V)$ on $\widetilde{M}$, i.e, a family of finite non-zero borelian measures $(\mu_x)_{x\in \widetilde{M}}$ on $\partial_\infty \widetilde{M}$, so that for any $x, x'\in \widetilde{M}$, $\gamma \in \Gamma$, 
\begin{equation}
\gamma_\ast \mu_x = \mu_{\gamma x} \quad \frac{d\mu_x}{d\mu_{x'}}(q) = \exp\left\{ \int_x^q - \int_{x'}^q V-\delta(\Gamma, V) \right\},\ q\in \partial_\infty \widetilde{M}.
\end{equation}
Additionally, the $\mu_x$'s are exactly supported on $\Lambda(\Gamma)=\partial_\infty \widetilde{M}$, so $\mu_x(\mathscr{G}) > 0$. Take $x\in \mathscr{B}$. We have
\begin{equation*}
\infty > \mu_x(\partial_\infty \widetilde{M}) = \sum_{\gamma \in \Gamma_p} \mu_x(\gamma \mathscr{G}) + \mu_x(\{p\})
\end{equation*}
But, 
\begin{equation*}
\mu_x(\gamma^{-1} \mathscr{G}) = \gamma_\ast \mu_x(\mathscr{G}) = \int_{\mathscr{G}} \exp\left\{\int_{\gamma x}^q - \int_x^q V - \delta(\Gamma, V) \right\} d\mu_x(q)
\end{equation*}
So we find
\begin{equation*}
\int_{\mathscr{G}} \sum_{\gamma \in \Gamma_p} \exp\left\{\int_{\gamma x}^q - \int_x^q V - \delta(\Gamma, V) \right\} d\mu_x (q) < \infty.
\end{equation*}

For $q\in \mathscr{G}$, let $x_q \in \mathscr{B}$ be its projection on $H(p,a_i)$. Since we have $d(x,x_q) = \mathcal{O}(1)$ --- from the choice of $\mathscr{B}$ --- we use \eqref{eq:triangle-small-side} and uniformly in $\gamma \in \Gamma_p$,
\begin{equation*}
\int_{x_q}^q - \int_x^q V - \delta(\Gamma, V) = \mathcal{O}(1) \quad ; \quad  \int_{x_q}^{\gamma x} - \int_x^{\gamma x} V - \delta(\Gamma, V) = \mathcal{O}(1)
\end{equation*}

Take $z(x')$ the intersection of the geodesic $[q, x']$ and the horosphere $H_q$ based at $q$ through $x_q$. The set of $z(x')$, $x' \in H(p,a_i)$ has to be bounded. Indeed, $H(p,a_i)$ is not compact, but the only way to go to infinity in $H(p,a_i)$ is to tend to $p$, and we find that as $x' \to p$, $z(x') \to x_q$. The geometry is described in figure \ref{fig:Poincare-series-1}.

\begin{figure}
\def\svgwidth{0.6\linewidth}
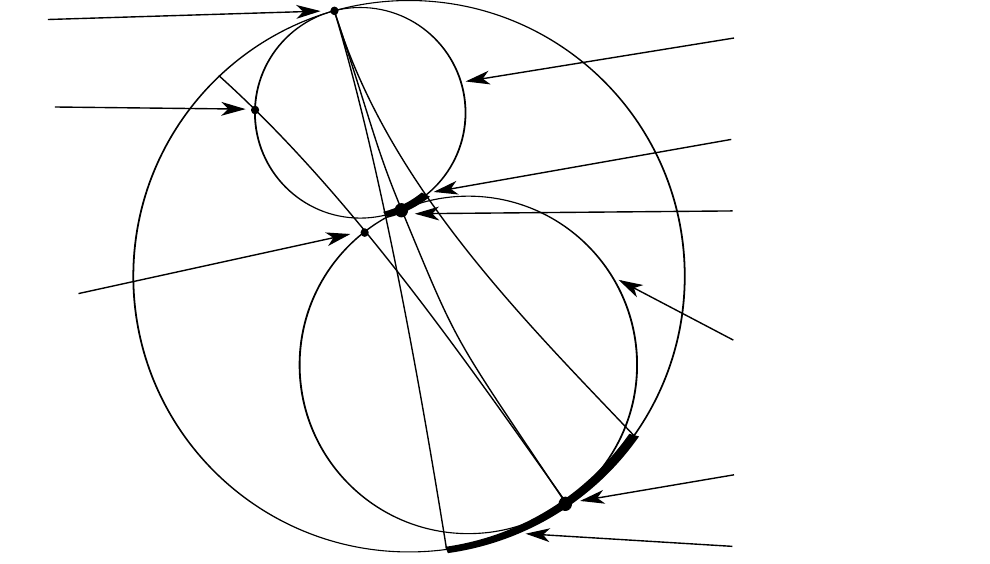
\caption{\label{fig:Poincare-series-1}} 
\end{figure}

Using again \eqref{eq:triangle-small-side},
\begin{equation*}
\int_{x_q}^q - \int_{z(x')}^q V - \delta(\Gamma, V) = \mathcal{O}(1) \quad ; \quad \int_{x'}^{z(x')} - \int_{x'}^{x_q} V - \delta(\Gamma, V) = \mathcal{O}(1),
\end{equation*}
and sum everything up (with $x'=\gamma x$)
\begin{align}
\int_{\gamma x}^q - \int_x^q V - \delta(\Gamma, V) &= \mathcal{O}(1) + \int_{\gamma x}^{z(\gamma x)} + \int_{z(\gamma x)}^q - \int_{x_q}^q V - \delta(\Gamma, V) \notag \\
							&= \mathcal{O}(1) + \int_{\gamma x}^{x_q} V - \delta(\Gamma, V)\notag \\
							&= \mathcal{O}(1) + \int_{\gamma x}^x V - \delta(\Gamma, V).\label{eq:Poincare-series-decomposition-1}
\end{align}
As a consequence, 
\begin{equation*}
P_{\Gamma_p, V}(x, \delta(\Gamma, V))\mu_{x}(\mathscr{G}) < \infty,
\end{equation*}
and since $\mu_{x}(\mathscr{G}) > 0$,
\begin{equation}
P_{\Gamma_p, V}(x, \delta(\Gamma, V)) < \infty.
\end{equation}

Since $V$ is an admissible potential, $P_{\Gamma_p, V}(x, \delta(\Gamma, V)) \asymp P_{\Gamma_p}(x, \delta(\Gamma, V) - V_\infty)$. Since $\Gamma_p$ is divergent, we deduce that $\delta(\Gamma, V) - V_\infty > d/2$.  
\end{proof}

In the article, we will need the convergence of a \emph{modified Poincar\'e series}. Take $V$ an admissible potential. For a cusp $Z_i$, take a point $p\in \Lambda^i_{par}$, and let $\pi^{a_i}_p$ be the intersection of the geodesic through $p$ and $x$ with $H(p, a_i)$. The horoballs $B(p,a_i)$, $p\in \Lambda_{par}^i$ are all pairwise disjoint. Indeed, the restriction of the projection $\widetilde{M} \to M$ to any such horoball is a universal cover of $Z_i$. This implies that for $x\in B(p,a_i)$, the part of the orbit of $x$ under $\Gamma$ that stays in $B(p,a_i)$ has to be its orbit under $\Gamma_p$.

For $x \in M$, take $\tilde{x}\in \widetilde{M}$ a lift of $x$, and define
\begin{equation*}
P_{Z_i, V} (x,s) := \sum_{ [\gamma] \in \Gamma_p \backslash\Gamma, \gamma \tilde{x} \notin B(p,a_i)} \exp\left\{ \int_{\pi^{a_i}_p(\gamma \tilde{x})}^{\gamma \tilde{x}} V-s \right\}. 
\end{equation*}
This does not depend on the choice of $\tilde{x}$. Given a point $x\in \widetilde{M}$, among a family $\{ \gamma x, [\gamma]\in \Gamma_p \backslash \Gamma\}$, there is at most one point in $B(p, a_i)$, and such a point has to be one that minimizes $G_p$. So, for a point $x\in M$, let $x_p$ be a lift minimizing $G_p$ among the lifts of $x$. For $s\in \R$, also let $G_i(x) := G_p(x_p)$. For $q\in \Lambda_{par}^j$, also let
\begin{equation}\label{eq:convergence_series}
P^{ij}_V(s):=\sum_{\Gamma_p\gamma \Gamma_q \neq \Gamma_p} \exp \left\{(V_\infty - s)\mathcal{T}(p, \gamma q) + \int_p^{\gamma q} V- V_\infty \right\},
\end{equation}
where $\mathcal{T}(p, \gamma q)$ is the sojourn time of the geodesic on $M$ that lifts to $[p, \gamma q]$. Observe that the set $\{\Gamma_p \gamma \Gamma_q \neq \Gamma_p\}$ can be identified with $\mathcal{SG}_{ij}$, from equation \eqref{eq:identification-scattered-geodesics}. The main result of this section is
\begin{lemma}\label{lemma:convergence-modified-Poincare}
The series $P_{Z_i, V}(x,s)$ and $P^{ij}_V$ converge if and only if $s> \delta(\Gamma, V)$. Additionally, when $\epsilon>0$,  there is a constant $C_\epsilon >0$ such that for $s> \delta(\Gamma, V) + \epsilon$, 
\begin{equation}
\left\| P_{Z_i,V}(x,s)\right\|_{L^2(M)} \leq C_\epsilon
\end{equation}
\end{lemma}

Our proof is inspired by \cite{Belabas-Hersonsky-Paulin}, and we generalize their Theorem 1.1. One can also see the article \cite{Parkkonen-Paulin-13}, or the proposition 3 in \cite{Paulin-13}. For two real valued functions $f$ and $g$, we write $f \asymp g$ when there is a constant $C>0$ with $C g \leq f \leq g / C$. In the following, when we use that notation, we let the constant $C$ depend on $s$, but not on $x,\ \gamma,\ p$. We fix a cusp $Z_i$, a representing parabolic point $p\in \Lambda^i_{par}$.

\begin{proof}

The proof is divided into 3 parts. First, we compare the values of terms of the sum for different $x$'s, to check that the convergence does not depend on $x$ indeed. Then, we study the sum for some well chosen $x$, to find the convergence exponent. At last, we turn to asymptotics in cusps. We let $P^\ast$ be the series where we have not excluded $\gamma x_q \in B(p, a_i)$ from the sum.

\textbf{1.} Take $x, x'$ two points in $M$, at distance $D>0$, and two lifts $\tilde{x}$ and $\tilde{x}'$ such that $d(\tilde{x}, \tilde{x}') = D$.

Take $\gamma \in \Gamma$. Assume that $G_p(\gamma \tilde{x}') \geq G_p(\gamma \tilde{x})$. Then the projection $x_\gamma^1$ of $\gamma \tilde{x}'$ on the horoball $B(p, G_p( \gamma \tilde{x}))$ is at distance $\mathcal{O}(D+1)$ from $\gamma \tilde{x}$. This is a simple consequence of equation \eqref{eq:obtuse-triangle} for the triangle with vertices $\gamma \tilde{x},\gamma \tilde{x}', x_\gamma^1$. Write
\begin{equation*}
\int_{\pi_p^{a_i}(\gamma \tilde{x}')}^{\gamma \tilde{x}'} V - s - \int_{\pi_p^{a_i}(\gamma \tilde{x})}^{\gamma \tilde{x}} V - s = \int_{x_\gamma^1}^{\gamma \tilde{x}'} V - s + \int_{\pi_p^{a_i}(\gamma \tilde{x}')}^{x_\gamma^1} V - \int_{\pi_p^{a_i}(\gamma \tilde{x})}^{\gamma \tilde{x}} V.
\end{equation*}
Since $V$ is H\"older, and bounded, we deduce that
\begin{equation*}
\int_{\pi_p^{a_i}(\gamma \tilde{x}')}^{\gamma \tilde{x}'} V - s - \int_{\pi_p^{a_i}(\gamma \tilde{x})}^{\gamma \tilde{x}} V - s = \mathcal{O}(D + 1)\left\{ (1 + |s|) + \int_0^\infty (e^{-k_{min} t})^{\mu} dt \right\}
\end{equation*}
where $\mu$ is the H\"older exponent of $V$. The constants in the estimates do not depend on $x$ and $x'$. We have used that the geodesics joining $\gamma \tilde{x}$, $\pi_p^{a_i}\gamma \tilde{x}$ and $x_\gamma^1$, $\pi_p^{a_i}\gamma \tilde{x}'$ are on the same strong stable manifold. We deduce that for some constant $C>0$,
\begin{equation}
e^{-C(D|s|+ 1)} \leq \frac{P^\ast(x', s)}{P^\ast(x, s)} \leq e^{C (D|s|+ 1)} \quad x,x'\in M, d(x,x') = D.
\end{equation}

\textbf{2.} Take now a point $x\in M$ so that $ x_p \in H(p, a_i)\subset B(p, a_i)$. We claim that for all $x'\in H(p,a_i)$,
\begin{equation}\label{eq:Poincare-series-decomposition-2}
\int_{x'}^{\gamma  x_p } V-s = (V_\infty - s)d(x',\pi_p^{a_i}(\gamma x_p)) + \mathcal{O}(1) + \int_{\pi_p^{a_i}(\gamma x_p)}^{\gamma  x_p } V-s.
\end{equation}
The remainder being bounded independently from $ x_p $ and $\gamma$. Let us assume that this holds for now. Then, we write
\begin{align*}
P_{\Gamma, V}( x_p ,s) - P_{\Gamma_p, V}( x_p ,s) &= \sum_{\Gamma_p \gamma \neq \Gamma_p} \sum_{\alpha \in \Gamma_p} \exp\left\{\int_{\alpha  x_p }^{\gamma  x_p } V-s\right\}, \\
				& \asymp \sum_{\Gamma_p \gamma \neq \Gamma_p} \sum_{\alpha \in \Gamma_p} \exp\left\{(V_\infty - s) d(\alpha x_p, \pi_p^{a_i}(\gamma x_p)) + \int_{\pi_p^{a_i}(\gamma x_p)}^{\gamma  x_p } V-s\right\} \\
				& \asymp P_{Z_i, V}( x_p ,s) P_{\Gamma_p}( x_p ,s- V_\infty).
\end{align*}
Hence
\begin{equation*}
P_{Z_i, V}( x_p ,s)\asymp \frac{P_{\Gamma, V}( x_p ,s) - P_{\Gamma_p, V}( x_p ,s)}{ P_{\Gamma_p}( x_p ,s- V_\infty)}.
\end{equation*}
But from lemma \ref{lemma:divergence}, we know that $\delta(\Gamma, V) > \delta(\Gamma_p, V)$.

The proof of \ref{eq:Poincare-series-decomposition-2} is left as an exercise, very similar to the proof of \ref{eq:Poincare-series-decomposition-1} --- replacing $q$ by $\gamma  x_p $, $\gamma x$ by $\pi_p^{a_i}(\gamma x_p)$ and $x$ by $x'$.

\textbf{3.} We turn to asymptotics in the cusps. Take $x\in Z_j$ and $q\in \Lambda^j_{par}$ (if $i=j$, take $p=q$). Let $x_q$ minimize $G_q$ among the lifts of $x$. Observe that the map $( \Gamma_p \gamma \Gamma_q \neq \Gamma_p, \alpha \in \Gamma_q ) \mapsto \Gamma_p \gamma \alpha$ is a bijection onto $\Gamma_p \backslash \Gamma$ if $i\neq j$, and $\Gamma_p \backslash (\Gamma - \Gamma_p)$ if $i=j$. We hence rewrite
\begin{equation*}
P_{Z_i, V} (x,s) = \sum_{\Gamma_p\gamma \Gamma_q \neq \Gamma_p} \sum_{\alpha \in \Gamma_q} \exp\left\{\int_{\pi^{a_i}_p(\gamma \alpha x_q)}^{\gamma \alpha x_q} V-s\right\}.
\end{equation*}

Consider $H_q$ the horosphere based at $q$, through $x_q$. Let $z_\gamma$ (resp. $z_\gamma'$) be the point of intersection of the geodesic $[p, \gamma q]$ with $\gamma H_q$ (resp. $H(p,a_i)$). From \eqref{eq:Poincare-series-decomposition-1}, we have
\begin{equation*}
\int_{\pi^{a_i}_p(\gamma \alpha x_q)}^{\gamma \alpha x_q} V-s = \mathcal{O}(1) + \int_p^{z_\gamma} + \int_{z_\gamma}^{\gamma\alpha x_q} - \int_p^{\pi_p^{a_i}(\gamma \alpha x_q)} V-s .
\end{equation*}

However, the distance between $z_\gamma'$ and $\pi_p^{a_i}(\gamma \alpha x_q)$ is uniformly bounded. This is a direct consequence of lemma \ref{lemma:point-entry-compact}. Hence
\begin{equation*}
\int_{z_\gamma'}^{p} -\int_{\pi_p^{a_i}(\gamma \alpha x_q)}^p = \mathcal{O}(1),
\end{equation*}
and
\begin{equation*}
\int_{\pi^{a_i}_p(\gamma \alpha x_q)}^{\gamma \alpha x_q} V-s = \left\{\int_p^{\gamma q} V- V_\infty\right\} + (V_\infty - s)\left(\mathcal{T}(p, \gamma q) - G_q(x_q) + d(\gamma^{-1} z_\gamma, \alpha x_q)\right) + \mathcal{O}(1).
\end{equation*}
where $\mathcal{T}(p,\gamma q)$ is the sojourn time for the geodesic $[p, \gamma q]$. It follows that
\begin{equation*}
\begin{split}
P_{Z_i, V}(x,s) \asymp \sum_{\Gamma_p\gamma \Gamma_q \neq \Gamma_p} \exp &\left\{ (V_\infty - s)\mathcal{T}(p, \gamma q) + \int_p^{\gamma q} V- V_\infty\right\}  \\
	&\times\sum_{\alpha \in \Gamma_q} \exp\left\{(V_\infty-s)( - G_q(x_q) + d(\gamma^{-1} z_\gamma, \alpha x_q)) \right\}
\end{split}
\end{equation*}

In the RHS, the first term does not depend on $x$; we recognize $P_V^{ij}(s)$. The second is related to $P_{\Gamma_q}(x_q)$. We can see it as a Riemann sum as $x_q \to q$. Indeed, $\Gamma_q \simeq \Z^d$, and we can write explicitely the second term as
\begin{equation}
e^{(s-V_\infty) G_q(x_q)} \sum_{\theta \in \Lambda_i} \exp \left\{2(V_\infty - s) \argsh \frac{|\theta - \theta_0|}{2 e^{-G_q(x_q)}} \right\}
\end{equation}
As $x_q \to q$, $y= e^{-G_q(x_q)} \to +\infty$, and we can see this as a Riemann sum for the function $f=\exp \{ 2 (V_\infty - s) \argsh \}$ for the parameter $2y$. It should be equivalent to $(2y)^d \int_{\R^d} f$. However $f$ is integrable if and only if $s-V_\infty > d/2$. As a result, we find that
\begin{equation*}
\sum_{\alpha \in \Gamma_q} \exp\left\{(V_\infty-s)( - G_q(x_q) + d(\gamma^{-1} z_\gamma, \alpha x_q)) \right\} \asymp e^{(s-V_\infty -d) G_q(x_q)}, \quad s > V_\infty + d/2.
\end{equation*}
It is easy to check that the $L^2$ norm of this is finite whenever $s \geq V_\infty + d/2 + \epsilon$. The proof of the lemma is complete when we observe that the $L^2$ norm decreases when $\Re s$ increases.
\end{proof}

\section{Parametrix for the Eisenstein functions}\label{sec:param-E}

In the case of constant curvature, the universal cover $\widetilde{M}$ is the real hyperbolic space $\Hh^{d+1}$. On it, there is the \emph{Poisson kernel} $P(x, p, s)$ that associates a function on the boundary $f(p)$ with a function on $\Hh^{d+1}$, $u(x)$ such that
\begin{equation*}
(-\Delta -s(d-s))u(x) =0 \quad u(x) = \int P(x,p,s) f(p) dp.
\end{equation*}
In addition, we require that $u$ corresponds to the superposition of outgoing stationary plane waves at frequency $s$, with weight $f(p)$ in the direction $p$. When the curvature is variable, one cannot build such a kernel anymore, because the geometry of the space near the boundary is quite singular. In other words, the metric structure on the boundary is not differentiable, only H\"older. Hence, no satisfactory theory of distributions is available. However, in the special case of parabolic points that correspond to hyperbolic cusps, the fact that small enough horoballs have constant curvature enables us to construct an \emph{approximate} Poisson kernel for $p\in \Lambda_{par}$. 

Taking the half space model for $\Hh^{d+1}$, the Poisson kernel for the point $p= \infty$ is $P = y^s$, so one can rewrite formula \eqref{eq:Expansion_Eisenstein_constant_curvature} as
\begin{equation*}
E_i(s,x) = \sum_{[\gamma] \in \Gamma_p \backslash \Gamma} P(\gamma x, p, s).
\end{equation*}

This is exactly the type of expression we are looking for. In the first subsection, we introduce some notations. In the second we recall some facts on Jacobi fields that we will need. Then we build the approximate Poisson kernel, and later, we prove that summing over $p\in \Lambda^i_{par}$ gives a good approximation of $E_i$.

\subsection{Some more notations}\label{section:some_more_notations}

Fix some $Z_i$ and let $p\in \Lambda^i_{par}$ be a parabolic point. We denote by $\varphi^p_t$ the flow on $ \widetilde{M} $ generated by $\nabla G_p$. It is conjugated to the geodesic flow on $W^{u0}(p)$ by the projection $\pi : T^\ast  \widetilde{M}  \to  \widetilde{M} $. We have 
\begin{equation*}
\frac{\mathrm{d}}{\mathrm{d}t} \Jac \varphi^p_t |_{t=0} = \Tr \nabla^2 G_p = \Delta G_p,
\end{equation*}
so that the Jacobian is
\begin{equation}\label{eq:integral-expression-jacobian}
\Jac \varphi^p_t  = \exp\left\{ \int_0^t \Delta G_p \circ \varphi^p_\tau \mathrm{d}\tau \right\}.
\end{equation}
Thanks to the rigid description in the cusps, we have 
\begin{equation*}
G_p \leq -\log a_i \Leftrightarrow \text{ we are above cusp $Z_i$ and } G_p = -\log y_i, \text{ for all $p \in \Lambda^i_{par}$.}
\end{equation*}
In that case, we can compute $\Delta G_p = d$, and it makes sense to define a \emph{twisted} Jacobian:
\begin{equation}\label{eq:def-J}
\tilde{J}_p(x) := \lim_{t\to +\infty} \sqrt{\Jac\varphi^p_{-t}  e^{t d }} = \sqrt{ \Jac \varphi^p_{-t} e^{t d}}_{t \geq G_p(x) + \log a_i}, \text{ for }p\in \Lambda^i_{par}.
\end{equation}
This $\tilde{J}_p$ is constant equal to $1$ in the horoball $B(p, a_i)$. It is useful to define 
\begin{equation}\label{eq:def-b-i}
b_i := \inf \{y>0,\ B(p, y) \text{ has constant curvature} \}.
\end{equation}
We have $b_i \leq a_i$, and $\tilde{J}$ equals $1$ on $B(p, b_i)$. We also let 
\begin{equation}\label{eq:def-F}
F_p(x):= \log \tilde{J}_p(x).
\end{equation}
Recall the curvature of $M$ is pinched between $-k^2_{max} \leq -1 \leq -k^2_{min}<0$. Then by Rauch's comparison theorem, \cite[Theorem 1.28]{Cheeger-Ebin},
\begin{equation}\label{eq:Bound_Jacobian}
d(1- k_{max}) \leq \frac{2 F_p}{(G_p + \log b_i)^+} \leq d(1-k_{min}).
\end{equation}
What is more, by \ref{lemma:Regularity_Jacobian}, $\nabla^n F_p$ is bounded for $n\geq 1$, because $\nabla G_p$ is in $\mathscr{C}^\infty( \widetilde{M} )$. 
 
On the other hand, the Unstable Jacobian $F^{su}$ is the H\"older function on $S M$ defined by
\begin{equation}\label{eq:def-unstable-jacobian}
F^{su}(x,v) := - \frac{d}{dt}_{| t=0} \det \left[(d\varphi_t)_{| E^u(x,v)}\right] < 0.
\end{equation}
The fact that it is H\"older is a consequence of the H\"older regularity of $E^u$ --- see \cite[Theorem 7.1]{PPS-12}. In what follows, we will be interested by the potential
\begin{equation}\label{eq:def-V_0}
V_0 = \frac{1}{2} F^{su} + \frac{d}{2}.
\end{equation}
We let $\delta_g = \delta(\Gamma, V_0)$. This is the relevant abscissa of convergence of theorem \ref{theorem:Dirichlet-series-expansion-vertical-strip} in the introduction, as we will see.

	\subsection{Unstable Jacobi fields}\label{section:Jacobi-fields}

We want to relate $V_0$ and $F_p$. We have to make a digression, and recall some facts on Jacobi fields. Take a geodesic $x(t)$, and a Jacobi field $J$ along $x(t)$, orthogonal to $x'(t)$. By parallel transport, one can reduce $J$ to some function of time valued in $T_{x(0)}M$. If one also uses parallel transport for the curvature tensor, we get the equation
\begin{equation}\label{eq:Jacobi_fields}
J''(t) + K(t) J(t) = 0.
\end{equation}
if $x(t)$ lives in constant curvature $-1$, $K$ is the constant matrix $- \mathbf{1}$. If $J(0) = J'(0)$, then $J(t) = e^t J(0)$, and conversely, if $J(0) = - J'(0)$, $J(t) = e^{-t} J(0)$. 

For $v\in T M$, denote by $v^\perp$ the space of vectors in $T_x M$ orthogonal to $v$. Recall that $H$ and $V$ are the horizontal and vertical subspaces introduced in remark \ref{remark:measuring_regularity}. Then we can identify $T_v S M \simeq (\R v \oplus v^\perp) \oplus v^\perp$. In this identification, the first term $\R v \oplus v^\perp$ is $H$. The second term $v^\perp$ is $V \cap T_v SM$. In this notation, $\R v$ is the direction of the geodesic flow, and $v$ its vector. 

This identification is consistent with Jacobi fields in the sense that if
\begin{equation*}
\mathrm{d}\varphi_t .(l, v_1,v_2) = (l(t), v_1(t), v_2(t)),
\end{equation*}
then $l(t) = l$ for all $t$, $v_1(t)$ is a Jacobi field orthogonal to $v(t) =x'(t)$, and $v_2(t)$ is its covariant derivative (also orthogonal to $v(t)$).

An \emph{unstable Jacobi field} $\J^u(t)$ along $x(t)$ is a $d\times d$ matrix-valued solution of \ref{eq:Jacobi_fields} along $x(t)$ that is invertible for all time, and that goes to $0$ as $t\to -\infty$ --- it just gathers a basis of solutions. Similarly, one can define the \emph{stable} Jacobi fields. Such fields always exist; they never vanish, nor does their covariant derivative --- see \cite{Ruggiero-07}. We denote by $\J_t^u(s)$ the unstable Jacobi field that equals $\mathbf{1}$ for $s=t$ --- given a geodesic $x(t)$. Actually, $s\mapsto \J_t^u(t+s)$ only depends on $v=(x(t), x'(t))\in S M$. We will write it $s\mapsto \J_v^u(s)$.

From the identification with $T SM$, we find that vectors in $E^u$ take the form $(\J^u(t)w, {\J^{u}}'(t)w)$, whence we deduce that 
\begin{equation}\label{eq:coordinates-E^u}
E_v^u = \{ (w, {\J_v^u}'(0)w) | w \perp v \}.
\end{equation}
The matrix ${\J_v^u}'(0)$ only depends on $v$, we denote it by $\U_v$. Similarly, we define $\Ss_v$ for the \emph{stable} Jacobi fields. They satisfy the Ricatti equation (along a geodesic $v(t)$):
\begin{equation*}
\U' + \U^2 + K = 0.
\end{equation*}
They take values in symmetric matrices (with respect to the metric), which is equivalent to saying that the stable and unstable directions are Lagrangians. Given a geodesic curve $x(t)$, $\J_u(t)$ and $\J_s(t)$ two Jacobi fields along it, we can write $\U = (\J_u^{-1})^T (\J_u')^T$, and find that
\begin{equation}\label{eq:Wronskian}
\frac{d}{dt} \left\{ \J_u^T (\U - \mathbb{S}) \J_s \right\} = 0.
\end{equation}
This is a Wronskian identity. We can also compute
\begin{equation}\label{eq:unstable_jacobian_Jacobi_fields}
\det {d \varphi_t}_{|E^u(v)} = \det \J_v^u(t) \sqrt{\frac{\det \left(\mathbf{1} + \U_{\varphi_t(v)}^2\right)}{\det \left( \mathbf{1} + \U_{v}^2\right)}}.
\end{equation}
We have a map $i^u : w \in H(v) \mapsto (w, \U_v w) \in E^u(v)$ from the horizontal subspace to the unstable one. If one considers the metric $ds^2_u$ obtained on $E^u$ by restriction of the Sasaki metric in $T SM$, this gives a structure of Euclidean bundle to $E^u$ over $SM$. 

\begin{lemma}
The matrix $\mathbf{1} + \U_v^2$ is the matrix of the metric $i^\ast ds^2_u$ on $H$. This is bounded uniformly on $SM$.
\end{lemma}

\begin{proof}
this metric is always $\geq \mathbf{1}$ --- here, $\mathbf{1}$ refers to the metric on $H$, i.e, the metric on $TM$. The only way it can blow up would be that for a sequence of $v$, $\tilde{v} \perp v$, $\U_v\tilde{v} \to \infty$. If $v_\infty$ was a point of accumulation of $v$ in $\widetilde{M}$, that implies that $E^u$ and $H$ are not transverse at $v_\infty$. That is not possible since there are no conjugate points in strictly negative curvature. We deduce that $\pi v \in M$ has to escape in a cusp. 

However, in the cusp, the curvature $K$ is constant with value $-\mathbf{1}$. Hence, unstable Jacobi fields in the cusp write as $A e^t + B e^{-t}$, where $A$ and $B$ are constant matrices along the orbit. Then $\U_v = \mathbf{1} + \mathcal{O}(e^{-t})$ as the point $v$ travels along a trajectory $\varphi_t$ that remains in a cusp. In particular, $i^\ast ds^2_u = 2.\mathbf{1} + \mathcal{O}(1/y)$ for points of height $y$ in a cusp.
\end{proof}

In this context, from the definition, we find that for $x\in \widetilde{M} $, 
\begin{equation}\label{eq:J_and_Jacobi_fields}
\tilde{J}^2_p(x) = e^{td} \det \J_{(x, \nabla G_p(x))}^u(-t), \text{ for } t\geq G_p(x) + \log a_i.
\end{equation}

As a consequence, 
\begin{lemma}\label{lemma:equivalence-Jacobian}
For $x\in  \widetilde{M} $, and $t\in \R$, 
\begin{equation*}
\int_x^{\varphi^p_t(x)} V_0 = F_p(\varphi^p_t(x)) - F_p(x) + \mathcal{O}(1).
\end{equation*}
What is more, $V_0$ is an admissible potential.
\end{lemma}

\begin{proof}
The first part of the lemma comes directly from equations \eqref{eq:J_and_Jacobi_fields} and \eqref{eq:unstable_jacobian_Jacobi_fields}, and the observation just afterward.

To prove the second part, it suffices to prove that $F^{su}$ is an admissible potential. Consider a point $v\in TSM$ so that $\varphi_t(v)$ remains in a cusp for times $t\in [0, T]$. Taking the Jacobi fields starting from $v$ along its orbit, for $t\in [0, T]$, we find
\begin{equation}\label{eq:formula-U}
\U_{\varphi_t v} = (A e^t - B e^{-t})(Ae^t + B e^{-t})^{-1} = \mathbf{1} + \mathcal{O}(e^{-t}),
\end{equation}
and 
\begin{equation*}
F^{su} = - \frac{d}{ds}_{|s=0}\left\{ \det {\J_{\varphi_t(v)}^u}(s) \sqrt{\frac{\det \mathbf{1} + \mathcal{O}(e^{-t})}{\det \mathbf{1} + \mathcal{O}(e^{-t-s})}}\right\} = - d + \mathcal{O}(e^{-t}).
\end{equation*}

The last thing we have to check is that $F^{su}$ is reversible. However, $\imath F^{su}$ is the strong \emph{Stable} Jacobian $F^{ss}$
\begin{equation}
F^{ss} = \frac{d}{dt}_{|t=0} \log \det {d\varphi_t}_{|E^s(x,v)}.
\end{equation}
From equation \eqref{eq:unstable_jacobian_Jacobi_fields}, and the Wronskian identity \eqref{eq:Wronskian}, we find that
\begin{equation}
\det {d\varphi_t}_{|E^s(x,v)} \det {d\varphi_t}_{|E^u(x,v)} = \frac{\det \U_v - \Ss_v}{\det \U_{\varphi_t(v)} - \Ss_{\varphi_t(v)}} \sqrt{\frac{\det \left(1+\U_{\varphi_t(v)}^2\right)\left(1+\Ss_{\varphi_t(v)}^2\right)}{\det \left(1+\U_v^2\right)\left(1+\Ss_v^2\right)}}.
\end{equation}
Since the function
\begin{equation}
\frac{\sqrt{\det (1+\U^2)(1+\Ss^2)}}{\det \U - \Ss}
\end{equation}
is well defined on $S\widetilde{M}$, H\"older continuous, and bounded, $F^{su}$ is reversible. 
\end{proof}

%In what follows, we will only use unstable and stable Jacobi fields along the orbits of $\varphi^p_t$. We will thus write for $x\in \widetilde{M}$, $\J^u_x(t)$ the unstable Jacobi field along $\varphi^p_t(x)$, with value $\mathbf{1}$ at $x$.

	\subsection{On the universal cover}

In this section, we fix $p\in \Lambda_{par}$, and we omit the dependency on $p$; it shall be restored afterwards. We use notations introduced in section \ref{section:some_more_notations}. We will use the WKB Ansatz to find our approximate Poisson kernel. Consider a formal series of functions on $ \widetilde{M} $, 
\begin{equation*}
f(x)=\sum_{n\geq 0} s^{-n} f_n,
\end{equation*}
with $s\in \C$ and $f_0=1$, and compute
\begin{equation*}
(-\Delta - s(d-s))[e^{-sG} \tilde{J} f ] = e^{-sG}\left[ s \left( 2\nabla G. \nabla(\tilde{J} f) + \tilde{J} f \Delta G  -  \tilde{J} f d \right) - \Delta (\tilde{J} f)\right],
\end{equation*}
where we have used that $G$ satisfies the eikonal equation $|\nabla G|^2 = 1$. If we expand the formal series, we find that this expression (formally) vanishes if for all $k>0$,
\begin{equation*}
2\tilde{J} \nabla G. \nabla f_n  = \Delta (\tilde{J} f_{n-1}).
\end{equation*}
Indeed,
\begin{equation*}
2 \nabla G . \nabla \tilde{J} = \tilde{J}( d - \Delta G).
\end{equation*}
We can rewrite those equations in terms of $F=\log \tilde{J}$ :
\begin{equation}\label{eq:def-Q}
2 \nabla G.\nabla f_n = Q f_{n-1} \text{ where } Q f(x) = \Delta f  + 2\nabla F. \nabla f  + ( |\nabla F|^2+ \Delta F) f.
\end{equation}
These are transport equations, with solutions :
\begin{equation}\label{eq:def-f_n}
f_n = \frac{1}{2} \int_{-\infty}^0 (Q f_{n-1})\circ \varphi^p_\tau \mathrm{d}\tau 
\end{equation}
Remark that on $\{ G \leq -\log b\}$, from the definition \eqref{eq:def-b-i} $F$ vanishes, and so does $Qf_0$. Hence all $f_n$'s but $f_0$ vanish, and the formula above is legit. We prove :

\begin{lemma}\label{lemma:estimate_f}
There are constants $C_{n,N}>0$ for $n>1$, $N \in \N$, such that for all $\tau \in \R^+$
\begin{equation*}
\| f_n \|_{\mathscr{C}^N(\{G_p \leq \tau - \log b\})} \leq C_{n,N} \tau^n.
\end{equation*}
\end{lemma}

\begin{proof}
We use lemma \ref{lemma:Regularity_Jacobian} again, and proceed by induction. The result is obvious for $n=0$. Now assume it holds for some $n\geq 0$. Taking $g_0 = f_n$, $g_1 = f_{n+1}$, $\ell = -\log b$, the lemma enables us to conclude directly if we can prove that
\begin{equation*}
\| Q f_n \|_{\mathscr{C}^k(G \leq \tau - \log b)} \leq C_{n,k} \tau^n.
\end{equation*}
But this is a simple consequence of the induction hypothesis and the fact that $\nabla F \in \mathscr{C}^\infty( \widetilde{M} )$.
\end{proof}

All the functions defined above depended on choosing a parabolic point $p$, and now we make it appear in the notations:
\begin{equation}\label{eq:approximate-Poisson}
f_p^N  (s):= \sum_{n=0}^N   s^{-n} f_{n,p} \text{ and } P_ N (\cdot,p,s):= e^{-sG}\tilde{J}_p f_p^ N (s).
\end{equation}
This is the approximate Poisson kernel. Then for all $ N >0$,
\begin{equation*}
[-\Delta - s(d-s)] e^{-s G_p} \tilde{J}_p f_p^N  (s) =  - s^{- N } e^{-sG_p}\tilde{J}_p Q_p f_{ N ,p}.
\end{equation*}
so we let
\begin{equation}\label{eq:approximate-remainder}
R_N(., p, s) :=  - e^{-sG_p}\tilde{J}_p Q_p f_{ N ,p}
\end{equation}
This will be the remainder term. Now, as the last point in this section, observe the equivariance relation
\begin{equation}\label{eq:equivariance-Poisson}
P_ N (\gamma x, p, s) = P_ N (x, \gamma^{-1}p, s).
\end{equation}

	\subsection{Poincar\'e series and convergence}

The functions defined by \eqref{eq:approximate-Poisson} and \eqref{eq:approximate-remainder} in $ \widetilde{M} $ are already invariant under the action of $\Gamma_p$, so to define a function on $M$, we only have to sum over $\Gamma_p \backslash \Gamma$. As in section 1.3, take a cusp $Z_i$, a parabolic point $p \in \Lambda^i_{par}$. For $x \in M$, let $x_p\in  \widetilde{M} $ be a point minimizing $G_p$ amongst the lifts of $x$. Then
\begin{lemma}\label{lemma:convergence_parametrix_E}
For $\epsilon >0$, and $ N \in \N$, there is a constant $ C_{ N ,\epsilon}>0$ such that for all $x\in M$, and all $\Re s > \delta(\Gamma, V_0) + \epsilon$,
\begin{equation*}
\left\| \sum_{[\gamma] \in \Gamma_p \backslash \Gamma, [\gamma] \neq [0]} | P_ N (\gamma x_p, p,s) | \right\|_{L^2_x(M)} < C_{N, \epsilon}  a_i^{\Re s}.
\end{equation*}
Further, with the same condition on $s$, the remainder satisfies
\begin{equation*}
\left\|\sum_{[\gamma] \in \Gamma_p \backslash \Gamma} | R_N(\gamma x_p, p, s) | \right\|_{L^2_x(M)} < C_{ N ,\epsilon} b_i^{\Re s }.
\end{equation*} 
with $b_i$ as defined in \eqref{eq:def-b-i}.
\end{lemma}

\begin{proof}
First, we give a proof for $ N =0$. write
\begin{equation*}
\sum_{[\gamma] \in \Gamma_p \backslash \Gamma, [\gamma] \neq [0]} | P_0(\gamma x_p, p,s) |  = \sum_{[\gamma] \in \Gamma_p \backslash \Gamma, [\gamma] \neq [0]} \exp\left\{ \Re s \log a_i +  F_p(\gamma x_p) - \int_{\pi_p^{a_i} (\gamma x_p)}^{\gamma x_p}\Re s \right\}.
\end{equation*}
Recall $- F_p(\pi_p^{a_i} (\gamma x_p)) =0$. By lemma \ref{lemma:equivalence-Jacobian}, losing constants not depending on $s$, the RHS is comparable with
\begin{equation*}
a_i^{\Re s}\sum_{[\gamma] \in \Gamma_p \backslash \Gamma, [\gamma] \neq [0]} \exp\left(\int_{\pi_p^{a_i} (\gamma x_p)}^{\gamma x_p} V_0 - \Re s\right).
\end{equation*}
Lemma \ref{lemma:convergence-modified-Poincare} states that the term in the right part of the product is bounded uniformly in $L^2$ norm.

Now, we deal with the higher order of approximation. Let $n>0$ and consider the sum
\begin{equation*}
\sum_{[\gamma] \in \Gamma_p \backslash \Gamma} (e^{-sG_p}\tilde{J}_p f_{n,p})\circ \gamma.
\end{equation*}
By lemma \ref{lemma:estimate_f}, this is bounded term by term by
\begin{equation*}
C_k \sum_{[\gamma] \in \Gamma_p \backslash \Gamma} (e^{-sG_p}\tilde{J}_p ((G_p + \log b_i)^+)^n )\circ \gamma.
\end{equation*}
Inserting $1 = b_i^s b_i^{-s}$, this is
\begin{equation}\label{eq:differentiating-Dirichlet}
C_k b_i^s \sum_{[\gamma] \in \Gamma_q \backslash \Gamma, G_q \geq - \log b_i } \partial_s^k(e^{-s(G_q+ \log b_i)}\tilde{J}_q  )\circ \gamma.
\end{equation}
Let
\begin{equation*}
L_0^{b_i}:= \sum_{[\gamma] \in \Gamma_q \backslash \Gamma, G_q \geq - \log b_i } (e^{-s(G_q+ \log b_i)}\tilde{J}_q  )\circ \gamma.
\end{equation*}
By the argument above, for $s> \delta(\Gamma, V_0) + \epsilon$, this sum converges and the value is bounded in $L^2$ norm by some function of $s$. What is more, since all the exponents are nonpositive, this a decreasing function of $s\in \R$. We deduce that when $\epsilon >0$, there is $C_\epsilon>0$ such that for $x\in M$, $\Re s > \delta(\Gamma, V_0) + \epsilon$, we have $\|L_0^{b_i}\|_{L^2} \leq C_\epsilon$. 

Consider $L= \sum a_k \lambda_k^s$ a Dirichlet series, with $a_k\in \R^+$, $\lambda_k \geq 1$, converging for $\Re s > s_0$. Then, if $s-\epsilon > s_0$, we find $|L'(s)| \leq L(\Re s-\epsilon) \sup_n |\log \lambda_n| \lambda_n^{-\epsilon}$. Since $L_0^{b_i}$ has this Dirichlet series structure in the $s$ variable, it implies that for some constants $C_{\epsilon, k}>0$,
\begin{equation*}
\|\partial_s^k L_0^{b_i}(s)\|_{L^2} \leq C_{\epsilon, k},\ x\in M,\ \Re s > \delta(\Gamma, V_0) + \epsilon.
\end{equation*}
Observe that $C_{\epsilon, k}$ may depend on $b_i$. Hence
\begin{equation*}
\left\|\sum_{[\gamma] \in \Gamma_p \backslash \Gamma} |(e^{-sG_p}\tilde{J}_p f_{n,p})\circ \gamma|\right\|_{L^2(M)} \leq C_{\epsilon, n} b_i^{\Re s},\ x\in M,\ \Re s > \delta(\Gamma, V_0) + \epsilon.
\end{equation*}
Moreover, this also holds if we replace $f_{n,p}$ by $Q_p f_{ N ,p}$, and this observation concludes the proof.
\end{proof}

Now, we can prove our first theorem:
\begin{theorem}[Parametrix for the Eisenstein functions]\label{theorem:parametrix_E}
For $ N  \in \N$, let $Z_i$ be some cusp, and $p\in \Lambda^i_{par}$ a representing point. For $\Re s > \delta(\Gamma, V_0)$, let
\begin{equation*}
E_{i, N }(s,x) := \sum_{[\gamma] \in \Gamma_p \backslash \Gamma} P_ N (\gamma x_p, p,s). 
\end{equation*}
this function is defined on $ \widetilde{M} $, but invariant by $\Gamma$, so it descends to $M$; it does not depend on the choice of $p\in \Lambda^i_{par}$. Then, uniformly in $s$ when $\Re s$ stays away from $\delta(\Gamma, V_0)$, and $s\notin [d/2, d]$,
\begin{equation}\label{eq:parametrix_E_remainder}
\|\partial_s^m( E_i - E_{i, N }) \|_{H^k(M)} = \mathcal{O}\left( s^{k- N } b_i^s \right)
\end{equation}
\end{theorem}

\begin{proof}
From lemma \ref{lemma:convergence_parametrix_E}, we deduce that $E_{i, N }$ is well defined; it does not depend on $p$ thanks to the equivariance relation \eqref{eq:equivariance-Poisson}. Additionally, for a cutoff $\chi$ that equals $1$ sufficiently high in $Z_i$ and vanishes outside of $Z_i$, $E_{i, N } - \chi y_i^s$ is in $L^2$, uniformly bounded in sets $\{ \Re s > \delta(\Gamma, V_0) + \epsilon\}$. The sum
\begin{equation*}
\sum_{[\gamma] \in \Gamma_p \backslash \Gamma}  R_N(\gamma \cdot, p, s)   
\end{equation*}
converges normally on compact sets, and in $L^2(M)$ also, so we find
\begin{equation*}
E_i - E_{i, N } = s^{- N } \left( -\Delta - s(d-s) \right)^{-1} \sum_{[\gamma] \in \Gamma_p \backslash \Gamma}  R_N(\gamma \cdot, p, s)
\end{equation*}
Since $\partial_s^m\left( -\Delta - s(d-s) \right)^{-1}$ is bounded on $H^n(M)$ with norm $\mathcal{O}(1)$ when $ s$ stays in sets $\{ \Re s > d/2 + \epsilon, s \notin [d/2,d]\}$, it suffices to prove that when $\Re s > \delta(\Gamma, V_0) + \epsilon$,
\begin{equation*}
\left\|\partial_s^m\sum_{[\gamma] \in \Gamma_p \backslash \Gamma}  R_N(\gamma \cdot, p, s) \right\|_{H^k(M)} = \mathcal{O}( s^k b_i^{\Re s})
\end{equation*}
Actually, since the sum has a Dirichlet series structure, we see that this is true for all $m\geq 0$ as long as it is true for $m=0$. From the bounds in lemma \ref{lemma:estimate_f}, and the bounds on $\nabla G_p\in \mathscr{C}^\infty$, we see that for $x'\in \widetilde{M}$,
\begin{equation*}
\| \nabla^k R_N(\cdot, p, s) \|(x') \leq C e^{-s G_p(x')} \tilde{J}_p((G_p+ \log b_i)^+)^N  
\end{equation*}
We conclude the proof using the arguments of the proof of lemma \ref{lemma:convergence_parametrix_E} again --- from equation \eqref{eq:differentiating-Dirichlet} and below.
\end{proof}

\begin{remark}
We have given estimates for the convergence in $H^k$, $k\geq 0$. However, the sum also converges normally in $\mathscr{C}^k$ topology on compact sets. 
\end{remark}

\section{Parametrix for the scattering matrix}\label{sec:param-phi}

Let us recall that the zero-Fourier mode of $E_i$ at cusp $Z_j$ is
\begin{equation*}
y^s \delta_{ij} + \phi_{ij}(s) y^{d-s}.
\end{equation*}
This formula is valid a priori for $y\geq a_j$. However, if we integrate $E_i$ along a projected horosphere of height $b_i \leq y \leq a_i$, we still obtain the same expression, even though the projected horosphere may have self-intersection --- recall they are the projection in $M$ of horospheres in $\widetilde{M}$. This is true because following those projected horospheres, we do not leave  an open set of constant curvature $-1$ --- see \eqref{eq:def-b-i} --- and we can apply a unique continuation argument.

The smaller the $b_i$'s are, the better the remainder is. In constant curvature, there is no remainder --- the remainder in \ref{eq:parametrix_E_remainder} goes to zero as $N \to \infty$, with fixed $s$. Observe that the parameters $b_i$ are only related to the support of the variations of the curvature, and not to their size.

\subsection{Reformulating the problem}

In this section, let $p\in \Lambda^i_{par}$, $q\in \Lambda^j_{par}$. Recall from \eqref{eq:identification-scattered-geodesics} that when $i\neq j$, $\mathcal{SG}_{ij} \simeq \Gamma_p \backslash \Gamma / \Gamma_q$, and $\mathcal{SG}_{ii} \simeq \Gamma_p \backslash (\Gamma - \Gamma_p) / \Gamma_p$. We prove 
\begin{lemma}\label{lemma:first_estimate_phi}
When $\Re s > \delta(\Gamma, V_0)$, integrating on horospheres in $ \widetilde{M} $,
\begin{equation}\label{eq:phi_integral_parametrix_2}
\phi_{ij}(s) = b_j^s \sum_{[\gamma] \in \mathcal{SG}_{ij}} \int_{H(\gamma q, b_j)} P_ N (\cdot, p,s)d\mu(\theta)  + \mathcal{O}( s^{1/2 -  N } b_i^{s} b_j^{ s} ).
\end{equation}
The constants are uniform in sets $\{ \Re s > \delta(\Gamma, V_0) + \epsilon \}$, $\epsilon>0$. What is more, this expansion can be differentiated, differentiating the remainder.
\end{lemma}

\begin{proof}
First, for $\tilde{H}_j(b_j)\simeq \Gamma \backslash H(q,b_j)$ the projected horosphere from cusp $Z_j$ at height $b_j$, integrating in $M$, we claim
\begin{equation}\label{eq:phi_integral_parametrix_1}
\phi_{ij}(s) = - b_j^{2s-d} \delta_{ij} + \int_{x\in \tilde{H}_j(b_j)} b_j^{s-d} E_{i, N } d\theta^d + \mathcal{O}(s^{1/2 -  N } b_i^{s} b_j^{ s} ),
\end{equation}
where the remainder can be differentiated. Considering zero Fourier modes of $E_i$ in the cylinder $\Gamma_q \backslash \widetilde{M}$ we see that the formula holds if we replace $E_{i,N}$ by $E_i$, without remainder. Hence, we only have to estimate
\begin{equation*}
\left|\partial_s^m \left(b_j^{s-d} \int E_i - E_{i, N } d\theta^d \right)\right|.
\end{equation*}
The surface measure obtained by disintegrating the riemannian volume on $\{ y = b_j\}$ is $\mathrm{d}\mu(\theta)= \mathrm{d}\theta^d /b_j^d$. According to the Sobolev trace theorem, the $L^2$ norm of a restriction to $\tilde{H}_j(b_j)$ --- it is an immersed hypersurface --- is controlled by the $H^{1/2}$ norm on $M$. Using this and theorem \ref{theorem:parametrix_E}, we obtain that the remainder is bounded up to a constant by
\begin{equation*}
b_j^{\Re s -d/2} \sup_{k=0, \dots m} \|\partial_s^k( E_i - E_{i, N })\|_{H^{1/2}(M)} = \mathcal{O}( s^{1/2 -  N } b_i^{\Re s} b_j^{\Re s}).
\end{equation*}

Now, to go from \eqref{eq:phi_integral_parametrix_1} to \eqref{eq:phi_integral_parametrix_2}, we just have to use the description given by theorem \ref{theorem:parametrix_E}. Indeed, consider a cube $\mathcal{C}_q$ in $H(q, b_j)$ that is a fundamental domain for the action of $\Gamma_q \simeq \Z^d$. Then
\begin{align*}
\int_{x\in \tilde{H}_j(b_j)} b_j^{s-d} E_{i, N } d\theta^d &= \sum_{[\gamma] \in \Gamma_p \backslash \Gamma}\int_{\mathcal{C}_q} b_j^{s-d} P_ N (\gamma \cdot, p,s) d \theta^d \\
								&  =\sum_{[\gamma] \in \Gamma_p \backslash \Gamma}\int_{\gamma \mathcal{C}_q} b_j^{s-d} P_ N (\cdot, p,s) d \theta^d \\
								&  =\sum_{\substack{[\gamma] \in \Gamma_p \backslash \Gamma/\Gamma_q\\ [\gamma] \neq \Gamma_p}}\sum_{\gamma'\in \Gamma_q} \int_{\gamma \gamma'\mathcal{C}_q} b_j^{s-d} P_ N (\cdot, p,s) d \theta^d + \delta_{ij}\int_{\mathcal{C}_p} b_i^{s-d} P_ N (\cdot, p,s) d \theta^d\\
								&  =\sum_{\substack{[\gamma] \in \Gamma_p \backslash \Gamma/\Gamma_q\\ [\gamma] \neq \Gamma_p}}\int_{H(\gamma q, b_j)}  b_j^{s-d} P_ N (\cdot, p,s) d \theta^d  + \delta_{ij}\int_{\mathcal{C}_p} b_i^{s-d} P_ N (\cdot, p,s) d \theta^d
\end{align*}
It suffices to observe now that
\begin{equation*}
\int_{\mathcal{C}_p} b_i^{s-d} P_ N (\cdot, p,s) d \theta^d = b_i^{2s- d}.
\end{equation*}
\end{proof}

We want to give an asymptotic expansion for each term in \eqref{eq:phi_integral_parametrix_2}. To be able to use stationary phase, the next section is devoted to giving sufficient geometric bounds on the position of $H(\gamma q, b_j)$ with respect to $W^{u0}(p)$.

	\subsection{Preparation lemmas and main asymptotics}\label{sec:main_asymptotics}

Take $p\in \Lambda^i_{par}$, $q\in \Lambda^j_{par}$, $q \neq p$. We will work in $H(q,b_j)\subset  \widetilde{M} $. As an embedded Riemannian submanifold, it is isometric to $\R^d$, and the isometry is given by the $\theta$ coordinate; we use this to measure distances on $H(q,b_j)$ unless mentioned otherwise. We are considering
\begin{equation}\label{eq:stationary_1}
b_j^{s-d}\int_{H(q,b_j)} P_ N (\cdot, p,s)\mathrm{d}\mu(\theta) = \int_{\R^d} e^{-s (G_p-\log b_j)} \tilde{J}_p f_p^N  (s, \theta) \mathrm{d}\theta^d
\end{equation}
At all the points where $\nabla G_p$ is not orthogonal to the horosphere, this integral is non-stationary as $|s| \to +\infty$. There is only one point in $H(q,b_j)$ where $\nabla G_p$ \emph{is} orthogonal to $H(q,b_j)$; it is exactly the point where the geodesic $c_{p,q}$ from $p$ to $q$ intersects $H(q,b_j)$ for the first time --- the second is $q$. It is reasonable to expect that the behaviour of the approximate Poisson kernel around this point will determine the asymptotics of the integral.

It is indeed the case, as we will show that $G_p$, $\tilde{J}_p$ and $f_p^N  $ satisfy appropriate symbol estimates on $H(q,b_j)$. If $a\in C^\infty(\R^d)$, we say that $a$ is a \emph{symbol} of order $n\in \Z$ if for all $k \in \N$,
\begin{equation}\label{eq:symbol_R^d}
|\langle x \rangle^{-n+k} \partial^k a(x) |_{L^\infty(\R^d)} < \infty \quad \text{ where } \langle x \rangle^2 = 1+ x^2.
\end{equation}

For a geodesic coming from $p$ intersecting $H(q,b_j)$, call the first intersection the \emph{point of entry} and the second one the \emph{exit point} --- they may be the same. We can assume that the point of entry of $c_{p,q}$ is $0$ in the $\theta$ coordinate --- denoted $0_\theta$. It is also the point where $G_p$ attains its minimum on $H(q,b_j)$, and this is $\mathcal{T}(c_{p,q}) + \log b_j$, with $\mathcal{T}(c_{p,q})$ the sojourn time as defined in \eqref{eq:def-T}. We start with a lemma :

\begin{lemma}\label{lemma:point-entry-compact}
The set of entry points $\mathcal{I}\subset H(q, b_j)$ is compact. Its radius is bounded independently from $p,q$, for the distance on $H(q,b_j)$ given by $H(q,b_j) \simeq \R^d$.
\end{lemma}

\begin{proof}
First, we prove it is compact. By continuity, $\mathcal{I}$ contains a small neighbourhood $U$ of $0_\theta$. Let $U'$ be the set of exit points of geodesics whose entry point is in $U$. It is a neighbourhood of $q$ in $H(q,b_j)$ --- by definition of the \emph{visual} topology on $\overline{M}$. The complement of $U'$ has to contain $\mathcal{I}$, and it is compact, so $\mathcal{I}$ is relatively compact. The claim follows because $\mathcal{I}$ is closed.

Now, since $G_p$ is $C^\infty$, and the horosphere is smooth, the boundary of $\mathcal{I}$ only contains points where $\nabla G_p$ is tangent to $H(q,b_j)$. Take such a point $\theta$, and consider the triangle with vertices $p$, $0_\theta$, and $\theta$. Let $\alpha$ be the angle at $0_\theta$, and $L$ the distance between $0_\theta$ and $\theta$ in $ \widetilde{M} $. Since the horoball $B(q,b_j)$ is convex, $\alpha > \pi /2$. From the remark on obtuse triangles \eqref{eq:obtuse-triangle} for $[p, \theta, 0_\theta]$, if $l= G_p(\theta) - G_p(0_\theta)$, we have $L-l = \mathcal{O}(1)$. 

\begin{figure}[h]
%\begin{wrapfigure}{r}{0.35\textwidth}
\centering
\def\svgwidth{0.40\textwidth}
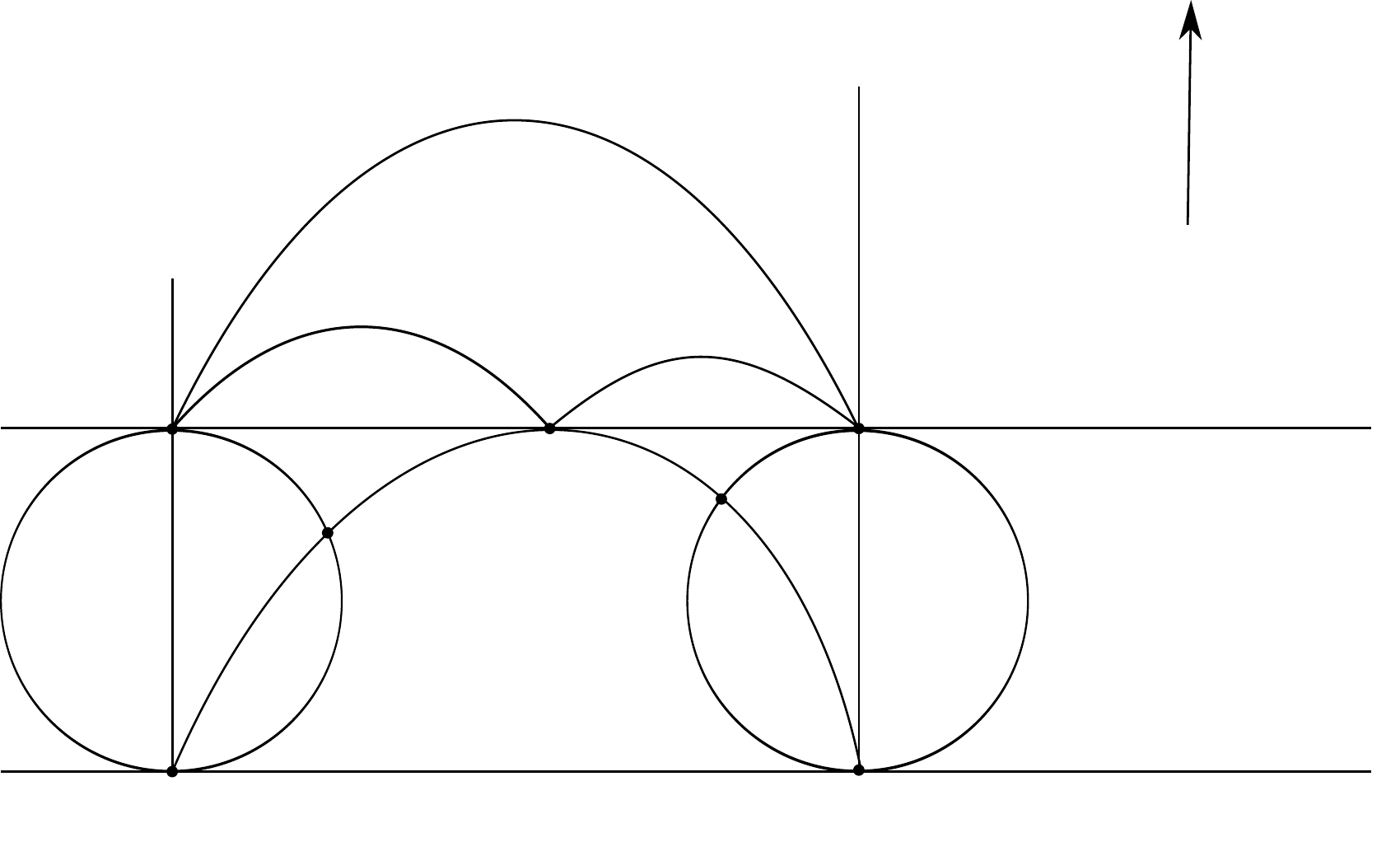
\vspace*{-10pt}
\caption{\label{fig:zones_resonances} the situation.}
%\end{wrapfigure}
\end{figure}

We want to prove that $L$ is bounded independently from $p$, $q$. To see that, consider $p'$ the other endpoint of the geodesic through $p$ and $\theta$, and $\theta'$ its projection on $H(q,b_j)$. Let $l' = G_{p'}(\theta) - G_{p'}(\theta')$ and let $L'$ be the distance in $ \widetilde{M} $ between $\theta$ and $\theta'$. By the same argument, $L' = l' + \mathcal{O}(1)$. Moreover, $l+l'$ is the distance between $H(p,G_p(0_\theta))$ and $H(p', G_{p'}(\theta'))$. So that if $L''$ is the distance between $0_\theta \in H(p,G_p(0_\theta))$ and $\theta'\in H(p',G_{p'}(\theta'))$, $L'' > l+l'$. However, by the triangle inequality, $L'' \leq L+ L'$. We deduce that $L''= L+ L' + \mathcal{O}(1)$. By theorem 4.9 and 4.6 of \cite{Heintze-ImHof}, $L''$ is bounded by constants depending only on the pinching of $M$, and so is $L$.

Additionally, from equation \eqref{eq:length-upper-space}, we deduce:
\begin{equation*}
|\theta| = 2 b_j \sinh \frac{L}{2}.
\end{equation*}
\end{proof}

Our second lemma is the following:
\begin{lemma}\label{lemma:non-degeneracy}
In $\mathcal{I}$, $G_p$ is convex. That is, on $\mathcal{I}$, if $\overline{\alpha}$ is the angle $\nabla G_p$ makes with the horosphere $H(q,b_j)$, we have $d_\theta^2 G_p \geq (\sin \overline{\alpha} + \sin^2 \overline{\alpha} K_{min})/b_j^2$.
\end{lemma}

\begin{proof}
Let $\theta \in \mathcal{I}$. Take $\mathbf{u} \in \R^d$ with $|u| = 1$ and $\theta' = \theta+  \epsilon \mathbf{u}$ for $\epsilon>0$ small. We apply Topogonov's theorem to the triangle with vertices $\theta$, $\theta'$ and $\tilde{p}$, where $\tilde{p}$ is a point that will tend to $p$. Let $\alpha(\epsilon)$ be the angle at $\theta$. Then by comparison, we have
\begin{equation*}
\begin{split}
\cosh (K_{min} d(\theta',\tilde{p})) \geq \cosh(K_{min} & d(\theta, \tilde{p})) \cosh(K_{min} d(\theta, \theta'))\\
			& - \sinh(K_{min} d(\theta, \tilde{p})) \sinh(K_{min} d(\theta, \theta'))\cos \alpha(\epsilon).
\end{split}
\end{equation*}
As we let $\tilde{p}\to p$,
\begin{equation*}
\frac{\cosh(K_{min}d(\theta',\tilde{p}))}{\cosh(K_{min}d(\theta, \tilde{p}))} \text{  and  } \frac{\cosh(K_{min}d(\theta',\tilde{p})) }{ \sinh(K_{min}d(\theta, \tilde{p}))} \to \exp( K_{min}(G_p(\theta') - G_p(\theta)))
\end{equation*} 
and
\begin{equation*}
K_{min}(G_p(\theta') - G_p(\theta)) \geq \log \left[ \cosh(K_{min} d(\theta, \theta')) - \sinh(K_{min} d(\theta, \theta'))\cos \alpha(\epsilon) \right].
\end{equation*}
Now, we let $\epsilon$ go to $0$. We have $G_p(\theta') - G_p(\theta) = \epsilon\nabla G_p(\theta).\mathbf{u} + \epsilon^2 d_\theta^2 G_p(\theta).\mathbf{u}^{\otimes 2}/2 + o(\epsilon^2)$. Additionally, $\nabla G_p(\theta).\mathbf{u} = -\cos \alpha(0)/b_j$ and $d(\theta, \theta')\sim \epsilon /b_j$ by \eqref{eq:length-upper-space}, so the RHS becomes
\begin{equation*}
\log \left[ 1 - \frac{K_{min} \epsilon \cos\alpha(\epsilon)}{b_j} +  \frac{K_{min}^2 \epsilon^2}{2b_j^2}  + o(\epsilon^2)\right]= - \frac{ \epsilon K_{min}}{b_j} \cos\alpha(\epsilon) + \frac{\epsilon^2 K_{min}^2}{2b_j^2}(1   - \cos \alpha(0)^2  + o(1)).
\end{equation*}
And we deduce
\begin{equation*}
d_\theta^2 G_p(\theta). \mathbf{u}^{\otimes 2} \geq  \frac{K_{min}}{b_j^2} \sin^ 2 \alpha(0) + 2\frac{(\cos \alpha)'(0)}{b_j}.
\end{equation*}
Now, computing in the hyperbolic space, we find that the angle $\beta$ at which the geodesic between $\theta$ and $\theta'$ intersects $H(q,b_j)$ satisfies $\beta \sim \epsilon /2b_j$. If $\overline{\alpha}$ is the angle between $\nabla G_p(\theta)$ and $H(q,b_j)$, we find
\begin{equation*}
\cos \alpha = \cos \overline{\alpha}\cos \measuredangle( u, \partial_\theta G_p)  + \sin \overline{\alpha} \sin \beta \quad \text{ and }\quad (\cos \alpha)'(0) = \frac{\sin \overline{\alpha}}{2 b_j}.
\end{equation*}
Finally, we can observe that $\alpha(0) \geq \overline{\alpha}$ and
\begin{equation*}
d_\theta^2 G_p(\theta) \geq \frac{\sin \overline{\alpha} + K_{min} \sin^2 \overline{\alpha}}{b_j^2}.
\end{equation*}
\end{proof}

We have to separate the integral \eqref{eq:stationary_1} into two parts, let us explain how we choose them. The stable and the unstable distributions of the flow $\varphi_t$ are always transverse. Since they are continuous, the angle between them is uniformly bounded by below by some $\alpha >0$ in any given compact set of $M$ --- we say that they are uniformly transverse. Lifting this to $ \widetilde{M} $, the angle is uniformly bounded by below on sets that project to compact sets in $M$. In particular, this is true on the union of the $H(q,b_j)$ for $q\in \Lambda^j_{par}$. 

Now, we can also consider the geodesic flow in the hyperbolic space of dimension $d+1$. It has stable and unstable distributions. The cusp $Z_j$ is the quotient of an open set of that space by a group of automorphisms, so that those stable and unstable distributions project down to subbundles $E^s_{hyp}$, $E^u_{hyp}$ of $TS^\ast Z_i$, invariant by the geodesic flow. We call them the $\ast$-stable and $\ast$-unstable manifolds of $Z_i$. The angle between them is constant equal to $\pi/2$, and they are smooth --- even analytic.\label{page:ast-stable}

By definition of the stable and $\ast$-stable manifolds, if the trajectory of a point $\xi \in S^\ast Z_i$ stays in $Z_i$ for all times positive, its stable and $\ast$-stable manifolds coincide. This is the case of $(0, dG_p)$. As a consequence, there is a small neighbourhood $V_1$ of $0$ in the $\theta$ plane, whose size can be taken independent from $p,q$, where the unstable manifold of $(\theta, dG_p(\theta))$ and its $\ast$-stable manifolds are uniformly transverse.

By the arguments in the proof of lemma \ref{lemma:point-entry-compact}, we see that the set of points of $H(q,b_j)$ that are not exit points of geodesics whose entry points are in $V_1$, is a compact set. Denote it by $V_2$. Its radius is also bounded independently from $p$ and $q$. Now, let $\chi\in C^\infty_c(\R^d)$ take value $1$ on $V_2$, and introduce $1= \chi + (1-\chi)$ in \eqref{eq:stationary_1}, to separate it into $(\mathrm{I})$ and $(\mathrm{II})$.

From theorem 7.7.5 (p.220) in H\"ormander \cite{Hormander-1}, we deduce 
\begin{lemma}\label{lemma:asymptotic_expansion_scattering}
For each $p,q$, there are coefficients $A_n(p,q)$ so that for every $ N \geq 1$
\begin{equation*}
\begin{split}
(\mathrm{I}) = \left(\frac{\pi}{s}\right)^{d/2}& \exp\left\{\int_{c_{p,q}} V_0-s \mathcal{T}(c_{p,q})\right\} \\
	&\left[\sum_{n \leq  N -1} \frac{A_n(p,q)}{s^n}   + \frac{1}{s^{ N }} \mathcal{O}\left( (1+(\mathcal{T}(c_{p,g})+ \log b_j)^+)^{ N } \right)\right].
\end{split}
\end{equation*}
We have $A_0(p,q) = 1$, and $A_n(p,q) = \mathcal{O}(1+ (\mathcal{T}(c_{p,q})+ \log b_j)^+)^n$. What is more, the $A_n$ do not depend on $ N $ for $n\leq  N -1$.
\end{lemma}

\begin{proof}
From lemma \ref{lemma:estimate_f}, we already know that the functions under the integral are smooth, uniformly in $p,q$. From lemma \ref{lemma:non-degeneracy}, we know that the phase is non-degenerate at $0$. To apply H\"ormander's theorem, we need to check that the derivative $|\partial_\theta G_p|$ is uniformly bounded from below in $V_2 - V_1$.

The general observation is that $|\partial_\theta G_p|$ remains bounded from below if $\nabla G_p$ stays away from being the outer normal to $H(q,b_j)$.

Start with $\theta \in V_2$ an exit point. Consider $c$ the geodesic along $\nabla G_p$, going out at $\theta$. The closer to the outer normal $\nabla G_p(\theta)$ is, the longer the time $c$ had to spend in the horoball. Since the set of entry points is uniformly compact, this implies that points where $\nabla G_p$ is almost vertical --- i.e along $\partial_y$ --- have to be far from $0$. But $V_2$ is uniformly bounded, so $|\partial_\theta G_p|$ is bounded by below on the exit points in $V_2$. 

For the entry points that are not exit points, we use the uniform convexity from lemma \ref{lemma:non-degeneracy}. By that lemma, $\partial_\theta G_p$ is a local diffeomorphism in $\mathcal{I}'=\{\theta,\: \nabla G_q(\theta). \nabla G_p(\theta) <0\}$. On the boundary of $\mathcal{I}'$, $|\partial_\theta G_p| = b_j^2$. By continuity, there is $\epsilon>0$ such that $|\partial_\theta G_p(\theta)| < b_j^2 /2$ implies $d(\theta, \partial \mathcal{I}') > \epsilon$. As a consequence, from the local inversion theorem, there is $0< \epsilon' < \epsilon$ and $\epsilon''>0$ such that if $|\partial_\theta G_p(\theta)| < b_j^2 /2$, 
\begin{equation*}
B(\partial_\theta G_p(\theta), \epsilon'') \subset \partial_\theta G_p (B(\theta, \epsilon')).
\end{equation*}
Then, when $|\partial_\theta G_p(\theta)| < \epsilon''$, $\theta$ has to be at most at distance $\epsilon'$ from a zero of $\partial_\theta G_p$, i.e $0_\theta$.

The constants $\epsilon'$ and $\epsilon''$ can be estimated independently from $p$ and $q$.

Now, we have an expansion
\begin{equation}
(I) = \left(\frac{2\pi}{s}\right)^{d/2} \exp\left\{-s\mathcal{T}(c_{p,q})\right\}\left(C_0 + \frac{1}{s} C_1 + \dots \right)
\end{equation} 
We have 
\begin{equation}\label{eq:formula-C_0}
C_0 = \frac{\tilde{J}_p(0_\theta)}{(\det d^2_\theta G_p(0_\theta))^{1/2}}. 
\end{equation}
We factor out $C_0$ from the sum, and define $A_n(p,q) = C_n / C_0$. From lemma \ref{lemma:estimate_f} and the fact that $\nabla F_p$ is $\mathscr{C}^\infty( \widetilde{M} )$, it is quite straightforward to prove the estimates on the $A_n$'s.

Now, we have to compute $C_0$. From \cite[proposition 3.1]{Heintze-ImHof}, we see that
\begin{equation}\label{hessian-G_p}
\nabla^2 G_p (x)= \U_{x, \nabla G_p(x)}.
\end{equation}
Where $\U$ was introduced after equation \eqref{eq:coordinates-E^u}. 

We use a simple trick. Along the geodesic $c_{p,q}$, $\nabla (G_p + G_q) = 0$, so that the Hessian $d^2 (G_p + G_q)$ is well defined along $c_{p,q}$. This implies that $d^2_\theta (G_p + G_q) = \nabla^2 (G_p + G_q)$. But on the horosphere $H(q,b_j)$, $G_q$ is constant, and we find $d^2_\theta G_p(0_\theta) = \nabla^2 G_p(0_\theta) + \nabla^2 G_q(0_\theta)$. 

The unstable Jacobi fields along $c_{q,p}$ are the \emph{stable} Jacobi fields along $c_{p,q}$ so $\U_{x, \nabla G_q(x)} = - \mathbb{S}_{x, \nabla G_p(x)}$. Hence,
\begin{equation}
d^2_\theta G_p(0_\theta) = \U_{x, \nabla G_p(x)} - \mathbb{S}_{x, \nabla G_p(x)}.
\end{equation}
In constant curvature, this is the constant matrix $2\times \mathbf{1}$. 

Now, we give another expression for $\tilde{J}^2_p(0_\theta)$. Let $x\in  \widetilde{M} $. Consider $\J_u$ \label{page:def-J_u}the unstable Jacobi field along $(\varphi^p_t(x))$, that equals $(1/y)\mathbf{1}$ for a point along the orbit that is close enough to $p$ --- where $y$ is the height coordinate $\exp - G_p$. Then 
\begin{equation}\label{eq:formula-J}
\tilde{J}^2_p(x) = \frac{e^{d. G_p(x)}}{ \det \J_u(x)}.
\end{equation}
When $x= 0_\theta$, for $t> 0$, we can write $\J_u(\varphi^p_t(0_\theta)) = A e^t + B e^{-t}$. We can also define $\J_s$ the stable Jacobi field along $\varphi^p_t(0_\theta)$ that equals $\mathbf{1}$ at $0_\theta$. From the equation \eqref{eq:Wronskian}, we find that
\begin{equation}\label{eq:Wronskian-2}
W :=\J_u(t)^T (\U(t) - \mathbb{S}(t)) \J_s(t) \text{ is constant.}
\end{equation}
Hence
\begin{equation*}
C_0^2 = \frac{e^{d. G_p(0_\theta)}}{\det W} \det \J_s(0).
\end{equation*}
The limit for $t\to + \infty$ gives $W = 2 A $. Whence
\begin{equation}
C_0^2 = \frac{e^{d. G_p(0_\theta)}}{2^d \det A}.
\end{equation}
On the other hand, 
\begin{align*}
\exp\left\{ \int_p^q -2 V_0\right\} &= \lim_{t \to +\infty} \det\left( d {\varphi_t}_{| E^u(v)}\right) e^{-t d} \text{ for $v \in [p,q]$ sufficiently close to $p$.} \\
				& = \det\left\{ A e^{-G_p(0_\theta)}\right\} \text{ from formulae \eqref{eq:unstable_jacobian_Jacobi_fields} and \eqref{eq:formula-U}.}
\end{align*}
We conclude that
\begin{equation}
C_0 = 2^{-d/2}\exp\left\{\int_{c_{p,q}} V_0 \right\}.
\end{equation}
\end{proof}

\subsection{Estimating the remainder terms}

Now, we consider 
\begin{equation}\label{eq:integrating-by-parts}
(\mathrm{II}) := \int_{\R^d} e^{-s (G_p- \log b_j)} \tilde{J}_p f_p^N  (s,\theta)(1-\chi) d\theta  = \frac{1}{s^k} \int_{\R^d}  e^{-s G_p} L_q^k(\tilde{J}_p f_p^N  (s,\theta)(1-\chi))d\theta
\end{equation}
where $L_p f = \mathrm{div} \left[\frac{\partial_\theta G_p}{\|\partial_\theta G_p\|^2} f \right]$. This holds for any $k\in \N$; if we get symbolic estimates on the integrand, we will find that $(\mathrm{II}) = \mathcal{O}(s^{-\infty})(\mathrm{I})$. Our next step is to study the growth of $G_p$ as $\theta \to \infty$.

\begin{lemma}\label{lemma:uniform_integrability}
The function $\frac{\partial_\theta G_p}{\| \partial_\theta G_p \|}$ is a symbol of order $1$ in $\theta$ in $\R^d - V_2$, bounded independently from $p,q$.

Additionally, $\exp(-s(G_p - \log b_j))$ is integrable, and for any $\epsilon>0$, there is a constant $C_\epsilon>0$ such that whenever $\Re s > d/2 + \epsilon$,
\begin{equation*}
\int_{\R^d - V_2} e^{-s(G_p - \log b_j)} d\theta \leq C_\epsilon e^{-\Re s \mathcal{T}(c_{p,q})}.
\end{equation*}

\end{lemma}

\begin{proof}
With each $\theta \in \R^d - V_2$ we associate the point of entry $\theta_0$ --- $\theta_0 \in V_1$ by definition. Consider the geodesic coming from $p$, entering the horoball at $\theta_0$ and going out at $\theta$. Then, if $\epsilon$ is the angle of this geodesic with the normal to the horosphere, 
\begin{equation*}
|\theta - \theta_0| = \frac{2 b}{\tan \epsilon}.
\end{equation*}
but, we also have that 
\begin{equation*}
|\partial_{\theta} G_p| = \frac{\sin \epsilon}{b}.
\end{equation*}
Hence
\begin{equation*}
\frac{1}{|\partial_\theta G_p|} = \frac{1}{2} | \theta - \theta_0| \sqrt{1 + \frac{4b^2}{|\theta - \theta_0|^2}},
\end{equation*}
and
\begin{equation*}
\frac{\partial_\theta G_p}{\| \partial_\theta G_p \|^2} = \frac{1}{2} (\theta - \theta_0) \sqrt{1 + \frac{4b^2}{(\theta - \theta_0)^2}}.
\end{equation*}
It suffices to see that $\theta \mapsto \theta_0$ is a symbol of order $-1$ to obtain the first part of the lemma. But $\theta \mapsto \theta_0$ is a one-to-one map, and by means of an inversion in the hyperbolic space, we see that $\theta_0 \to \theta / \| \theta \|^2$ is a smooth map. Its derivatives are controlled by the angle that $\nabla G_p$ makes with the vertical (and its derivatives). As a consequence $\theta \to \theta_0$ is a symbol of order $-1$, uniformly in $p$ and $q$.

Then, using formula \eqref{eq:length-upper-space}, as $\theta \to \infty$, 
\begin{equation*}
G_p = 2 \log \frac{|\theta|}{2b} + G_p(\theta_0) = 2 \log \frac{|\theta|}{2b} + \mathcal{T}(c_{p,q}) + \log b_j + o(1)
\end{equation*}
where the remainder is a symbol of order $-1$. We deduce that $\exp -s G_p$ is integrable ($\Re s > d/2$), and
\begin{equation*}
\int_{\R^d} \mathrm{d}\theta^d e^{-\Re s (G_p-\log b_j)} \leq C e^{- \Re s \mathcal{T}(c_{p,q})}.
\end{equation*}
\end{proof}

\begin{lemma}\label{lemma:uniform-bound-J}
On the horosphere $H(q,b_j)$, $\tilde{J}_p$ is a symbol of order $0$ with respect to $\theta$. In symbol norm, it is $\mathcal{O}(\tilde{J}_p(0_\theta))$.
\end{lemma}

\begin{proof}
We use Jacobi fields and notations introduced in section \ref{section:Jacobi-fields}. We also use the uniform transversality condition in the definition of $V_1$ --- see page \pageref{page:ast-stable}. In the neighbourhood $V_1$, since $E^u$ is transverse to the constant curvature stable direction, there exists a smooth matrix $A(\theta)$ such that
\begin{equation*}
E^u(\theta) = \{ X^+ + X^- | \; X^+ \in E^s_{hyp}, \; X^- \in E^u_{hyp}, \; X^+ = A(\theta) X^- \}.
\end{equation*}
When we transcribe this to Jacobi field coordinates, 
\begin{equation*}
E^u(\theta) = \{ ((\mathbf{1}+A)\xi^\perp, (\mathbf{1}-A)\xi^\perp) | \xi^\perp \in \perp\}
\end{equation*}
Remark here that $(\mathbf{1}+A)$ is invertible ; indeed, if it were not, there would be an unstable Jacobi field on $M$ that would vanish at some point. But a Jacobi field that vanishes at some point cannot go to $0$ as $t \to -\infty$, it has to grow.

Now, we consider a trajectory entering the horoball at $\theta_0$. We use the coordinates $(\theta_0, t)$ to refer to $\varphi^p_t(\theta_0)$. We use parallel transport to work with vectors in $T S^\ast M_{| V_1}$. We have 
\begin{equation*}
E^u(\theta_0, t) = \{ X^+ + X^- | \; X^+ \in E^s_{hyp}, \; X^- \in E^u_{hyp}, \; X^+ = e^{-2t} A(\theta_0) X^- \}
\end{equation*}
and in the horizontal-vertical coordinates
\begin{equation*}
E^u(\theta_0, t) = \{ ((\mathbf{1}+e^{-2t}A)\xi, (\mathbf{1}-e^{-2t}A)\xi) | \xi \perp d/dt\}
\end{equation*}
Actually, for $t\in [0, T]$, the jacobian 
\begin{equation*}
\Jac \: \varphi^p_t (\theta_0)
\end{equation*}
is the determinant of $J(0) \mapsto J(t)$ where $J(t)$ are the unstable Jacobi fields along the trajectory $\varphi^p_t(\theta_0)$. From the description with matrix $A$ above, we deduce that this is 
\begin{equation*}
e^{t.d}\det \frac{\mathbf{1}+ e^{-2t} A(\theta_0)}{\mathbf{1}+ A(\theta_0)}
\end{equation*}
As a consequence, 
\begin{equation*}
\frac{\tilde{J}_p(\varphi^p_t(\theta_0))}{\tilde{J}_p(\theta_0)} = \sqrt{ \det \frac{\mathbf{1}+ e^{-2t} A(\theta_0)}{\mathbf{1}+ A(\theta_0)}}.
\end{equation*}

Recall that $t \sim 2 \log |\theta|$ when the trajectory reaches the horosphere again, and that $\theta \to \theta_0$ is a symbol of order $-1$. We deduce that $\tilde{J}_p (\theta)$ is a symbol of order $0$.
\end{proof}

\begin{lemma}\label{lemma:uniform-bound-f}
For all $n \geq 0$, in the region of the horoball corresponding to trajectories entering in $V_1$, we can write
\begin{equation*}
f_{n,p} (\theta_0, t) = \tilde{f}_{n,p}(\theta_0, e^{-2t}).
\end{equation*}
We have for all $k \geq 0$,
\begin{equation*}
\| \tilde{f}_{n,p} \|_{C^k} \leq C_{n,k} ((\mathcal{T}(c_{p,q}) + \log b_j)^+)^n
\end{equation*}
with $C_{n,k}$ not depending on $p$ nor on $q$.
\end{lemma}

\begin{proof}
We start by considering two functions $a_1$ and $a_2$ of $\theta_0$ and $e^{-2t}$. Then
\begin{equation*}
e^{2t}\nabla a_1 . \nabla a_2  \text{ and } e^{2t} \Delta a_1
\end{equation*}
are still smooth functions of $\theta_0$ and $e^{-2t}$. Consider a trajectory $x(t)=(\theta_0, t)$. We can take normal coordinates along this geodesic $(t, x')$. We then only need to prove that $e^t \partial \theta_0 / \partial x'_{|x'=0}$ is a smooth function of $\theta_0$ and $t$. First, we observe that $\partial \theta_0 / \partial x'_{|x'=0, t=0}$ is only controlled by the angle between the geodesic and the horosphere $H(q,b_j)$, and this angle we have shown to be smooth. We only have to consider $\partial x'(t)/\partial x'(0)_{|x'=0}$, that is, the differential of the flow $\varphi^p_t$ transversally to $\nabla G_p$. We have computed it in the previous proof; it is $e^t (\mathbf{1} + e^{-2t} A(\theta_0))(\mathbf{1} + A(\theta_0))^{-1}$.

Now, we proceed by induction on $n$. First, $f_{0,p} = 1$ so it obviously satisfies the assumptions; it is also the case of $F_p = \log \tilde{J}_p$. Assume that the hypothesis has been verified for some $n \geq 0$. Then by the above and \eqref{eq:def-Q}, \ref{lemma:estimate_f}, $e^{2t} Q_p f_{n,p}$ is a smooth function of $\theta_0$ and $e^{-2t}$, with the same control as for $f_{n,p}$, and
\begin{equation*}
f_{n+1,p}(\theta_0, t) = f_{n+1,p}(\theta_0, 0) + \frac{1}{2}\int_0^t Q_p f_{n,p} \mathrm{d}s.
\end{equation*}
we can write the integral as
\begin{equation*}
\int_0^t e^{-2s} a(\theta_0, e^{-2s}) \mathrm{d}s = \left[ \int a(\theta_0, \rho) \mathrm{d} \rho \right]_{ e^{ -2t } }^1,
\end{equation*}
for some smooth function $a$. This ends the proof.
\end{proof}

Now, recall that $e^{-2t} \sim |\theta - \theta_0|^{-4}$, so this proves that $f_{n,p}(\theta)$ is a symbol of order $0$ as $\theta \to \infty$. 

Putting lemmas \ref{lemma:uniform_integrability}, \ref{lemma:uniform-bound-J}, \ref{lemma:uniform-bound-f} together, we deduce from equation \eqref{eq:integrating-by-parts} that for all $ N ,k>0$ and $\epsilon>0$, there is a constant $C_{ N ,k,\epsilon}>0$ such that, when $\Re s> d/2 + \epsilon$,
\begin{equation}\label{eq:estimate_remainder_non_stationary_2}
|(\mathrm{II})| \leq C_{ N ,k} e^{-\mathcal{T}(c_{p,q}) \Re s } (1+(\mathcal{T}(c_{p,q})+\log b_j)^+)^N   s^{-k}.
\end{equation}

	\subsection{Main result}

With the notations of lemma \ref{lemma:asymptotic_expansion_scattering}, for $c \in \pi^{ij}_1(M)$ with endpoints $p,q$ in $\widetilde{M}$, we define 
\begin{equation}\label{eq:def-a^n(g)}
a^n (c) := \exp\left\{\int_{c} V_0\right\} A_n(p,q).
\end{equation}
We also define
\begin{equation}\label{eq:def-T^0-T^hash}
\mathcal{T}^0_{ij}:= \inf \{ \mathcal{T}(c),\ c \in \pi^{ij}_1(M) \} \text{ and } \mathcal{T}^\#_{ij}= \min(-\log b_i b_j, \mathcal{T}^0_{ij}).
\end{equation}

Putting together lemmas \ref{lemma:first_estimate_phi}, \ref{lemma:asymptotic_expansion_scattering}, and equation \eqref{eq:estimate_remainder_non_stationary_2}, we get
\begin{theorem}\label{theorem:parametrix_phi}
For two cusps $Z_i$ and $Z_j$ not necessarily different, and for every $ N >0$, when $\Re s > \delta(\Gamma, V_0)$,
\begin{equation}\label{eq:parametrix_scattering_matrix}
\phi_{ij}(s) = \left(\frac{\pi}{s}\right)^{d/2} \sum_{[c] \in \pi_1^{ij}(M)} \sum_{n = 0}^{ N -1} \frac{1}{s^n} \frac{a^n(c)}{e^{s \mathcal{T}(c)}} + \frac{\mathcal{O}(1)}{s^N   e^{s \mathcal{T}^\#_{ij}} }.
\end{equation}
\end{theorem}

We proceed to give a parametrix for $\varphi$. When taking the determinant of the scattering matrix $\phi(s)$, we use the Leibniz formula
\begin{equation*}
\varphi(s) = \sum_{\sigma} \varepsilon(\sigma)\prod_{i=1}^\kappa \phi_{i, \sigma(i)}(s).
\end{equation*}
The sum is over the permutations $\sigma$ of $\ldbrack 1, \kappa \rdbrack$, and $\varepsilon(\sigma)$ is the signature of $\sigma$. The remainder will be bounded by terms of the form
\begin{equation*}
s^{\kappa /2 -  N } \exp\left\{-s \Big(\mathcal{T}_{1\sigma(1)}^\# + \dots + \mathcal{T}_{\kappa\sigma(\kappa)}^\#\Big) \right\}.
\end{equation*}
This one corresponds to the error of approximation for the product $\phi_{1\sigma(1)}(s)\dots \phi_{\kappa\sigma(\kappa)}(s)$ where $\sigma$ is a permutation of $\ldbrack 1 , \kappa \rdbrack$. Hence, we define
\begin{equation}
\mathcal{T}^\# = \min_{\sigma} \sum \mathcal{T}_{i\sigma(i)}^\#.
\end{equation}
It corresponds to the slowest decreasing  remainder term as $\Re s \to +\infty$. Recall the definition in \eqref{eq:def-SC}: the scattering cycles ($\mathcal{SC}$) are numbers of the form $T_1 + \dots + T_\kappa$, where $T_i \in \mathcal{ST}_{i \sigma(i)}$. We define $\mathcal{T}^0$ to be the smallest scattering cycle. It corresponds to the slowest decreasing term in the parametrix. By definition, $\mathcal{T}^\# \leq \mathcal{T}^0$.

\begin{remark} There are two cases. When $\mathcal{T}^\# < \mathcal{T}^0$, the error is bigger than the main term in the parametrix, for $\Re s$ too big with respect to $\Im s$. That occurs when the incoming plane waves from the cusps encounter variable curvature before they have travelled the shortest scattered geodesics.

When $\mathcal{T}^\# = \mathcal{T}^0$, the error term is always smaller than the main term in the parametrix. This means that the variations of the curvature happen not too close to the cusps. It is in particular the case when the curvature is constant.
\end{remark}

In any case, let $\lambda_\#= \exp \mathcal{T}^\#$. Also let $\lambda_0 < \lambda_1 < \dots < \lambda_k < \dots $ be the ordered elements of $\{ \exp \mathcal{T}, \ \mathcal{T} \in \mathcal{SC} \}$. We can now state the conclusion of this section:

\begin{theorem}\label{theorem:parametrix_varphi}
There are real coefficients $\{a^n_k\}_{k, n \geq 0}$ such that if
\begin{equation*}
L_n := \sum_{k \geq 0} \frac{a^n_k}{\lambda_k^s},
\end{equation*}
all the $L_n$'s converge in the half plane $\{\Re s > \delta(\Gamma, V_0)\}$. In that half plane, for all $ N \geq 0$,
\begin{equation*}
\varphi(s) = s^{-\kappa d/2}\left[ \sum_{n= 0}^N   s^{-n} L_n(s) + \frac{\mathcal{O}(1)}{s^{ N +1} \lambda_\#^s}\right].
\end{equation*}
\end{theorem}

\section{Dependence of the parametrix on the metric}\label{section:continuity-parametrix-metric}

This section is devoted to studying the regularity of the coefficients $a^n(c)$ with respect to the metric. We prove that they are continuous in the appropriate spaces in sections \ref{section:marked-sojourn-spectrum} and \ref{section:regularity_higher_coeff}. Then, we prove an openness property in $\mathscr{C}^\infty$ topology on metrics. While essential to the proof of theorem \ref{theorem:main}, this part is quite technical, and the impatient reader may skip directly to section \ref{section:Application}.

\subsection{The marked \emph{Sojourn Spectrum}}\label{section:marked-sojourn-spectrum}

As announced in section \ref{section:parabolic-points-scattered-geodesics}, we emphasize the dependence of objects on the metric from now on. In particular, when we write $[\overline{c}_g] \in \pi^{ij}_1(M)$, we mean that we take some class in $\pi^{ij}_1(M)$, and consider $\overline{c}_g$, the unique scattered geodesic for $g$ in that class. 

We denote by $T_g$ the application $T_g : [\overline{c}_g] \in \pi^{ij}_1(M) \mapsto \mathcal{T}(\overline{c}_g)$. We also write $a^n(g, [c])$ instead of just $a^n(c)$.

In what follows, we are interested in the regularity and openness properties of $\varphi$. It is obtained as the determinant of $\phi(s)$. Since the determinant is a polynomial expression, it is certainly smooth and open with respect to $\phi$. As a consequence, it suffices to study the regularity of each $\phi_{ij}$ independently, and the openess properties of $\phi(s)$ instead of $\varphi$. 

The following lemma is classical:
\begin{lemma}\label{lemma:regularity_geodesic_flow}
Let $(g_{\epsilon})_{\epsilon \in \R}$ be a family of $\mathscr{C}^\infty$ cusp metrics on $M$, so that their curvature varies in a compact set independent from $\epsilon$. Suppose additionally that $g_\epsilon$ is $C^{2+k}$ on $\R \times M$ for $k\geq 0$. Then we say that $g_\epsilon$ is a $C^{2+k}$ family of metrics.

In that case, the geodesic flow $\varphi_t^{g_\epsilon}$ is $C^{1+k}$ on $\R \times M$ for $k\geq 0$.
\end{lemma}

This is the direct consequence of
\begin{lemma}
Let $f$ be a $C^k$ function on $\R \times U \subset \R^m$ where $k\geq 1$ and $U$ is an open set. Then the flow associated to 
\begin{equation*}
\dot{x} = f(t,x).
\end{equation*}
is $C^k$.
\end{lemma}

The proof of this can be found in any introduction to dynamical systems. Now, we can prove that both the marked set of scattered geodesics and the marked Sojourn Spectrum are continuous along a perturbation of the metric that is at least $C^1$ in the $C^2$ topology on metrics.
\begin{lemma}
Let $g_\epsilon$ be a $C^{2+k}$ ($k\geq 0$) family of cusp metrics on $M$. Let $\overline{c}$ be a scattered geodesic for $g= g_0$. Then there is a $C^{1+k}$ family of curves $\overline{c}_\epsilon$ on $M$ such that $\overline{c}_\epsilon$ is a scattered geodesic for $g_\epsilon$. In particular, this proves that $g \mapsto \overline{c}_g$ (given a class in $\pi^{ij}_1(M)$) $g \mapsto T_g$ are $C^{1+k}$ in $C^{2+k}$ topology on $g$, when $k\geq 0$.
\end{lemma}

\begin{proof}
Let us assume that $\overline{c}$ enters $M$ in $Z_i$ and escapes in $Z_j$. We can assume that the variations of the curvature of $g_\epsilon$ always take place below $y= y_0$. Let $x_0$ (resp. $x_1$) be the point where $\overline{c}$ intersects the projected horosphere $H_i$ (resp. $H_j$) at height $y_0$ in $Z_i$ (resp. $Z_j$), entering (resp. leaving) the compact part. For $x\in H_i$ and $\epsilon$ close to $0$, we can consider the following curve: $c_{x,\epsilon}$ is the geodesic for $g_\epsilon$, that passes through $x$, and is directed by $-\partial_y$ at $x$. We have $\overline{c} = c_{x_0, 0}$. For $(x,\epsilon)$ close enough to $(x_0, 0)$, $c_{x,\epsilon}$ intersects the projected horosphere $H_j$, for a time close to $\mathcal{T}(\gamma)+2\log y_0$. We let $x'(x,\epsilon)$ be that point of intersection, and $v(x,\epsilon)$ the vector $c_{x,\epsilon}'$ at $x'(x,\epsilon)$.

Now, by the lemma above, $v(x,\epsilon)$ is $C^{1+k}$, and by the Local Inversion Theorem, there is a unique $\epsilon\mapsto x(\epsilon)$, $C^{1+k}$, such that $v(x(\epsilon),\epsilon)$ is the vertical for all $\epsilon$ sufficiently close to $0$, as soon as $\partial_{x}v(0,0)$ is invertible. But the fact that it is invertible is a direct consequence of the non-degeneracy of the phase function shown in lemma \ref{lemma:non-degeneracy}.
\end{proof}

Let
\begin{equation*}
L^{ij}_n(s) = \sum_{[\overline{c}] \in \pi_1^{ij}(M)} \frac{a^n(g,[\overline{c}])}{e^{s T_g([\overline{c}])}}
\end{equation*}

\begin{lemma}\label{lemma:regularity_first_asymptotics}
Let $g_\epsilon$ be a $C^{2+k}$ family of metrics, $k\geq 0$. Then, as a formal series, $L^{ij}_0$ depends on $\epsilon$ in a $C^k$ fashion. In particular, the series $L_0$ giving the first asymptotics for $\varphi$ at high frequencies, also depends in a $C^k$ fashion on $\epsilon$.
\end{lemma}

\begin{proof}
We only have to prove that $a^0(g_\epsilon, [\overline{c}])$ depends on $\epsilon$ in a $C^k$ fashion. Since $V_0$ is only a H\"older function, it is easier to study the regularity of $a^0$ with the original expression \eqref{eq:formula-C_0}. That is, we have to study $d^2_\theta G_p(0_\theta)$ and $\tilde{J}_p(0_\theta)$. 

First, consider $\tilde{J}_p$. It is a function of the jacobian of the flow $\varphi^p_t$ along $\overline{c}$. Since $\varphi^p_t$ is just some restriction of $\varphi_t$, $\varphi^p_t$ is $C^{1+k}$ on $\R \times M$. We also have that $\overline{c}$ is $C^{1+k}$, so that $\tilde{J}_p$ is $C^k$ on $\epsilon$.

For $d^2_\theta G_p(0_\theta)$, consider that it is obtained as the first variation of $\nabla G_p$ along the horocycle $H(q,b_j)$. But this means that $d^2_\theta G_p(0_\theta)$ is again obtained directly in terms of $d\varphi^p_t$ along $\overline{c}$ and this ends the proof.
\end{proof}

Here already, we see that the first order behaviour of the scattering determinant at high frequency ($\Re s$ bounded and $\Im s \to \pm \infty$) depends continuously on $g$ in $C^2$ topology. The next section is devoted to studying this regularity for other terms.

\subsection{Higher order coefficients of the parametrix}\label{section:regularity_higher_coeff}

Now, we are interested in the regularity of $a^n(g,[c])$ for $n\geq 1$. 
\begin{lemma}
Let $g_\epsilon$ be a $C^{2+k}$ family of cusp metrics on $M$. Then the coefficients $a^n(g,[c])$ depend in a $C^{k-2n}$ fashion on $\epsilon$, as soon as $k\geq 2n$.
\end{lemma}

In the following proof, we fix two points $p,q$ on the boundary. Most of the functions that appear depend on $p$ and $q$, but to simplify notations, we omit that dependence. We do not fix $n$, but $k$ will always be assumed to be greater or equal to $2n$.

\begin{proof}
Let us start by a discussion of classical stationary phase in $\R^d$. Let $\sigma$ be smooth compactly supported function on $\R^d$. Then, as $|s| \to + \infty$ with $\Re s > 0$,
\begin{equation*}
\int_{\R^d} e^{-s x^2} \sigma(x) dx = \pi^{d/2} s^{-d/2} \left[\sigma(0) + \frac{1}{4s} \Delta \sigma(0) + \dots + \frac{1}{n! 4^n s^n} \Delta^n \sigma(0) + \mathcal{O}(s^{-n-1}) \right]
\end{equation*}
where $\Delta^n \sigma = \Delta \dots \Delta \sigma$. If $G$ is a non-degenerate phase function around $0$, we find $\Psi$ smooth around $0$ such that $G\circ \Psi(x) = x^2$, by Morse theory. Then, if $\sigma$ is still compactly supported but has an expansion $\sigma(s,x) \sim \sigma_0(x) + \sigma_1(x)/s + \dots + \sigma_n(x)/s^k + \dots$, we find
\begin{equation*}
\int_{\R^d} e^{-s G(x)} \sigma(s,x) dx \sim \left(\frac{\pi}{s}\right)^{d/2}\left[ \sum_{n=0}^\infty \sum_{l=0}^ \infty s^{-n-l} \frac{1}{l ! 4^l}\Delta^l(\sigma_n\circ\Psi . \Jac(\Psi))(0) \right].
\end{equation*}
In other words, the coefficient of $\pi^{d/2} s^{-d/2 - n}$ is
\begin{equation}\label{eq:coefficient-stationary-phase}
\sum_{l=0}^n \frac{1}{l ! 4^l}\Delta^l (\sigma_{n-l}\circ\Psi . \Jac(\Psi))
\end{equation}
It is a well known fact that the Morse chart $\Psi$ is not uniquely defined. However, from the computations above, the operators
\begin{equation*}
\sigma \mapsto \Delta^j (\sigma\circ \Psi .\Jac \Psi)(0)
\end{equation*}
do not depend on the choice of $\Psi$, but only on $G$. By writing the condition $G \circ \Psi = x^2$, one can see that $d\Psi^T(0) .d\Psi(0)= d^2 G(0)$. This determines $d\Psi(0)$. The higher order derivatives of $\Psi$ are undetermined, but one can see that they can be chosen recursively, so that $d^k \Psi(0)$ only depends on the $(k+1)$-jet of $G$ at $0$.

We apply the discussion above to $a^n(g, [c])$. The notations are coherent with section \ref{sec:main_asymptotics} and equation \eqref{eq:stationary_1} if one set $G= G_p$ and $\sigma= \tilde{J}_p f^N_p$. We decompose $a^n$, following equation \eqref{eq:coefficient-stationary-phase}. We find that $a^n$ depends on derivatives of $\tilde{J}$, $f$ and $G$. We expand each summand in the decomposition, using the Leibniz rule. Then we gather the terms involving the highest order derivatives of the metric. They are
\begin{equation}\label{eq:higher_order_a^n_1}
\frac{1}{n ! 4^n} \left(\Jac \Psi(0) \Delta_\theta^n \tilde{J}(0) + \tilde{J}(0)\Delta_\theta^n \Jac \Psi(0) \right) + \tilde{J}(0)\Jac \Psi(0) \sum_{l=1}^n \frac{1}{(n-l) ! 4^{n-l}}\Delta_\theta^{n-l}(f_l\circ \Psi).
\end{equation}

In the proof of \ref{lemma:regularity_first_asymptotics}, we saw that along a $C^{2+k}$ perturbation, $\tilde{J}$ is $C^k$, so that $\Delta_\theta^n (\tilde{J}\circ \Psi)$ is $C^{k-2n}$. By the same argument, we find that the $(2n+2)$-jet of $G$ at $0_\theta$ is a $C^{k-2n}$ function of $\epsilon$, so that $\Delta_\theta^n \Jac \Psi$ is also $C^{k-2n}$.

\begin{remark*}
One can check that the numbers $\Delta^j \Jac \Psi(0)$ are, up to universal constants, the Taylor coefficients in the expansion of the function $\vol(G \leq r^2)$. 
\end{remark*}

Now, we deal with the $f_n$'s. From the definition of $Q$ in \eqref{eq:def-Q} and $f_n$ in \eqref{eq:def-f_n}, we can prove by induction that for $n\geq 1$,
\begin{equation}\label{eq:closed_form_f_n}
f_n = \frac{1}{2^n} \int_{-\infty}^0 dt_n \left[\int_{t_n}^0 dt_{n-1} \cdots \int_{t_2}^0 dt_1 Q_{t_{n-1}} \cdots Q_{t_1} Q_0 f_0 \right] \circ \varphi^p_{t_n}
\end{equation}
where $Q_t$ is defined by $Q (f\circ \varphi^p_t) = (Q_t f)\circ \varphi^p_t$. Since $F$ is essentially a jacobian of $\varphi^p_t$, it is $C^k$ on $\R \times M$ along a $C^{2+k}$ perturbation. From the formula \eqref{eq:closed_form_f_n}, we deduce that $f_n$ is $C^{k-2n}$ along a $C^{2+k}$ perturbation when $k\geq 2n$, and this ends the proof.

\end{proof}

\subsection{Openness in $C^\infty$ topology}

To find that the coefficients of the parametrix are open, we are going to adopt a different point of view from the previous section. We let $a^{-1}(g,[c]) = T_g ([c])$. We aim to prove the following:
\begin{lemma}\label{lemma:openness-coefficients}
Let $[c_1], \dots, [c_N]$ be distinct elements of $\pi_1^{i_1 j_1}(M), \dots, \pi_1^{i_N j_N}(M)$, and take indices $n_1, \dots, n_N$. Then the application 
\begin{equation*}
a_N : g \mapsto \left( (a^{-1}, a^0, \dots, a^{n_1})(g,[c_1]), \dots, (a^{-1}, a^0, \dots, a^{n_N})(g,[c_N]) \right)\in \R^{\sum n_i+1}
\end{equation*}
is open in $C^\infty$ topology on $g$.
\end{lemma}

\begin{proof}
First, observe that it suffices to prove that the differential of $a_N$ is surjective. Indeed, we can then use the local inversion theorem to prove the openness property.

For each class $[c_i]$, we will compute the variation of $(a^0, \dots, a^{n_i})$ along a well chosen smooth family of cusp metrics $g_\epsilon$. We will find that this variation is a linear form in a jet of the variation $\partial_\epsilon g_\epsilon$ along $\overline{c}_i$. From the properties of this linear form, we will find that there are functions with arbitrary compact support on which it does not vanish. This will prove the lemma for $N=1$.

For the case when $N\geq 1$, observe that since the $[c_i]$ are distinct, the $\overline{c}_i$ are also distinct. Then, it suffices to observe that we can take a finite number of small open sets $U_i$ such that $U_i \cap U_j = \emptyset$ when $i\neq j$, and $U_i \cap \overline{c}_i \neq \emptyset$. Then we can perturb in each open set independently, and in this way, we see that the differential of $a_N$ is surjective, and this ends the proof.

\begin{remark} There might seem to be a difficulty when the geodesic $\overline{c}$ has a self intersection, because at the point of intersection, we have less liberty on the perturbations we can make. However, we will always choose to perturb away from those intersection points. 
\end{remark}

As we have reduced the proof to the case $N=1$, let $[c] \in \pi_1^{ij}(M)$. 

\textbf{First case, $\mathbf{n=0}$.}
\begin{lemma}
Let $g_\epsilon$ be a $C^\infty$ family of cusp metrics. Then
\begin{equation*}
\partial_\epsilon T([c]) = \frac{1}{2}\int (\partial_\epsilon g)_0(\overline{c}'_0(t), \overline{c}'_0(t))dt.
\end{equation*}
In particular, if $U$ is an open set that intersects $\overline{c}_0$, one can find a perturbation of the metric, supported in $U$, along which $\partial_\epsilon T([c])\neq 0$.
\end{lemma}

\begin{proof}
From the arguments above, we can construct a variation $c_\epsilon$ of $\overline{c}_0$ such that each $\gamma_\epsilon$ is an \emph{unparametrized} geodesic for $g_\epsilon$. We can assume that for $t$ negative (resp. positive) enough, $y_i(c_\epsilon) = y_i(\overline{c}_0)$ (resp. $y_j(c_\epsilon) = y_j(\overline{c}_0)$). Then in local coordinates
\begin{equation*}
\partial_\epsilon T([c]) = \frac{1}{2}\int (\partial_\epsilon g)_0(\overline{c}'_0(t), \overline{c}'_0(t)) + 2g_0(\overline{c}'_0(t), \partial_\epsilon c'_0(t)) dt
\end{equation*}
In the RHS, the second term, we can interpret as the 1st order variation of the length of the curve $c_\epsilon$ for $g_0$. Since $\overline{c}_0$ is a geodesic, this has to be zero
\end{proof}

\begin{lemma}
The logarithmic differential
\begin{equation*}
d_g \log a^0(g,[c])
\end{equation*}
is non-degenerate on the set of symmetric $2$-tensors $h$ on $M$ such that $h$ and $dh$ vanish at $\overline{c}$. This proves the property for $n=0$, because if $h$ is such a $2$-tensor, along the perturbation $g+\epsilon h$, the curve $\overline{c}$ is always a scattered geodesic of constant sojourn time, and $d_g a^{-1}(g, [c]).h = 0$.
\end{lemma}

\begin{proof}
since the curve $\overline{c}$ does not depend on the metric in our context, it is reasonable to use the method of variation of parameters. Let $g_\epsilon = g + \epsilon h$, and consider $\J_{u, \epsilon} = \J_u + \epsilon \tilde{\J}_u + o(\epsilon)$ the unstable Jacobi field for $g_\epsilon$ along $\overline{c}$ defined in page \pageref{page:def-J_u}. We also use $\J_s$ as defined in the same page. We can write 
\begin{equation*}
\tilde{\J}_u = \J_u \widetilde{A}(t) + \J_s \widetilde{B}(t)
\end{equation*}
and we find the equations for $\widetilde{A}$ and $\tilde{B}$:
\begin{align*}
\J_u \widetilde{A}' + \J_s \widetilde{B}' &=0 \\
\J_u' \widetilde{A}' + \J_s' \widetilde{B}' &= - (d_g K.h) \J_u.
\end{align*}
Recall from the arguments in page \pageref{page:def-J_u} that 
\begin{equation*}
a^0(g_\epsilon,[\overline{c}])^2 = \lim_{t\to + \infty} \frac{e^{d.G_p(0_\theta)}}{\det \J_u \J_s}= a^0(g_\epsilon, [\overline{c}])^2 \frac{1}{\det \mathbf{1} + \epsilon \widetilde{A}(+\infty)} .
\end{equation*}
where $\widetilde{A}(+\infty)$ is the limit of $\widetilde{A}$ when $t\to \infty$. Hence
\begin{equation*}
\frac{d}{d\epsilon} \log a^0(g_\epsilon, [\overline{c}]) = - \frac{1}{2} \Tr \widetilde{A}(+\infty).
\end{equation*}
We find
\begin{equation*}
\widetilde{A}(+\infty) = \int_{\R} \J_u^{-1}(\Ss - \U)^{-1}(d_g K(t).h)\J_u(t) dt
\end{equation*}
and conclude
\begin{equation}\label{eq:da^0_g}
\frac{d}{d\epsilon} \log a^0(g_\epsilon, [\overline{c}]) = \frac{1}{2} \int_\R \Tr \left\{(\U - \Ss)^{-1}(d_g K(t).h)\right\} dt
\end{equation}
When the curvature of $g$ is constant along $\overline{c}$, one may observe that this gives a particularly simple expression. Now we prove that the differential $h\mapsto d_g K.h$ is surjective on the set of symmetric matrices along the geodesic $\overline{c}$. We consider Fermi coordinates along $\overline{c}$. That is, the coordinate chart given by
\begin{equation*}
(x_1; x') \mapsto \exp_{\overline{c}(x_1)}\left\{ x' \right\} \in  N .
\end{equation*}

\begin{remark} When $\overline{c}$ has self-intersection, this chart is not injective. However, we can assume that $h$  vanishes around such points of intersection, and the computations below remain valid.
\end{remark}

In those coordinates, $g-\mathbf{1}$ and $dg$ vanish along the geodesic, which is $\overline{c}\simeq \{ x'=0 \}$. We deduce that the Christoffel coefficients $\Gamma^k_{ij}$ also vanish to second order on $\overline{c}$. Now we recall from \cite{Paulin-geom-riem} two useful formulae.
\begin{align}
\tag{p.210} \Gamma_{ij}^k &= \frac{1}{2} \sum_l g^{lk}(\partial_i g_{jl} + \partial_j g_{li} - \partial_l g_{ij}) \\
\tag{p.211} R(\partial_i, \partial_j) \partial_k &= \sum_l \left\{ \partial_i \Gamma_{jk}^l - \partial_j \Gamma_{ik}^l + \sum_m \Gamma_{jk}^m \Gamma_{im}^l - \Gamma_{ik}^m \Gamma_{jm}^l \right\} \partial_l.
\end{align}
Whence we deduce that on $\overline{c}$
\begin{align}
R_g (\partial_i, \partial_1) \partial_1 &= \sum_{l=1}^{d+1} \partial_i \Gamma^l_{11}\partial_l, \notag \\
								& = -\frac{1}{2} \partial_i \partial_k g_{11} \partial_k.
\end{align}
We see that $2 K(t) = - d^2 g_{11} (t)$, so that $d_g K.h = -1/2 d^2 h_{11}$, and this is certainly surjective onto the set of smooth functions along the geodesic valued in symmetric matrices. In particular, the RHS of \eqref{eq:da^0_g} defines a non-degenerate linear functional on the set of compactly supported $2$-symmetric tensors along $\overline{c}$. 
\end{proof}

This ends the case $n=0$.

\textbf{General case, $\mathbf{n\geq 1}$.} We introduce a special coordinate chart on $ \widetilde{M} $:
\begin{equation*}
\varsigma_g : (x,t) \in H(p,b_i) \times \R \mapsto \varphi^q_t(x).
\end{equation*}
Since $H(p,b_i) \simeq \R^d$, we are now working in $\R^{d+1}$. In the coordinates $\varsigma_g$, the flow has a very simple expression: $\varphi^p_t(x,s) = (x, s+t)$. The metric also:
\begin{equation}\label{eq:decomp-varsigma-metric}
\varsigma^\ast g = \tilde{g}(x,t ; dx) + dt^2;
\end{equation}
the jacobian
\begin{equation}\label{eq:expression_simplified_jacobian}
Jac(\varphi^q_t)(x,s) = \sqrt{\frac{\det \tilde{g}(x,s+t)}{\det \tilde{g}(x,s)}},
\end{equation}
and
\begin{equation}\label{eq:expression_simplified_F}
F(x,s) = \frac{1}{4}\left( \log \det \tilde{g}(x,s) - \log \det \tilde{g}(x,0) + 2 s d \right).
\end{equation}
We can find that $\tilde{g}(x,0)$ actually does not depend on $x$. We also have that $G_p(x,s) = s - \log b_i$.

\textbf{Idea of proof.} If we perturb $\tilde{g}$ by a symmetric 2-tensor $h$ on the slices, we obtain a new metric $\tilde{g}_h$ on $\R^{d+1}$. We can obtain a metric $g_h^1$ on $\widetilde{M}$, pushing forward by $\varsigma_g$. If the support $\Omega$ of this perturbation is small enough that $\gamma \Omega \cap \Omega = \emptyset$ for all $\gamma \neq 1$, we can periodize the perturbation to obtain a metric on $M$, or equivalently, a metric $g_h$ invariant by $\Gamma$ on $\widetilde{M}$. 

The metric $g_h$, seen in the chart $\varsigma_g$, does not have the nice decomposition \eqref{eq:decomp-varsigma-metric} anymore. However, that decomposition still holds in the complement of $\Upsilon:=\bigcup_{\gamma \notin \Gamma_p} \gamma \Omega$. If $\Omega$ was well chosen, this includes a neighbourhood $\Omega'$ of the geodesic $\overline{c}$ that we wanted to perturb.

\begin{figure}[h]
\centering
\def\svgwidth{0.35\linewidth}
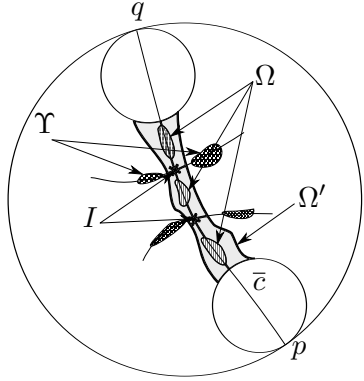
\caption{\label{fig:perturbation-global-situation}Global situation.}
\end{figure}

The condition for $\Omega$ to be appropriate is that the projection $\widetilde{M} \to M$ is injective on $\Omega$, and that $\gamma \Omega$ does intersect the lift $[p,q]$ of $\overline{c}$. For this, it suffices that $\Omega$ is not too close to the points $I$ in $[p,q]$ that project to self-intersection points of $\overline{c}$. See figure \ref{fig:perturbation-global-situation}.

There is a last difficulty. The point $0_\theta$ is represented by $(0, t_1)$ with $t_1 =  T_g(g, [c]) + \log b_i b_j$ --- see the paragraph after equation \eqref{eq:symbol_R^d}. It is possible that $t_1 < 0$. In that case, the geodesic $\overline{c}$ only encounters constant curvature. To perturb the coefficients, we will need to create variable curvature along the geodesic. In particular, that will change the values of $b_i$ and $b_j$. 

To overcome this difficulty, we proceed in the following way. Instead of integrating along the projected horosphere at height $b_j$ in the cusp $Z_j$, we integrate on the projected horosphere at height $b_j^\star \geq b_j$ in the proof of theorem \ref{theorem:parametrix_phi}. We do it so that for all $[\overline{c}]\in \pi^{ij}_1(M)$, $\mathcal{T}(\overline{c}) + \log b_i^\star b_j^\star > 0$ (for all $i, j$\dots). Since the marked sojourn time function is proper, only a finite number of scattered geodesics intervene here. All quantities that depended on $b_i$, $b_j$ before will now receive a $\star$ when we replace $b_i$ by $b_i^\star$.

%\begin{remark} We will perturb the metric $g$ to $g_h$ in this chart, and then obtain the metric $(\varsigma_g^\star)_\ast g_h$ on $\widetilde{M}$. However, we have to be careful that such a perturbation is invariant by $\Gamma$, and so corresponds to a metric on $M$ also. As we will only need to specify some jet of $g_h$ at $x=0$, we can always assume that the support of $h$ is very close to $\{x=0\}$, and compact. Let $K$ be such a compact set. 
%
%Let $I$ be the set of points of $\overline{c}$ (in $\widetilde{M}$) that correspond to self intersection points in $M$. Then, provided $K$ is small enough and $K \cap I = \emptyset$, $\pi_{|K} : K \to M$, the restriction of $\pi : \widetilde{M} \to M$ to $K$, is injective, and we can project $g_h$ to $M$. As a consequence, we will look for perturbations of $g$ that are supported away from $I$.
%\end{remark}

\begin{figure}[h]
\centering
\def\svgwidth{0.35\linewidth}
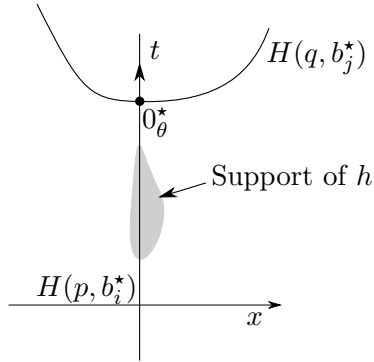
\caption{\label{fig:coordinates-varsigma}A close up.}
\end{figure}

Coming back to perturbing coefficients, the point $0_\theta^\star$ is represented by $(0, t_1^\star)$ with $t_1^\star = T_g(g, [c]) + \log b_i^\star b_j^\star$. If the perturbation $h$ is compactly supported in $\{ 0 < t < t_1^\star\}$, the expression of $g_h$, $H(q, b_j^\star)$, $\Delta_\theta^\star$, $\tilde{J}$, $G_p$ will not depend on $h$ in the chart $\varsigma_g^\star$, around $0_\theta^\star$. In particular, $a^{-1}$ and $a^0$ are always constant along such a perturbation.

Now, we assume that the change in the slices is $\epsilon h$ where $h$ is a 2-symmetric tensor such that $h, dh, \dots, d^{2n-1}h$ vanish along $\{ x = 0\}$ in the chart $\varsigma_g^\star$, and $h$ is supported for $0 < T < s < T' < t_1^\star$. Let $g_\epsilon := g_{\epsilon h}$.
\begin{lemma}
Under such a perturbation, $a^{-1}, a^0, \dots, a^{n-1}$ are constant.
\end{lemma}

\begin{proof}
As we saw in section \ref{section:regularity_higher_coeff}, the coefficient $a^k$ is computed from the $2k - 2\ell$ jet of $f_\ell$, $\ell=1, \dots, k$, at $0_\theta$, and also the $2k$ jet of $\tilde{J}$ and $G$. Those computations are done with $\Delta_\theta$, which in our chart $\varsigma_g$ has a complicated expression. However, since we are not perturbing the metric around $0_\theta$, the coefficients of $\Delta_\theta$ do not change under the perturbation. From equation \eqref{eq:expression_simplified_jacobian}, and the expression for $G$ in this chart, we see that the contribution of $\tilde{J}$ and $G$ to $a^k$ will not change under perturbation (independently from the order of cancellation of $h$).

We are left to prove that the $2k-2\ell$ jet of $f_\ell$ at $0_\theta$ does not change for $0 \leq \ell \leq k \leq n-1$. From formula \eqref{eq:expression_simplified_F}, we see that the $2n-1$ jet of $F$ along $\overline{c}$ will not change along the perturbation. From equation \eqref{eq:closed_form_f_n}, we see that the $m$ jet of $f_\ell$ at $0_\theta$ depends on the $m + 2\ell$ jet of $F$, and the $m+ 2\ell - 1 $ jet of $g$ --- recall that the coefficients in the Laplacian $\Delta$ depend on $dg$, and the coefficients in $\nabla $ depend on $g$. Taking this for $m=2k - 2\ell$ and $\ell \leq k \leq n-1$, we find that the $2k-2\ell$ jet of $f_\ell$ can be computed with only the $2n-2$ jet of $g$ at $\overline{c}$, and this proves the lemma.
\end{proof}

From the proof of the lemma above, we see that in $a^n$, the only change will come from the change in the derivatives of order $2n-2k$ of $f_k$, and more precisely, the parts of these variations that come from the change in $2n$ derivatives of $F$, in the $x$ direction. As a consequence, we can do all the forecoming computations as if the differential operators appearing had constant coefficients, and replace $\Delta$ (resp. $\Delta_\theta( \cdot \circ \Psi)\circ \Psi^{-1}$) by
\begin{equation*}
\tilde{\Delta} := \tilde{g}_{ij}(s)\partial_i \partial_j \quad (\text{resp. } \tilde{\Delta}_\theta := c_{ij}\partial_i \partial_j)
\end{equation*}
where the matrices $(\tilde{g}_{ij})(t)$ and $(c_{ij})$ are symmetric, positive matrices. Recall the metric $g$ has the expression
\begin{equation*}
g_{(x,t)}(dx,dt) = \tilde{g}_{x,t}(dx) + dt^2
\end{equation*}
and $(\tilde{g}_{ij})(t)$ is the value of $\tilde{g}_{0,t}^{-1}$, but this fact will not be used later. Recall that the operator $Q_t$ was defined by $(Q_t f)\circ \varphi^q_t = Q (f\circ \varphi^q_t)$. We define $\tilde{\Delta}_t$ in the same way. An easy computation shows that $\tilde{\Delta}_t = \sum g_{ij}(s-t)\partial_i \partial_j$. 

Now, we use formula \eqref{eq:higher_order_a^n_1}. We only keep the terms that vary under the perturbation $g_\epsilon$. This yields
\begin{equation*}
a^n(g_\epsilon, [\overline{c}]) - a^n(g, [\overline{c}]) = a^0(g,[\overline{c}])\sum_{l=1}^n \frac{1}{(n-l)! 4^{n-l}} \tilde{\Delta}_\theta^{n-l} \left\{(f_l)_\epsilon - f_l\right\}.
\end{equation*}
Next we use equation \eqref{eq:closed_form_f_n}, leaving out the constant terms again. We find:
\begin{equation*}
\frac{a^n(g_\epsilon, [\overline{c}]) - a^n(g, [\overline{c}])}{a^0(g,[\overline{c}])} =\sum_{l=1}^n \frac{1}{(n-l)! 4^n 2^{-l}} \int_{\mathfrak{S}} \hspace{-5pt} dt_l  \dots \  dt_1 \ \tilde{\Delta}_\theta^{n-l}\tilde{\Delta}_{t_{l-1}} \dots \tilde{\Delta}_{t_1} \tilde{\Delta} \left\{F_\epsilon - F\right\} (0, t_1^\star + t_l)
\end{equation*}
Here, $\mathfrak{S}$ is the simplex $\{ -\infty < t_l \leq t_{l-1} \leq \dots \leq t_1 \leq 0 \}$. Let $t_0=0$. Now, since $4 d_x F = \Tr g^{-1}dg$, each integrand in the above formula reduces to
\begin{equation}\label{eq:da^n_h}
\frac{\epsilon}{4} \Tr\left\{ g^{-1}(t_1^\star + t_l)\sum_{\{ (i_m, j_m) \}} \prod_{m=l}^{n-1} c_{i_{m} j_{m}}\prod_{m=0}^{l-1} \tilde{g}_{i_m j_m}(t_1^\star + t_l - t_m) \left(\prod_{m=1}^n \partial_{i_m}\partial_{j_m}\right) h(t_1^\star + t_l)  \right\}.
\end{equation}
It is still not clear why such a formula would lead to a non-degenerate differential. However, let us assume that $h$ has the following form in a neighbourhood of $\{ x=0\}$
\begin{equation*}
h(x,s, dx) = \lambda(s)dx^2 \sum_{|\alpha| = 2n} u^\alpha x^\alpha + o(|x|^{2n})
\end{equation*}
where $u^\alpha = u_{\alpha_1} \dots u_{\alpha_{2n}}$, and likewise for $x^\alpha$. We take $u$ a constant vector in $\R^d$, and $\lambda(s)$ a smooth function, supported in $]0, t_1^\star[$. Formula \eqref{eq:da^n_h} becomes
\begin{equation*}
(c(u,u))^{n-l} \Tr g^{-1}(t_1^\star + t_l)\left\{\lambda(t_1^\star + t_l) \prod_{m=0}^{l-1}  \tilde{g}(t_1^\star + t_l - t_m)(u,u) \right\} .
\end{equation*}
Observe that these are nonnegative numbers. From those computations, we see that
\begin{equation*}
\frac{a^n(g_\epsilon, [\overline{c}]) - a^n(g, [\overline{c}])}{a^0(g,[\overline{c}])} = \epsilon \int_0^{t_1^\star}  \lambda(t) H(t)dt
\end{equation*}
where $H(t)$ is a function that does not vanish. This ends the proof for $n\geq 1$.
\end{proof}

\section{Applications}\label{section:Application}

We use simple Complex Analysis to locate zones without zeroes for $\varphi$. We also give some explicit examples corresponding to part (I) and (III) of the main theorem.
 
\subsection{Complex Analysis and Dirichlet Series}

Let $\lambda_0 < \lambda_1 < \dots < \lambda_k < \dots$ be positive real numbers. For $\delta>0$, we let $\mathscr{D}(\delta, \lambda)$ be the set of Dirichlet series $L(s)$ whose abscissa of absolute convergence is $\leq \delta$, and
\begin{equation*}
L(s) = \sum_{k=0}^\infty \frac{a_k}{\lambda_k^s}.
\end{equation*}
We let $\mathscr{D}^k(\delta, \lambda)$ be the set of $L\in \mathscr{D}(\delta, \lambda)$ such that $a_0 = \dots = a_{k-1}= 0$ and $a_k \neq 0$. For $0 < \lambda_\# \leq \lambda_0$, also consider $\mathscr{D}(\delta, \lambda, \lambda_\#)$ the set of holomorphic functions $f$ on $\{ \Re s > \delta\}$ such that there are $L_n \in \mathscr{D}(\delta, \lambda)$ with, for all $n\geq 0$
\begin{equation*}
f(s)=  L_0(s) + \frac{1}{s} L_1(s) + \dots + \frac{1}{s^n}L_n(s) + \mathcal{O}\left( \frac{1}{s^{n+1}\lambda_\#^s} \right).
\end{equation*}
We will denote $(a^n_k)$ the coefficients of $L_n$. By taking notations coherent with the rest of the article, we have $s^{\kappa d/2}\varphi \in \mathscr{D}(\delta_g, \lambda, \lambda_\#)$. For $\delta' > \delta$ and $C>0$, let
\begin{equation*}
\Omega_{\delta', C} := \left\{s\in \C \quad \Re s > \delta' \quad \Re s \leq C \log |\Im s| \right\}.
\end{equation*}

\begin{lemma}\label{lemma:zeroes_Dirichlet_function_1}
Let $f\in \mathscr{D}(\delta, \lambda, \lambda_\#)$ such that $L_0 \in \mathscr{D}^0(\delta, \lambda)$. Then there is a $\delta'>\delta$ such that for any constant $C>0$, $f$ has a finite number of zeroes in $\Omega_{\delta',C}$.

In the special case where $\lambda_\# = \lambda_0$, we can take $\delta' >0$ such that $f$ has \emph{no} zeroes in
\begin{equation*}
\left\{ s\in \C \quad \Re s > \delta' \right\}.
\end{equation*}
\end{lemma}

\begin{proof}
We can write
\begin{equation*}
L_0(s) = \frac{a^0_0}{\lambda_0^s} + \underset{:=\tilde{L}_0(s)}{\underbrace{\sum_{k=1}^\infty \frac{a^0_k}{\lambda_k^s}}}
\end{equation*}
There is a $\delta' > \delta$ such that whenever $\Re s > \delta'$,
\begin{equation*}
|\tilde{L}_0(s)| \leq \frac{1}{3} \left| \frac{a^0_0}{\lambda_0^s}\right|
\end{equation*}
Take $N>0$. Then for $|s|$ big enough --- say $|s| > C_N$ --- and for $\Re s > \delta'$, 
\begin{equation*}
\frac{1}{|s|}| L_1(s) | + \dots + \frac{1}{|s|^{N-1}}|L_{N-1}(s)| \leq \frac{1}{3} \left| \frac{a^0_0}{\lambda_0^s}\right|.
\end{equation*}
We also find that
\begin{equation*}
  \Re s \leq \frac{N}{\log(\lambda_0/\lambda_\#)} \log |\Im s| + \mathcal{O}(1) \text{ and } |s| > C_N \implies \left| \frac{C'_N}{s^N\lambda_\#^s}\right| < \frac{1}{3}  \left| \frac{a^0_0}{\lambda_0^s}\right|.
\end{equation*}
When $\lambda_0 = \lambda_\#$, the condition on $\Re s$ is void. When $\lambda_0 > \lambda_\#$, by taking $N \sim C \log(\lambda_0 /\lambda_\#)$, we find that the zeroes of $f$ in the region described in the lemma are actually in a bounded region of the plane. Since $f$ is holomorphic, they have to be in finite number.
\end{proof}

We give another lemma:
\begin{lemma}\label{lemma:zeroes_Dirichlet_function_2}
Let $f\in \mathscr{D}(\delta, \lambda, \lambda_\#)$ be such that $L_0 \in \mathscr{D}^1(\delta, \lambda)$ and $L_1 \in \mathscr{D}^0(\delta, \lambda)$. Let 
\begin{equation*}
\tilde{f}(s) := \frac{a^0_1}{\lambda_1^s} + \frac{a^1_0}{s \lambda_0^s}
\end{equation*}
There is a $\delta' > \delta$ such that for any constant $C>0$, there is a mapping $W$ from the zeroes of $\tilde{f}$ in $\Omega_{\delta',C}$ to the zeroes of $f$ in $\Omega_{\delta',C}$, that only misses a finite number of zeroes of $f$, and such that
\begin{equation*}
W(s) - s  \underset{|s| \to \infty}{\to} 0.
\end{equation*}
\end{lemma}

A picture gives a better idea of the content of this abstract lemma. It is elementary to observe that the zeroes of $\tilde{f}$ are asymptotically distributed along a vertical $\log$ line $\Re s = a \log |\Im s| + b$, at intervals of lengths $\sim 2\pi  (\log \lambda_1/\lambda_0)^{-1}$.
\begin{figure}[h]
\centering
\def\svgwidth{15em}
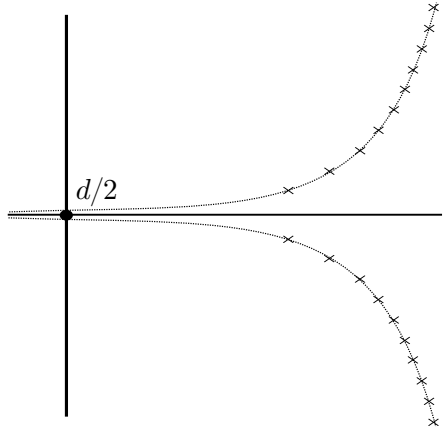
\caption{\label{fig:zeroes_log_lines} The zeroes of $\tilde{f}$.}
\end{figure}

\begin{remark} Instead of assuming $a^0_0 = 0$ and $a^1_0, a^0_1 \neq 0$, we could have assumed a finite number of explicit cancellations and non-cancellations. In that case, it is likely that one could prove a similar lemma, with $f$ having zeroes close to a finite (arbitrary) number of $\log$ lines instead of only one line. However, this leads to tedious computations that we did not carry out entirely.
\end{remark}

\begin{proof}

This is an application of Rouch\'e's Theorem. We aim to give a good bound for $| f- \tilde{f}|$ on appropriate contours. To this end, we decompose
\begin{align}
|f - \tilde{f}| &\leq \left|L_0 - \frac{a^0_1}{\lambda_1^s}\right| + \frac{1}{|s|}\left| L_1 - \frac{a^1_0}{\lambda_0^s}\right| + \frac{1}{|s|^2} \left| L_2 + \cdots + \frac{1}{s^{n-2}} L_n \right| + \mathcal{O}\left(\frac{1}{|s|^{n-1}\lambda_\#^{\Re s}} \right)  \\
\intertext{For some $\delta' > \delta$, and for $\Re s > \delta'$,this gives}
		&\leq C \left(\frac{1}{\lambda_2^{\Re s}} + \frac{1}{|s|\lambda_1^{\Re s}} + \frac{1}{|s|^2\lambda_0^{\Re s}} + \frac{1}{|s|^{n+1}\lambda_\#^{\Re s}}\right) \label{eq:bound_rouche_1}
\end{align}
where $C>0$ is a constant. We can always choose $\delta'$ big enough so that 
\begin{equation*}
\frac{C}{\lambda_2^{\Re s}} \leq \frac{1}{2} |\tilde{f}|
\end{equation*}
for $\Re s = \delta'$ and $|\Im s|$ big enough. Then, on the vertical line $\Re s = \delta'$, for $|\Im s|$ big enough, $|f-\tilde{f}| < |\tilde{f}|$.

Now, on the line $\Re s = n \log |\Im s| (\log \lambda_0 /\lambda_\#)^{-1}$, the 3 first terms of the RHS of equation \eqref{eq:bound_rouche_1} are very small in comparison to $\tilde{f}$. We can check that the last one is $\mathcal{O}(1/s)|\tilde{f}$ to see that on that line also, $|f - \tilde{f} | < |\tilde{f}|$.

Now, we observe that
\begin{equation*}
\begin{split}
|\tilde{f}| = &\left|\frac{a^0_1}{\lambda_1^s}\right| + \left|\frac{a^1_0}{s \lambda_0^s}\right| \\
		&\iff \Im s \log\frac{\lambda_0}{\lambda_1} + \arg s + \arg \frac{a_1^0}{a_0^1} \in 2\pi \Z.
\end{split}
\end{equation*}

Since in the region $\Omega =\{\delta' \leq \Re s \leq n \log |\Im s| (\log \lambda_0 /\lambda_\#)^{-1}\}$, $\arg s = \pi /2 + \mathcal{O}(\log |s| /|s|)$, we deduce that there is a constant $C>0$ such that, 
\begin{equation*}
2|\tilde{f}| \geq \left|\frac{a^0_1}{\lambda_1^s}\right| + \left|\frac{a^1_0}{s \lambda_0^s}\right|
\end{equation*}
on each line $\Im s = C+ 2\pi k (\log\frac{\lambda_1}{\lambda_0})^{-1}$, $k\in \N$, in $\Omega$. One can check that this implies that on each of those horizontal lines, $|f - \tilde{f} | < |\tilde{f}|$.

Now, the zeroes of $\tilde{f}$ are located on the line
\begin{equation*}
a^1_0 |s| \lambda_1^{\Re s} = \lambda_0^{\Re s} a^0_1.
\end{equation*}
At distance $\mathcal{O}(1)$ of that line, one can see that the RHS in \eqref{eq:bound_rouche_1} is bounded by
\begin{equation*}
\mathcal{O}(|s|^{-\alpha}) \left|\frac{a^0_1}{\lambda_1^s} \right|
\end{equation*}
for some $\alpha >0$. The proof of the lemma will be complete if we can find some circles $C_n$ around the zeroes $s_n$ of $\tilde{f}$, whose radii $r_n$ shrink, but such that on $C_n$, 
\begin{equation*}
|\tilde{f}| >> |s_n|^{-\alpha} \left|\frac{a^0_1}{\lambda_1^{s_n}} \right|.
\end{equation*}
Actually, this kind of estimate is true on the circles $C_n$ centered at $s_n$ of radius $r_n$, as long $r_n \to 0$ with $r_n >> |s_n|^{-\alpha}$.
\end{proof}

Now,
\begin{lemma}\label{lemma:recap}
There are different situations.
\begin{enumerate}
	\item When there is only $1$ cusp, we always have $L_0 \in \mathscr{D}^0(\delta, \lambda)$.
	\item In general, the set of $g\in \mathcal{G}(M)$ such that $L_0\in  \mathscr{D}^0(\delta, \lambda)$ is open and dense in $C^2$ topology.
	\item There are examples of \emph{hyperbolic} cusp surfaces with $L_0 \in  \mathscr{D}^1(\delta, \lambda)$.
	\item There are examples of \emph{hyperbolic} cusp surfaces $M$ that satisfy the following. First, $L_0 \in \mathscr{D}^0(\delta, \lambda)$. Second, there is an open set $U \subset \subset M$ such that for any cusps $Z_i$, $Z_j$, $d(U, Z_i) + d(U, Z_j) \geq \mathcal{T}^0_{ij} + \log a_i + \log a_j$. Then $\lambda_\# = \lambda$ for all the metrics $g\in \mathcal{G}_U(M)$ (the metrics with variable curvature supported in $U$).
\end{enumerate}
\end{lemma}

Lemmas \ref{lemma:zeroes_Dirichlet_function_1}, \ref{lemma:zeroes_Dirichlet_function_2} and \ref{lemma:recap} can be combined to prove theorem \ref{theorem:main}. Let us first prove lemma \ref{lemma:recap}.

\begin{proof}
When $\kappa = 1$, $\varphi = \phi_{11}$. From lemma \ref{lemma:asymptotic_expansion_scattering}, we see that $a^0_0$ is a sum of positive terms over the set of scattered geodesics whose sojourn time is $\mathcal{T}^0_{11}$, hence it cannot vanish.

In the general case, the openness property of lemma \ref{lemma:openness-coefficients} shows that for an open and dense set of $g\in \mathcal{G}(M)$ for the $C^2$ topology, the smallest element $\mathcal{T}^0$ of the set of sojourn cycles is simple. That implies that $a^0_0 \neq 0$.

For the third part of the lemma, an example will be constructed in section \ref{section:example_2_cusps}.

For the last part, an example will be given in section \ref{section:example_1_cusp}. The conclusion $\lambda_\# = \lambda_0$ is a consequence of the discussion just before theorem \ref{theorem:parametrix_varphi}
\end{proof}

\begin{proof}[Proof of theorem \ref{theorem:main}]
We can list the cases
\begin{enumerate}
	\item Consider the hyperbolic surface described in lemma \ref{lemma:recap}(4). For such a surface, for all $g\in \mathcal{G}_U(M)$, we have $L_0 \in \mathscr{D}^0(\delta, \lambda)$, and $\lambda_\# = \lambda_0$. We can apply the special case of lemma \ref{lemma:zeroes_Dirichlet_function_1}, to prove part (I).
	\item For all manifolds with one cusp only, $L_0 \in \mathscr{D}^0(\delta, \lambda)$ so we can apply the general case of lemma \ref{lemma:zeroes_Dirichlet_function_1}.
	\item When there is more than one cusp, case (2) of lemma \ref{lemma:recap} and lemma \ref{lemma:zeroes_Dirichlet_function_1} lead to part (II) of theorem \ref{theorem:main}.
	\item The example in case (3) of lemma \ref{lemma:recap} can be perturbed, preserving the condition $L_0 \in \mathscr{D}^1(\delta, \lambda)$, and with $L_1 \in \mathscr{D}^0(\delta, \lambda)$, according to lemma \ref{lemma:openness-coefficients}. We can then apply lemma \ref{lemma:zeroes_Dirichlet_function_2} to prove part (III).
	\item Finally, we can adapt the proof of lemma \ref{lemma:zeroes_Dirichlet_function_1} to show that whenever at least one $L_i$ is not the zero function, the conclusions of lemma \ref{lemma:zeroes_Dirichlet_function_1} apply, if we replace "for all constant $C>0$" by "for some constant $C>0$". This proves part (IV).
\end{enumerate}
\end{proof}

\subsection{Two examples}

In this last section, we construct explicit hyperbolic examples that satisfy the conditions given in lemma \ref{lemma:recap}.

	\subsubsection{An example with one cusp}\label{section:example_1_cusp}

Here, we construct a surface with one cusp, such that there are parts of the surface that are \emph{far} from the cusp, in the appropriate sense.

\begin{figure}[h]
\centering
\def\svgwidth{15em}
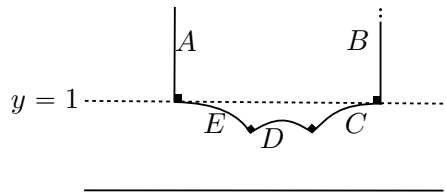
\caption{\label{fig:symmetric_pentagon} symmetric pentagon with an ideal vertex.}
\end{figure}

Topologically, we are looking at the most simple cusp surface: a punctured torus. It can be obtained explicitely by glueing two hyperbolic pentagons. Consider two copies of the pentagon in figure \ref{fig:symmetric_pentagon}. Glueing sides $A\leftrightarrow B'$, $B \leftrightarrow A'$, $D \leftrightarrow D'$, $C \leftrightarrow E$ and $C' \leftrightarrow E'$, we obtain a punctured torus, that we call $(M,g)$.

The scattered geodesic $c_0$ with the smallest sojourn time corresponds to the sides $AB'$ and $BA'$. Its sojourn time is $0$, i.e $\mathcal{T}^0=0$. However, the set $U$ of points that are strictly below the line $\{y=1\}$ is non empty (and open). This is the example in 4) in lemma \ref{lemma:recap}

	\subsubsection{An example with 2 cusps}\label{section:example_2_cusps}

Now, we aim to construct an example of surface with two cusps $(M,g)$ such that $L_0 \in \mathscr{D}^1$. We consider a two-punctured torus.

As in the previous example, we will glue pentagons. Only this time we glue 4 identical pentagons $a,b,c,d$, and they will not be symmetrical --- see figure \ref{fig:pentagon}.
\begin{figure}[h]
\centering
\def\svgwidth{25em}
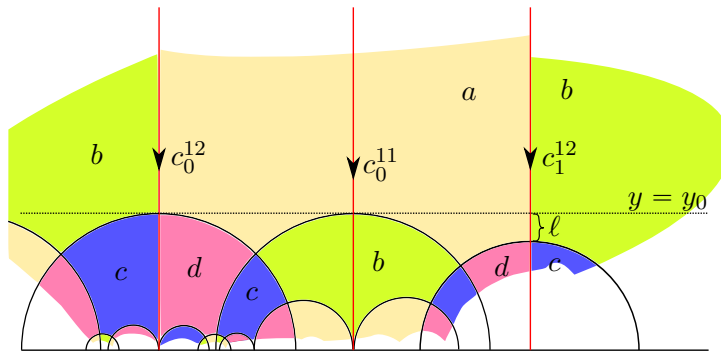
\caption{\label{fig:pentagon} A tiling of the hyperbolic plane with pentagons.}
\end{figure}

The cusp corresponding to pentagons $a$ and $b$ will be called cusp $Z_1$, and the other one, corresponding to pentagons $c$ and $d$ will be cusp $Z_2$. We obtain a surface $(M, g_\ell)$ with two cusps, depending on the hyperbolic length $\ell$. To ensure the normalization condition that the volume of a projected horosphere at height $y$ is $y^{-1}$, we have to take
\begin{equation}
y_0 = \frac{1}{2} \frac{1}{\sqrt{2} + \sqrt{1+ e^{-2\ell}}}.
\end{equation}

We number the geodesics from $i$ to $j$ ($i=1,2$, and $j=1,2$) by their sojourn time $c^{ij}_n$ with $\mathcal{T}(c^{ij}_0) \leq \mathcal{T}(c^{ij}_1) \leq \dots$. Since a geodesic coming from a cusp has to go under the $y=y_0$ line to exit a pentagon, we see that actually, $c^{11}_0$ and $c^{12}_0$ are the geodesics designated in figure \ref{fig:pentagon}. Actually, we also get that :
\begin{equation*}
-2 \log y_0 = \mathcal{T}(c^{11}_0) = \mathcal{T}(c^{12}_0) < \mathcal{T}(c^{ij}_1), \quad i,j=1,2.
\end{equation*}

This proves that $a^0_0=0$. Now, to obtain that $L_0\in \mathscr{D}^1$, we need to show that the second shortest sojourn time is $2\ell - 2 \log y_0$, and that it is \emph{simple}. That is to say, $c^{12}_1$ really \emph{is} the curve drawn in figure \ref{fig:pentagon}, and the only other curves with sojourn time $\leq 2\ell - 2 \log y_0$ are $c^{11}_0$ and $c^{12}_0$.

In order to prove this, draw a line at height $y_0 e^{-2\ell}$. A scattered geodesic coming from cusp $Z_1$ can be lifted to $\Hh^2$ as a curve coming from $\infty$ in the pentagon $a$, that stays in the same pentagon as $ y \leq y_0 e^{-2\ell}$. When $\ell$ is small enough, there are only 3 geodesics that satisfy such a property, and they are drawn on figure \ref{fig:pentagon}.

\appendix

\section{Regularity of Horospheres for some Hadamard manifolds}\label{appendix:horocycles}
In this appendix, we recall some results on the regularity of stable and unstable foliations.

\begin{lemma}\label{lemma:Regularity_horocycles}
Let $ N $ be a simply connected manifold of dimension $d+1$, with sectio nal curvature $-|K_{max}| < K < - |K_{min}| < 0$. Assume additionally that all the covariant derivatives of the curvature tensor $R$ of $ N $ are bounded. For a point $\xi\in S^\ast  N $, we define $W^s(\xi)$ as $\{ \xi' \in S^\ast M,\ d(\pi \varphi_t \xi', \pi \varphi_t \xi) \to 0 \text{ as } t \to +\infty \}$, and similarly $W^u(\xi)$. Those are $\mathscr{C}^\infty$ submanifolds of $T S^\ast  N $, \emph{uniformly} in $\xi$; they form a continuous foliation of $T S^\ast  N $, tangent respectively to $E^s$ and $E^u$.
\end{lemma}

\begin{proof}
We just check that the proof of the compact case also works for us.

Let $\xi \in S^\ast  N $. Take $\nu >0$, $t>0$ and $0< \epsilon < t$. For $k\geq 0$, let $\xi_k = \varphi_{kt}(\xi)$. Using the exponential charts for the Sasaki metric on $S^\ast  N $, we can conjugate $\varphi_t$ to diffeomorphisms from $T_{\xi_k}S^\ast  N $ to $T_{\xi_{k+1}}S^\ast  N $ that map $0$ to $0$. We still refer to those as $\varphi_t$. Let
\begin{equation*}
\mathscr{H}_{\nu} := \{ (z_k)\ |\ z_k \in T_{\xi_k} S^\ast  N ,\ \limsup \| z_k \| e^{\nu k} < \infty \}.
\end{equation*}
This is a Banach space when endowed with 
\begin{equation*}
\| (z_k) \|_\nu := \sup \| z_k\| e^{\nu k}.
\end{equation*}
On $\mathscr{H}_\nu$, we can define
\begin{equation*}
\Psi(z_k) := (0, \varphi_t(z_0) - z_1, \varphi_t(z_1) - z_2, \dots).
\end{equation*}
This is a $\mathscr{C}^\infty$ function on $\mathscr{H}_\nu$. We want to solve $\Psi = 0$ in $\mathscr{H}_\nu$. As the stable manifold $\xi$ should be a graph over $E^s(\xi)$, for $(z_k)\in \mathscr{H}_\nu$, we decompose $(z_k) = (z_0^s, r)$. We need to show that $\partial_r \Psi$ is injective and surjective on a closed subspace, to use the implicit function theorem for Banach spaces. Let $V = (v_0^{u0}, v_1, \dots)\in \mathscr{H}_\nu$ where $\mathbf{u0}$ refers to the weak unstable direction $E^u \oplus \R \mathbf{X}$. We have
\begin{equation*}
\partial_r \Psi(0) V = (0, d_\xi \varphi_t \cdot v_0^{u0} - v_1, d_{\xi_1} \varphi_t \cdot v_1 - v_2, \dots).
\end{equation*}

First, we prove this is injective. Assume $\partial_r \Psi(0) V =0$. Then, we have $d_\xi \varphi_t \cdot v_0^{u0} = v_1$. Since the weak unstable direction is stable by the flow, $v_1 \in E^{u0}(\xi_1)$. By induction, $v_k \in E^{u0}(\xi_k)$. However, $V$ has to be in $\mathscr{H}_\nu$, so that there is a constant $C>0$ such that for all $k>0$,
\begin{equation*}
\| v_0^{u0}\| = \|(d_\xi \varphi_{kt})^{-1} v_k \| \leq C e^{-\nu k}.
\end{equation*}
This implies that $v_0^{u0}=0$, and $V=0$.

Now, we prove that $\partial_r\Psi$ is surjective on the space of sequences whose first term vanishes. Let $W=(w_0^{u0}, w_1, \dots) \in \mathscr{H}_\nu$. We decompose each $w_k = (w_k^s, w_k^{u0})$, and we try to solve $\partial_r \Psi. V = W$, with $v_i = (v_i^s, v_i^{u0})$. If $V$ is a solution, then, for all $k>0$,
\begin{equation*}
v_k = d\varphi_{kt} v_0^{u0} - \sum_{l= 0}^{k-1} d\varphi_{lt} w_{k-l}.
\end{equation*}
That is why we let
\begin{equation*}
v_0^{u0} = \sum_{l=1}^\infty (d\varphi_{lt})^{-1} w_{l}^{u0}.
\end{equation*}
This sum converges because $\| (d \varphi_{kt})^{-1}|_{E^{u0}} \|$ is bounded independently from $k$, and we assumed $w$ is in $\mathscr{H}_\nu$. Then, we have to check that the equations
\begin{equation*}
v_k = -\sum_{l=1}^{k} d\varphi_{(k-l)t} w^s_{l} + \sum_{l=1}^\infty (d\varphi_{lt})^{-1} w_{k+l}^{u0}
\end{equation*}
define a sequence in $\mathscr{H}_\nu$. By the Anosov property, there are constants $\lambda>0$ and $C>0$ not depending on $w$, nor on $\xi$ such that $\| d\varphi_{kt} w^s_l \| \leq e^{-\lambda kt} \| w^s_l\|$. So 
\begin{equation*}
\| v_k \| \leq \sum_{l=1}^k C e^{-\lambda (k-l)t} e^{-l \nu} \|w\|_{\nu} + \sum_{l=1}^\infty e^{-(k+l)\nu} \| w\|_\nu.
\end{equation*}
It suffices to choose $\nu < \lambda t$, and we find that $v \in \mathscr{H}_\nu$.

By the Implicit Function Theorem, in a small enough neighbourhood of $\xi$, the strong stable manifold of $\xi$ is a graph of a $\mathscr{C}^\infty$ function from $E^s(\xi)\to E^{u0}(\xi)$. Additionally, the derivatives of this function are controlled by the $\mathscr{C}^k$ norms of $\Psi$. These norms are the $\mathscr{C}^k$ norms of $\varphi_t$. They can be bounded independently from $\xi$ --- recall $\Psi$ depends on $\xi$ --- according to Lemma B.1 and Proposition C.1 of \cite{Bonthonneau-2}. Hence, the stable manifolds are uniformly smooth in the manifold.

Moreover, the tangent space of $W^s(\xi)$ at $\xi$ has to be $E^s(\xi)$, according to the dynamical definition of $W^s(\xi)$. We deduce that the regularity of the lamination formed by the collection of $W^s(\xi)$, $\xi \in S^\ast M$ has the regularity of the \emph{splitting} $E^s \oplus E^u$. Using the description as Green's fiber bundles for $E^s$ and $E^u$, one can prove that they are H\"older, and the lamination is actually a foliation.

The case of unstable manifolds is similar.
\end{proof}

\section{Estimating the regularity of solutions for transport equations}

\begin{lemma}\label{lemma:Regularity_Jacobian}
Let $N$ be a riemannian manifold such that all the derivatives of its curvature are bounded. Let $G$ be a $C^{\infty}$ function on $N$, such that $\| \nabla G\| = 1$, and $\| \nabla G \|_{\mathscr{C}^n(N)}$ is bounded for all $n$. Let $\varphi^G_t$ be the flow generated by $V= \nabla G$. Assume that $\varphi^G$ is expanding, that is, there is a $\lambda>0$ such that if $\mathbf{u} \perp V$, $\| d \varphi^G_t . u \| \geq C e^{\lambda t}\| u \|$ for $t>0$. 

Let $g_0$ be a $C^\infty$ function on $N$, supported in $G \geq \ell$. Let
\begin{equation*}
g_1(x) = \int_{-\infty}^0 g_0 \circ \varphi^G_t dt.
\end{equation*}
Then if $\mathscr{L}(\tau)= \sup\{ |g_0(x)|, G(x)=\tau \}$, for all $n$ there is a constant $C_n>0$ only depending on $G$ such that
\begin{equation*}
|g_1(x)| \leq \int_{\ell}^{G(x)} \mathscr{L}(\tau) d\tau, \quad \| \nabla g_1 \|_{\mathscr{C}^{n-1}(G\leq t)} \leq C_n \|  g_0 \|_{\mathscr{C}^{n}(G \leq t)}
\end{equation*}
\end{lemma}

\begin{proof}
The first part of the statement is obvious. We concentrate on the second part. The basic idea is that when differentiating in the direction of the flow, one obtains $g_0$, and when differentiating in other directions, one can use the contracting properties of $\varphi^G_t$ in negative time. Let $x\in N$, and $X_1, \dots, X_n$ vectors at $x$. We want to evaluate $\nabla_{X_1, \dots, X_n} g_1(x)$. We can decompose the $X_i$'s according to
\begin{equation*}
T_x N = \R V \oplus V^\perp.
\end{equation*}
By linearity, we can assume that either $X_i \propto V$ or $X_i \perp V$. Additionally, we assume $\| X_i \| = 1$. By taking symmetric parts, and antisymmetric parts of $\nabla$, we see that it suffices to evaluate $\nabla_{X_1, \dots, X_n}g_1$ when the $X_i$'s colinear to $V$ are the last in the list. That corresponds to differentiating $g_1$ first along $V$. Now, there are two cases. First, assume that one of the $X_i$'s is colinear to $V$. Then
\begin{equation*}
\nabla_{X_1, \dots, V} g_1 = \nabla_{X_1, \dots, X_{n-1}} g_0
\end{equation*}

We are left to consider the case when all the $X_i$'s are orthogonal to $V$. For this, we use the proof from \cite[appendix B]{Bonthonneau-2}. From therein, we know that for $t>0$,
\begin{equation*}
\nabla_{X_1, \dots, X_n} (g_0 \circ \varphi^G_{-t})  = W_t^n g_0((\varphi^G_t)^\ast X_1, \dots, (\varphi^G_t)^\ast X_n)
\end{equation*}
Where --- lemma B.2 --- $W_t^n g_0$ is a sum of tensors of the form
\begin{equation*}
\nabla^k_{T_1(0,t), \dots, T_k(0,t)} g_0.
\end{equation*}
The $T_i(s,t)$'s are tensors with a particular structure. Either they are of order $1$ and $T_i(s,t)(X) = X$, or they are of higher order and
\begin{equation*}
T_i(s,t)= \int_s^t (\varphi^G_u)^\ast R_i \left[(\varphi^G_u)_\ast T_{i,1}(u,t), \dots, (\varphi^G_u)_\ast T_{i,k_i}(u,t)\right]du
\end{equation*}
where $R_i$ is a bounded tensor with all derivatives bounded, and the $T_{i,j}$ have the same structure. Observe that for $X \in TN$, $\| (\varphi^G_t)^\ast X \| \leq C \| X \|$ when $t>0$, and $\| (\varphi^G_t)^\ast X \| \leq C e^{-\lambda t} \| X \|$ when $X \perp V$. By induction, we deduce that for $t>0$, 
\begin{equation*}
\| \nabla_{X_1, \dots, X_n} (g_0 \circ \varphi^G_{-t})\|(x) \leq C_n e^{-\lambda n t} \| g_0 \|_{\mathscr{C}^n(G \leq G(x))}.
\end{equation*}
We just have to integrate this for $t\in [0, +\infty[$, and the exponential decay ensures the convergence.
\end{proof}

\bibliographystyle{alpha}
\def\cprime{$'$} \def\dbar{\leavevmode\hbox to 0pt{\hskip.2ex \accent"16\hss}d}
  \def\cprime{$'$}

\end{document}